\documentclass[fleqn,a4paper,10pt]{amsart}

\usepackage{graphicx}
\usepackage{tikz}
\usetikzlibrary{shapes,decorations.pathmorphing}
\usepackage{amsmath}
\usepackage{amssymb,amsthm}
\usepackage{comment}
\usepackage{hyperref}

\newtheorem{lem}{Lemma}

\theoremstyle{definition}

\def\tb{t_{\bullet}}
\def\tw{t_{\circ}}

\title{The two-point function of bicolored planar maps}
\author{\'Eric Fusy}
\address{LIX, \'Ecole Polytechnique, 91120 Palaiseau, France. Academic year 2014-2015: PIMS-CNRS, University of British Columbia, Vancouver, BC, Canada.}
\email{fusy@lix.polytechnique.fr}

\author{Emmanuel Guitter}
\address{Institut de Physique Th\'eorique, CEA, IPhT, 91191 Gif-sur-Yvette, France, CNRS, URA 2306}
\email{emmanuel.guitter@cea.fr}

\begin{document}
\maketitle

\begin{abstract}
We compute the distance-dependent two-point function of vertex-bicolored planar maps, i.e.,\ maps whose vertices are colored in black and white
so that no adjacent vertices have the same color. By distance-dependent two-point function, we mean the generating function of these maps 
with both a marked oriented edge and a marked vertex which are at a prescribed distance from each other. As customary, the
maps are enumerated with arbitrary degree-dependent face weights, but the novelty here is that we also introduce color-dependent 
vertex weights. Explicit expressions are given for vertex-bicolored maps with bounded face
degrees in the form of ratios of determinants of fixed size. Our approach is based on a slice decomposition of maps which relates 
the distance-dependent two-point function to the coefficients of the continued fraction expansions of some distance-independent
map generating functions. Special attention is paid to the case of vertex-bicolored quadrangulations and hexangulations, whose two-point functions
are also obtained in a more direct way involving equivalences with hard dimer statistics. A few consequences of our results,
as well as some extension to vertex-tricolored maps, are also discussed.
\end{abstract}

\section{Introduction}
\label{sec:introduction}
Distance properties within planar maps have raised a lot of interest in the recent years and led to many remarkable results on the
statistics of distance correlations within families of random maps. Still many questions are not yet solved, and many improvements
of the present results, although quite natural, remain challenging. One of the simplest characterization of the distance statistics within maps
is probably the \emph{distance-dependent two-point function} which, roughly speaking, enumerates maps with two ``points" (typically edges or vertices) at a 
\emph{ fixed given graph distance} within the map. Such two-point functions were first computed in \cite{GEOD} for general 
families of bipartite planar maps with controlled face degrees (including the simplest case of quadrangulations). Although this is not quite 
the method used in \cite{GEOD},
it has now become clear that the simplest way to get two-point functions is via a distance-preserving bijection between maps
and tree-like objects called mobiles, originally
found by Schaeffer \cite{SchPhD,CMS09}  (rephrasing a bijection due by Cori and Vauquelin \cite{CoriVa}) in the case of quadrangulations
and later generalized to the case of arbitrary  maps \cite{BDG04}.
More recently, a similar bijection extending Schaeffer's ideas, due to Ambj\o rn and Budd, has made it possible to compute the two-point function 
of maps with arbitrary large face degrees, controlled by both their number of edges and faces \cite{AmBudd}, and more generally  
that of bipartite maps or hypermaps \cite{BFG} with arbitrarily large face degrees. 

In this quest for two-point functions, a conceptual progress was made in \cite{BG12} where it was realized that, due to the coding of general maps by mobiles, the distance-dependent two-point function of some given ensemble of maps was somehow hidden in the coefficients of the continued fraction expansions of some distance-independent generating functions for the same maps. This discovery made it possible to use the whole machinery of continued fractions to obtain very general expressions
for two-point functions of maps with bounded face degrees in the form of ratios of symplectic Schur functions, themselves expressible
in terms of determinants of a fixed size (typically given by the maximal face degree). 

Another particularly elegant explanation for this connection, which avoids the recourse to mobiles, is via the so-called \emph{slice decomposition}. 
Starting with maps having a marked face of controlled degree and a marked vertex, the slice decomposition consists in
cutting the maps along geodesic (i.e.,\ shortest) paths from the marked face to the marked vertex, creating pieces of maps 
called slices. The mathematical translation of this decomposition is that slice generating functions are nothing but the continued fraction expansion
coefficients of the generating function of the original maps \cite{BG12}. The two-point function is then easily recovered
from the slice generating functions.

The purpose of this paper is to apply the technique of \cite{BG12}, in its slice decomposition formulation, to compute the distance-dependent
two-point function of \emph{vertex-bicolored planar maps}, i.e.,\ maps whose vertices are colored in black and white
so that no adjacent vertices have the same color (note that these maps are necessarily bipartite). Beside controlling 
the face degrees by assigning degree-dependent face weights, the novelty of this paper is that we also incorporate two
different vertex weights: a weight $t_\bullet$ for black vertices and a weight $t_\circ$ for white vertices.
Our results therefore generalize those of \cite{BG12} for bipartite maps by keeping a control on the vertex colors.

The paper is organized as follows: in Section~\ref{sec:slicedecomp}, we first recall the mechanism of the slice decomposition
by applying it to vertex-bicolored planar maps with a marked face of fixed degree $2n$ and a marked vertex (Section~\ref{sec:sclidecomposition}).
For short, let us call these maps ``pointed maps with a boundary of length $2n$".
We then recall in Section~\ref{sec:continuedfractions} how one recovers from this decomposition the slice generating functions as
the coefficients of the continued fraction of the generating function for pointed maps with boundaries of arbitrary (but controlled) lengths.
This yields expressions for the slice generating functions in terms of \emph{Hankel determinants}, which are determinants of matrices whose elements
are themselves generating functions
of pointed maps with boundaries of fixed increasing lengths.  Section~\ref{sec:twopoinfunction} shows the connection between slice generating functions and two-point functions while Section~\ref{sec:recusiveislice} establishes non-linear systems of equations which implicitly determine
the slice generating functions \footnote{in practice, we however do not know how to solve directly these non-linear systems (apart from the simple
case of quadrangulations) and this is why we recourse to the continued fraction approach.}. Section~\ref{sec:expressionsFn} is
devoted to obtaining a tractable expression for the generating function of pointed maps with a boundary of length $2n$,
as required to compute our Hankel determinants. This expression is based on so-called \emph{conserved quantities}
introduced in Section~\ref{sec:conservedquantities}, and made explicit in Section~\ref{sec:expressionsFnexpli}. The explicit
computation of the Hankel determinants is presented in Section~\ref{sec:computationHankel}. They come in two families, a simpler
one, computed in Section~\ref{sec:computationHone} and a more involved one, containing most of the spicing due to bicoloring and computed in \ref{sec:computationHzero}.
Our final results are gathered in Section~\ref{sec:finalresult} where we also give, as a simple application, the first terms in the expansion in $t_\bullet$
and $t_\circ$ of the two-point functions of quadrangulations and hexangulations. Section~\ref{sec:harddimers} presents a completely different
approach to compute the Hankel determinants via an  equivalence with hard dimer statistics on segments. Indeed, the Hankel determinants may
be shown to enumerate sets of paths on appropriate graphs, which are so constrained that their configurations reduce to those of dimers
on one or a few linear segments. This approach was used in \cite{BG12} in the case of (uncolored) quadrangulations and we generalize
it to vertex-bicolored quadrangulations in Section~\ref{sec:harddimersquad} at the price of introducing dimers with parity-dependent weights. 
We also extend the method to the more involved case of vertex-bicolored hexangulations in \ref{sec:harddimershex}. We conclude in 
Section~\ref{sec:conclusion} where we discuss a few extensions of our results. First, we use our solution for vertex-bicolored quadrangulations 
to derive in Section~\ref{sec:otherintregquad} the solution of new sets of integrable equations, in connection with irreducible quadrangulations.
We then present in Section~\ref{sec:oherintegthreecolor} the solution of a particular vertex-tricolored problem, in connection with Eulerian triangulations. 
A few side results or technical derivations are presented in Appendices A, B and C.

\section{Slice decomposition and continued fractions}
\label{sec:slicedecomp}
\subsection{The slice decomposition}
\label{sec:sclidecomposition}
A \emph{vertex-bicolored planar map} denotes a connected graph embedded on the sphere whose 
vertices are colored, say in black and white, so that no two adjacent vertices have the same color.
Note that a vertex-bicolored map is necessarily bipartite, i.e.,\ with all its faces of even degree.
Conversely, a bipartite map has two vertex bicolorations, which are obtained from one another by a color switch. 
The map is said to be \emph{rooted} if it has a marked oriented edge (the root edge), and more precisely 
\emph{black-rooted}, (resp.\ white-rooted) if the origin of this edge is a black (resp.\ white) vertex,
hereafter called the root vertex. 
The face to the right of the root edge is called the root face and we shall call the \emph{boundary} the set of
vertices and edges incident to the root face, supposedly oriented clockwise around the root face
(i.e.,\ counterclockwise around the rest of the map). 
Finally, the map is said to be \emph{pointed} if it has a marked vertex (the pointed vertex).

Our goal here is to compute a number of generating functions for these vertex-bicolored maps
with a control on the degrees of their faces by assigning a weight $g_k$ to each face 
of degree $2k$, but also with \emph{a control on the number of black and white vertices} independently 
by giving them different weights, say a weight $t_\bullet$ to each black vertex and a weight $t_\circ$ to each white vertex.

The main result of this paper is an expression, depending implicitly on $t_\bullet$, $t_\circ$ and all the $g_k$'s, 
for the \emph{distance dependent two-point function} of vertex-bicolored 
planar maps, which is the generating function of, say pointed black-rooted
\footnote{Clearly, the generating function of pointed white-rooted bicolored planar maps with a root edge
whose white (resp.\ black) extremity is at graph distance $i$ (resp.\ $i-1$) from the pointed vertex is then obtained
by simply exchanging $t_\bullet$ and $t_\circ$.}
planar maps with a root edge
whose black (resp.\ white) extremity is at graph distance $i$ (resp.\ $i-1$) from the pointed vertex for any given $i\geq 1$.

To get this expression, we shall essentially follow the same procedure as in \cite{BG12}, which consists in
relating the two-point function to the coefficients of the continued fraction expansion of some simpler generating
function (the so-called resolvent in the matrix integral language) for maps with a control on the degree of
their root face.
\begin{figure}
\begin{center}
\includegraphics[width=5cm]{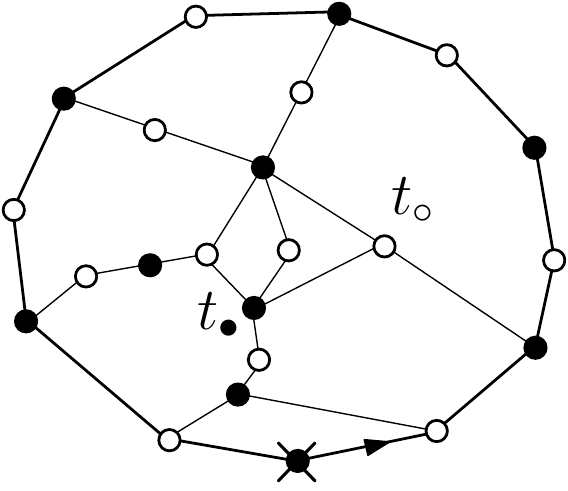}
\end{center}
\caption{An example of black-rooted bicolored planar map with root face of degree $12$. Each black vertex receives
a weight $t_\bullet$ and each white vertex a weight $t_\circ$, except for the root vertex (origin of the root edge
marked by an arrow) which does not receive the weight $t_\bullet$, as indicated here conventionally by
putting a cross on that vertex.}
\label{fig:rootedmap}
\end{figure}

More precisely, let us define $F_n^{\bullet}\equiv F_n^{\bullet}(\{g_k\}_{k\geq 1},t_\bullet,t_\circ)$ as the generating function for
black-rooted bicolored planar maps with a root face of degree $2n$. By convention, we decide to assign a weight $1$
to the root face (instead of $g_n$) as well as a weight $1$ to the root vertex (instead of $t_\bullet$),
see Figure \ref{fig:rootedmap} for an illustration.
We shall also consider the case of maps which are both black-rooted and pointed in such a way that the root vertex 
is at distance $d_\bullet \leq d$ from the pointed vertex and no boundary vertex
is at distance strictly less than $d_\bullet$ from the pointed vertex (in other words, among all boundary vertices
the root vertex is one of the closest ones to the pointed vertex).
We shall denote by $F_n^{\bullet}(d)$ the corresponding generating function, with now the convention that the pointed vertex receives a weight $1$
(instead of $t_\bullet$ or $t_\circ$ depending on its color) while the root vertex receives its normal weight $t_\bullet$ (unless it is itself the pointed
vertex, i.e.,\ when $d_\bullet=0$). Note that, in particular 
\begin{equation*}
F_n^{\bullet}=F_n^{\bullet}(0)\ . 
\end{equation*}
We may finally define similarly generating functions $F_n^{\circ}$ and $F_n^{\circ}(d)$, now for bicolored planar maps which 
are white-rooted instead of black-rooted. Clearly, by a symmetry which consists in simply flipping the map and reversing the orientation
of the root-edge, keeping all colors unchanged, we have:
\begin{equation}
t_\bullet F_n^{\bullet}(\{g_k\}_{k\geq 1},t_\bullet,t_\circ) = t_\circ F_n^{\circ}(\{g_k\}_{k\geq 1},t_\bullet,t_\circ)\ .
\label{eq:FbulletFcirc}
\end{equation}
\begin{figure}
\begin{center}
\includegraphics[width=14cm]{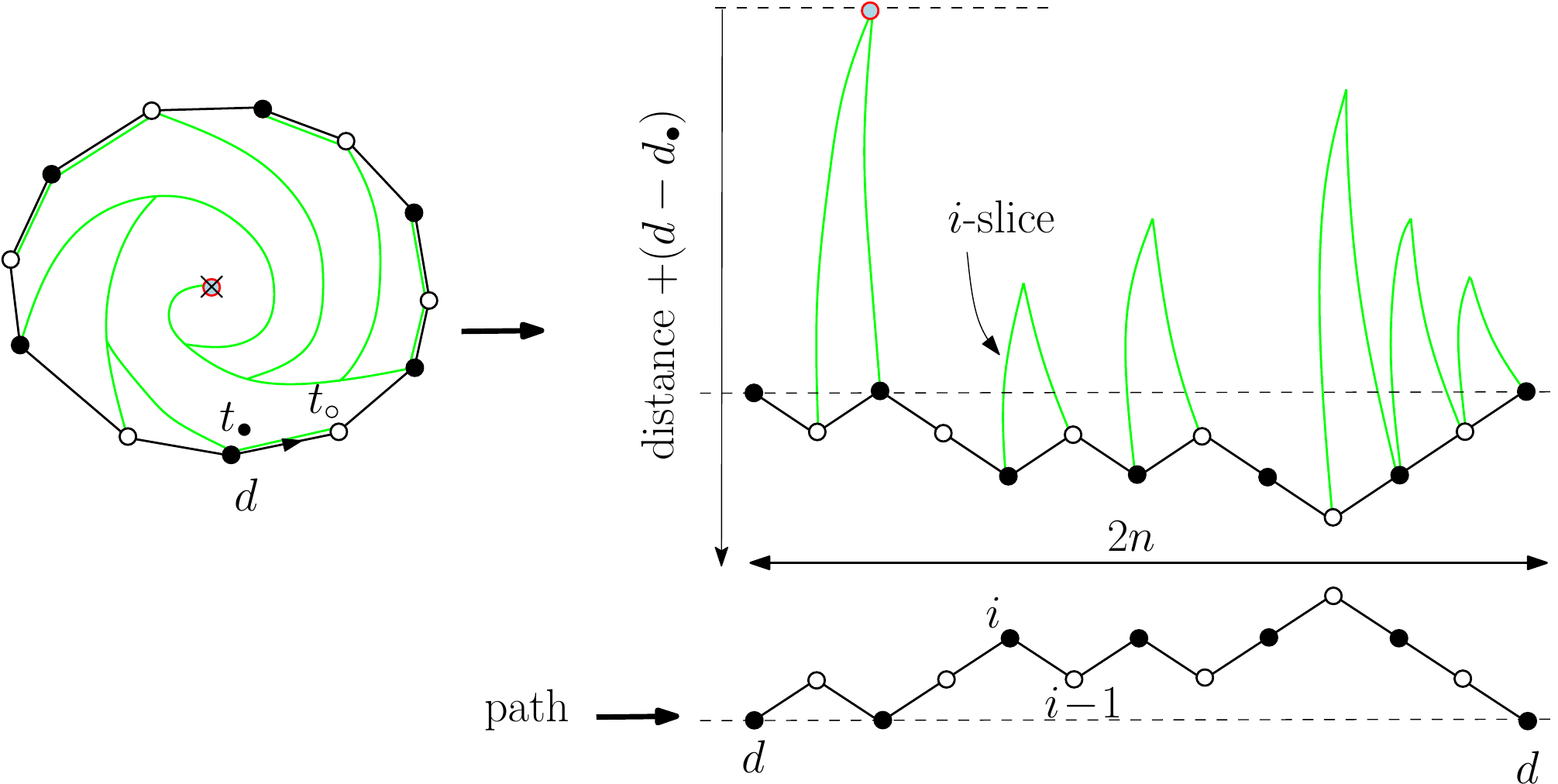}
\end{center}
\caption{Schematic picture of the slice decomposition of a map contributing to $F_n^\bullet(d)$. The pointed vertex 
is represented in red filled with light blue and receives no vertex weight (as indicated by the cross). When each boundary vertex
is labelled by its distance to the pointed vertex plus $d-d_\bullet$ ($d_\bullet$ being the distance from the root vertex to the pointed vertex), 
the sequence of labels forms a path from height $d$ to height $d$ staying above height $d$ (represented here at the bottom of the
figure). Cutting along leftmost geodesic paths from the boundary vertices to the pointed vertex results 
in a decomposition into slices (upper right part of the figure) with an $i$-slice corresponding to each descending step 
of the path from height $i$ to height $i-1$ (note that, since distances are measured from the top vertex, the boundary after
cutting does not give the desired path but its horizontally reversed image).}
\label{fig:rootedpointeddecomp}
\end{figure}
\begin{figure}
\begin{center}
\includegraphics[width=14cm]{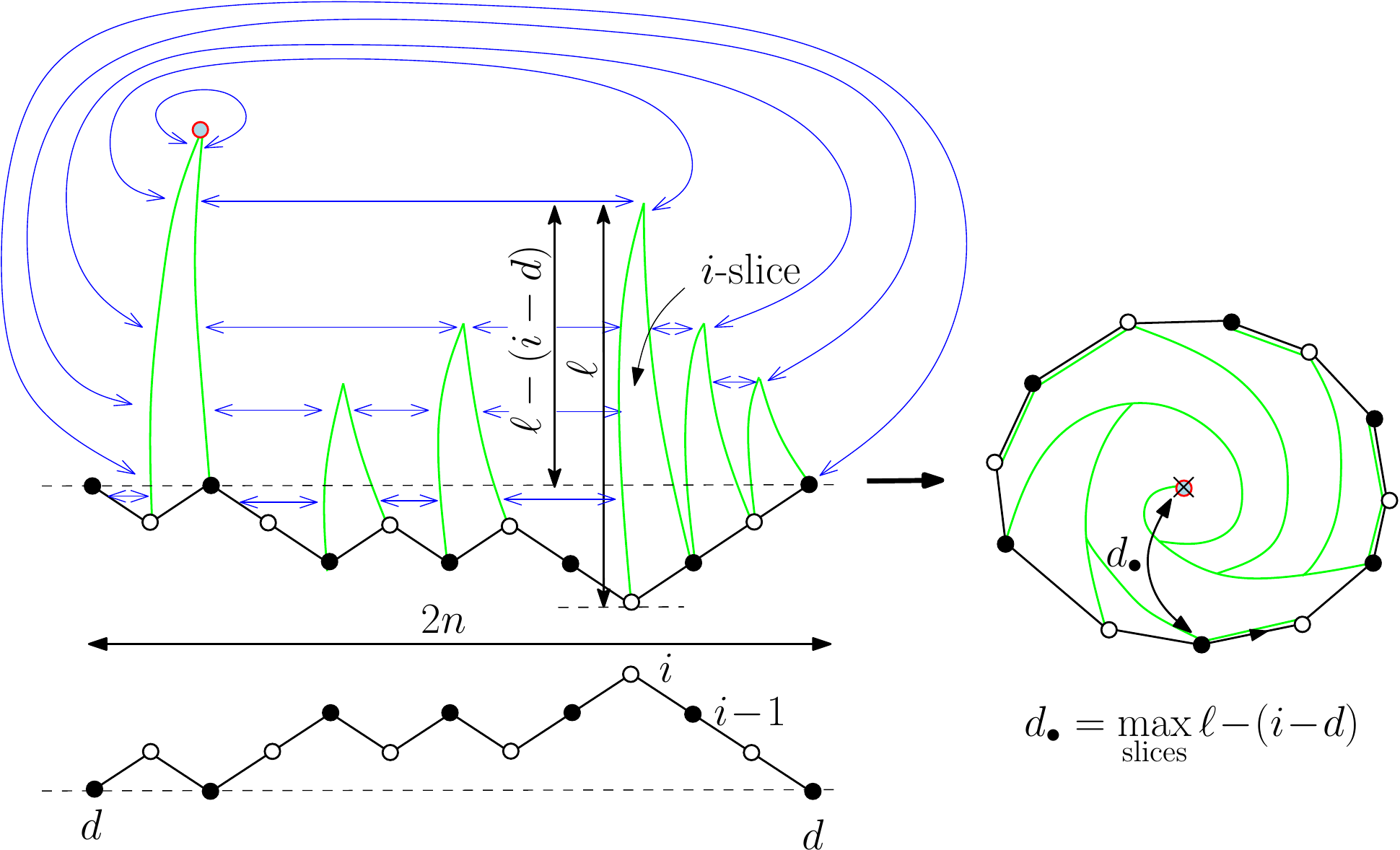}
\end{center}
\caption{Schematic picture of the reverse construction of the slice decomposition of Figure \ref{fig:rootedpointeddecomp}.
The map is recovered by gluing the left and right boundaries of the successive slices as indicated by blue arrows.
The actual distance $d_\bullet$ from the root vertex to the pointed vertex is the maximum over all slices
of $\ell-(i-d)$.}
\label{fig:sliceconcatenation}
\end{figure}
\begin{figure}
\begin{center}
\includegraphics[width=9cm]{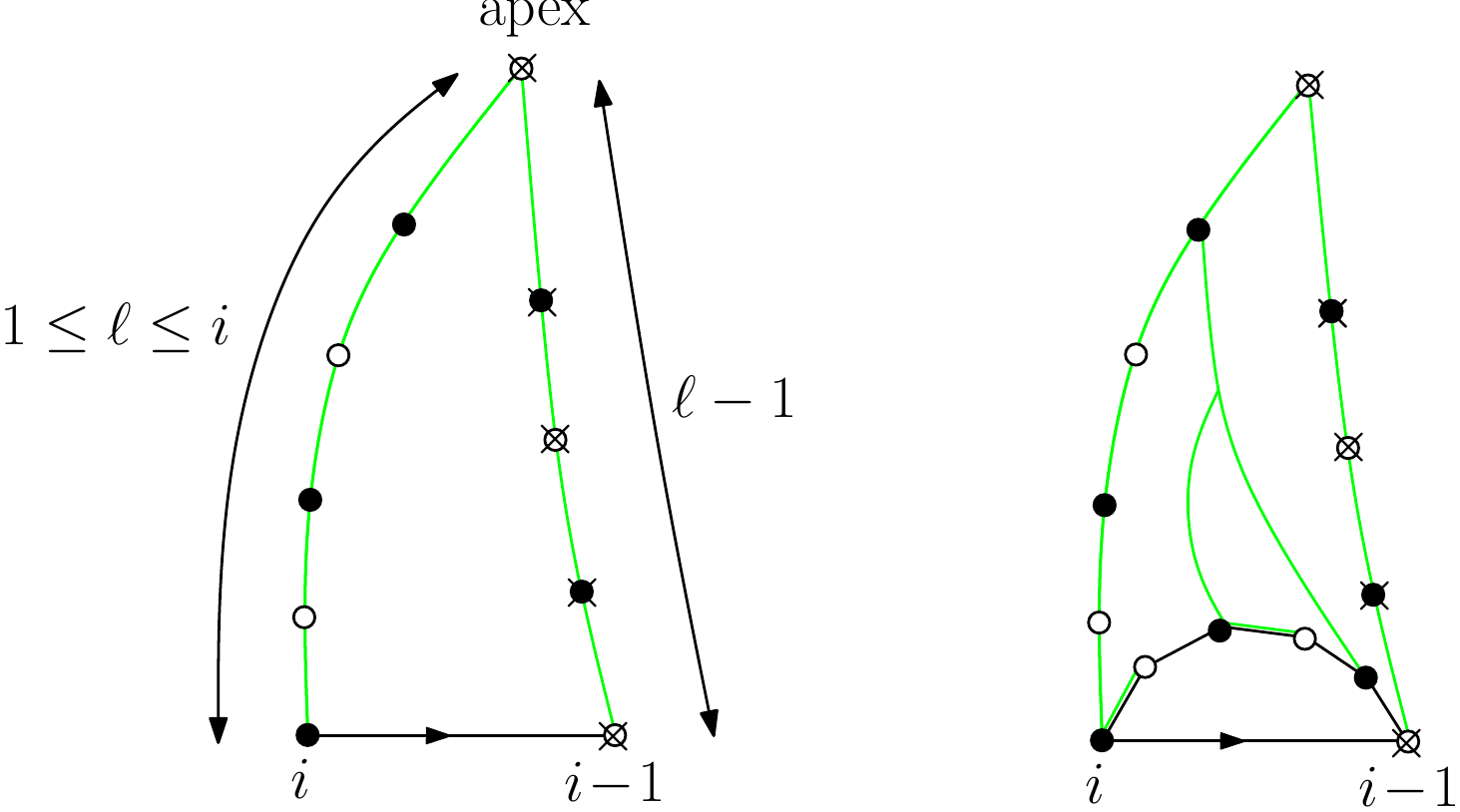}
\end{center}
\caption{Left: schematic picture of a black $i$-slice. Its left boundary is a geodesic path from the vertex labelled $i$ to the apex,
of length $\ell$ for some $1\leq \ell \leq i$. The right boundary is the \emph{unique} geodesic path from the 
vertex labelled $i-1$ to the apex and has length $\ell-1$. All vertices receive vertex weights, except those of the right boundary 
(which includes the apex by convention). Right: schematic picture of the slice decomposition of this $i$-slice (see text).}
\label{fig:islice}
\end{figure}
As discussed in \cite{BG12}, an expression for $F_n^{\bullet}(d)$ can
be obtained via a so-called \emph{slice decomposition} as follows
\footnote{As in \cite{BG12}, most of the study can alternatively be done using mobiles.}: 
let us view the maps enumerated by  $F_n^{\bullet}(d)$ as drawn 
in the plane with the root face as external face and let us label each vertex on the boundary by its distance to the pointed vertex plus 
$(d-d_\bullet)$ ($d_\bullet$ being the distance of the root vertex to the pointed vertex. In particular the root vertex receives the label $d$). 
The sequence of these labels when going counterclockwise around the map from the root vertex may be viewed as heights 
of a path of length $2n$ made of $+1$ or $-1$ steps (each associated to an edge side incident to the root face), starting and 
ending at height $d$, and remaining above height $d$. The path is naturally colored alternatively in black and white (according to the color of the underlying boundary vertex). We may then draw from each boundary vertex its leftmost geodesic (i.e.,\ shortest) path
to the pointed vertex, see Figure \ref{fig:rootedpointeddecomp}. The set of these geodesic paths decomposes the map into a number of \emph{slices}, where a slice is
associated to each $-1$ step of the associated path. More precisely, each step $i\to i-1$ gives rise to an $i$-slice,
which is a rooted map with the following properties: its boundary is made of three parts, see Figure \ref{fig:islice}: (i) its base consisting
of a single root edge oriented from a vertex labelled $i$ to a vertex labelled $i-1$, (ii) a left boundary
of length $\ell$ with $1\leq \ell \leq i$ connecting the vertex labelled $i$ to another vertex, the \emph{apex} and
which is a geodesic path within the slice, and (iii) a right boundary of length $\ell-1$ connecting 
the vertex labelled $i-1$ to the apex, and which is the unique geodesic path within the slice between these two vertices. 
By convention, we decide that the apex belongs to the right boundary but not to the left one.
The left and right boundaries are then required to have no common vertex (i.e., they do not meet before reaching the apex). 
The fact that $\ell \leq i$ is simply due to 
the fact that $\ell$ is smaller than the distance $i-(d-d_\bullet)$ from the base vertex of the slice labelled $i$ to the pointed vertex in the map. 
Note that the actual distance $d_\bullet$ from the black-root vertex to the pointed vertex is the maximum of $\ell-(i-d)$ over
all slices. Note also that, when $\ell=1$, the left boundary may stick to the base, in which case the $i$-slice is reduced to a single edge $i\to i-1$.
This degenerate situation occurs whenever the leftmost geodesic path from the boundary vertex labelled $i$
in the original map passes through the next boundary edge counterclockwise around the map (this edge leading to the next boundary vertex counterclockwise around the map, labelled $i-1$).
We shall distinguish black $i$-slices, whose root vertex is black from white $i$-slices, whose root vertex is white.
The fact that no slice is associated to any $+1$ step in the associated path is because in this case, the leftmost
geodesic path from the endpoint (labelled say $i+1$) of the associated $i\to i+1$ boundary edge passes via the origin of the same edge
(labelled $i$), hence sticks to the boundary without creating a slice, see Figure \ref{fig:rootedpointeddecomp}. 
The reverse of the slice decomposition is a simple slice concatenation, as shown in Figure \ref{fig:sliceconcatenation}.

\subsection{Continued fractions}
\label{sec:continuedfractions}

In view of the above slice decomposition, it is natural to introduce the generating function  
$Z_{d,d}^{\bullet\bullet +}(2n)\equiv  Z_{d,d}^{\bullet\bullet +}(2n,\{B_i\}_{i\geq 1},\{W_i\}_{i\geq 1})$ of paths 
made of $+1$ or $-1$ steps, colored alternatively in black and white, starting and ending at black height $d$ and 
remaining above height $d$ (with $d\geq 0$), and where each descending step from a black height $i$ to a white height $i-1$ receives a weight
$B_i$ and each descending step from a white height $i$ to a black height $i-1$ receives a weight
$W_i$ (and with no weights assigned to ascending steps).  More generally, we may define 
$Z_{d,d'}^{\bullet\bullet +}(2n)$ ($d,d'\geq 0$) as enumerating paths with the same weights, now going from a black height $d$ 
to a black height $d'$ (with $d'=d \mod 2$) and remaining above $\min(d,d')$. By obvious generalizations, we shall also consider the quantities 
$Z_{d,d'}^{\circ\circ +}(2n)$ (with $d'=d \mod 2$)
as well as
$Z_{d,d'}^{\bullet\circ +}(2n+1)$ and $Z_{d,d'}^{\circ\bullet +}(2n+1)$
(with $d'=d+1 \mod 2$) according to the color of the extremities of the path.

Following \cite{BG12}, the slice decomposition directly gives rise to the following expressions
\begin{equation}
F_n^{\bullet}(d)=Z_{d,d}^{\bullet\bullet +}(2n,\{B_i\}_{i\geq 1},\{W_i\}_{i\geq 1})\qquad \qquad
F_n^{\circ}(d)=Z_{d,d}^{\circ\circ +}(2n,\{B_i\}_{i\geq 1},\{W_i\}_{i\geq 1})
\label{eq:Fnd}
\end{equation}
where $B_i\equiv B_i(\{g_k\}_{k\geq 1},t_\bullet,t_\circ)$ (resp.\ $W_i=W_i(\{g_k\}_{k\geq 1},t_\bullet,t_\circ)$) is the generating function for 
black (resp.\ white) $i$-slices, see Figure \ref{fig:rootedpointeddecomp}. Each face of degree $2k$ in the slice but the root face receives a weight $g_k$. As for vertex weights
in the slice, they are designed so as
to reproduce after concatenation the proper weights for the vertices in the map. To this end, each vertex of the slice receives
the weight $t_\bullet$ or $t_\circ$ according to its color, except for 
the vertices of the right boundary (including the apex and the base vertex labelled $i-1$) which receive the weight $1$ instead. Indeed, after
concatenation of all slices, all the vertices of the map lying on slice boundaries belong to exactly one left boundary hence
already receive their weight from this boundary, see Figure \ref{fig:rootedpointeddecomp}. This holds except 
for the pointed vertex, which belongs only to right boundaries
hence receives a weight $1$, which is consistent with our convention.
Note also, after concatenation of slices, the distance $d_\bullet$ from the black-root vertex to the pointed vertex is the maximum of $\ell-(i-d)$ over
all slices, hence when $\ell$ varies between $1$ and $i$ for all $i$-slices and with $i$ larger than $d+1$ by construction, 
$d_\bullet$  can be any number between $0$ and $d$.

Taking $d=0$, we deduce in particular
\begin{equation}
F_n^{\bullet}=Z_{0,0}^{\bullet\bullet +}(2n,\{B_i\}_{i\geq 1},\{W_i\}_{i\geq 1})
\qquad \qquad F_n^{\circ}=Z_{0,0}^{\circ\circ +}(2n,\{B_i\}_{i\geq 1},\{W_i\}_{i\geq 1})\ .
\label{eq:pathexpr}
\end{equation}
Recall that, by definition, in both path generating functions, each descending step from a black height $i$ to a white height $i-1$ receives the weight
$B_i$ and each descending step from a white height $i$ to a black height $i-1$ receives the weight $W_i$ (and ascending steps receive
no weight).
With the expressions \eqref{eq:pathexpr}, it is now a standard result that we have the equalities
\begin{equation}
\begin{split}
&\sum_{n\geq 0} F_n^{\bullet} z^n=\frac{1}{\displaystyle{1-z \frac{W_1}{\displaystyle{1-z \frac{B_2}{\displaystyle{1-z \frac{W_3}{\displaystyle{1-z \frac{B_4}{\displaystyle{ 1- \cdots}}}}}}}}}}\\
&\sum_{n\geq 0} F_n^{\circ} z^n=\frac{1}{\displaystyle{1-z \frac{B_1}{\displaystyle{1-z \frac{W_2}{\displaystyle{1-z \frac{B_3}{\displaystyle{1-z \frac{W_4}{\displaystyle{ 1- \cdots}}}}}}}}}}\\
\end{split}
\label{eq:contfrac}
\end{equation}
with the convention $F_0^\bullet=F_0^\circ=1$, so that $B_i$ and $W_i$ for $i\geq 1$ can be viewed as the coefficients
in the continued fractions of the ``resolvents" $\sum_{n\geq 0} F_n^{\bullet} z^n$ and $\sum_{n\geq 0} F_n^{\circ} z^n$. 
Note in particular that, expanding at first order in $z$, eq.~\eqref{eq:FbulletFcirc} yields
\begin{equation}
t_\bullet W_1=t_\circ B_1\ .
\label{eq:WoneBone}
\end{equation}

Now it is also a standard result of the theory of continued fractions that, from the first line in \eqref{eq:contfrac}, we may write
for $i\geq 1$
\begin{equation}
B_{2i}=\frac{h_i^{(0)}}{h_{i-1}^{(0)}}\Big{/}\frac{h_{i-1}^{(1)}}{h_{i-2}^{(1)}}\qquad\qquad\qquad
W_{2i-1}=\frac{h_{i-1}^{(1)}}{h_{i-2}^{(1)}}\Big{/}\frac{h_{i-1}^{(0)}}{h_{i-2}^{(0)}}
\label{eq:BWtoh}
\end{equation} 
where $h_i^{(0)}$ and $h_i^{(1)}$ are Hankel determinants defined from the $F_n^{\bullet}$ as
\begin{equation*}
h_i^{(0)}=\det (F^{\bullet}_{n+m})_{0\leq n,m \leq i}\qquad\qquad
h_i^{(1)}=\det (F^{\bullet}_{n+m+1})_{0\leq n,m \leq i}
\end{equation*}
with the convention $h_{-1}^{(0)}=h_{-1}^{(1)}=1$. Similarly, the second line in \eqref{eq:contfrac} gives, for $i\geq 1$
\begin{equation*}
B_{2i-1}=\frac{\tilde{h}_{i-1}^{(1)}}{\tilde{h}_{i-2}^{(1)}}\Big{/}\frac{\tilde{h}_{i-1}^{(0)}}{\tilde{h}_{i-2}^{(0)}}\qquad\qquad\qquad
W_{2i}=\frac{\tilde{h}_i^{(0)}}{\tilde{h}_{i-1}^{(0)}}\Big{/}\frac{\tilde{h}_{i-1}^{(1)}}{\tilde{h}_{i-2}^{(1)}}
\end{equation*} 
where
\begin{equation*}
\tilde{h}_i^{(0)}=\det (F^{\circ}_{n+m})_{0\leq n,m \leq i}\qquad\qquad
\tilde{h}_i^{(1)}=\det (F^{\circ}_{n+m+1})_{0\leq n,m \leq i}
\end{equation*}
and again the convention $\tilde{h}_{-1}^{(0)}=\tilde{h}_{-1}^{(1)}=1$. 
Section \ref{sec:computationHankel} below will be devoted to the calculation of the Hankel determinants $h_i^{(0)}$, $h_i^{(1)}$, $\tilde{h}_i^{(0)}$ and
$\tilde{h}_i^{(1)}$, yielding explicit expressions for the slice generating functions $B_i$ and $W_i$.

\subsection{Link with the two-point function}
\label{sec:twopoinfunction}
\begin{figure}
\begin{center}
\includegraphics[width=9cm]{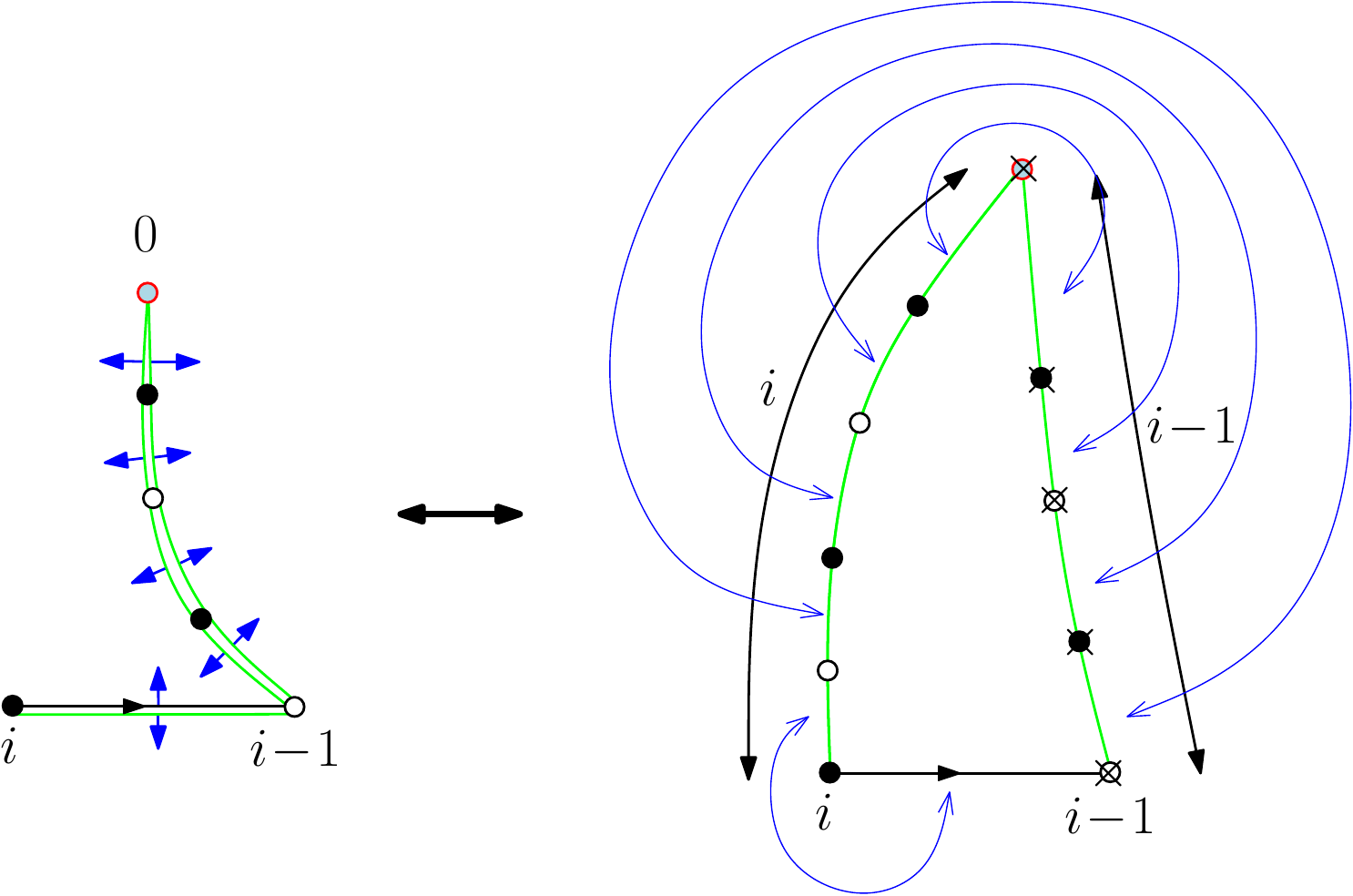}
\end{center}
\caption{Schematic picture of the bijection between maps contributing to the two-point function $G_i^\bullet$
and black $i$-slices with a left boundary of length $i$ (i.e.,\ the maximal allowed value).}
\label{fig:twopoint}
\end{figure}
The reason of our interest in the slice generating functions $B_i$ and $W_i$ is their intimate link
with the distance dependent two-point function. 
Indeed, let us consider a pointed black-rooted bicolored planar map with a root edge
whose black (resp.\ white) extremity is at graph distance $i$ (resp.\ $i-1$) from the pointed vertex.
By cutting the map along its leftmost geodesic path from the root vertex to the pointed vertex (the first step of the geodesic
path being the root edge itself), the resulting object is precisely a black $i$-slice
whose left boundary has the maximal allowed length $\ell=i$, see Figure \ref{fig:twopoint}. This slice must moreover 
contain at least one face. For $i>1$, this is automatic but for $i=1$, we must eliminate the slice 
reduced to a single base edge. This construction is clearly reversible
so that the two-point function of bicolored planar maps, defined as the generating function $G_i^\bullet$  
of  black-rooted bicolored planar maps with a root edge
whose black (resp.\ white) extremity is at graph distance $i$ (resp.\ $i-1$) from the pointed vertex,
is nothing but the generating function of black $i$-slices with a left boundary of length $i$ (and not reduced to the base edge
if $i=1$). For $i>1$, there is an obvious bijection between black $i$-slices with a left boundary of length $1\leq \ell<i$ and
black $i-1$-slices of arbitrary left boundary of length $\ell$ ($1\leq \ell \leq i-1$) by simply relabeling the root edge $i-1\to i-2$.
We immediately deduce the relations
\begin{equation}
G_1^\bullet=t_\circ (B_1-t_\bullet) \qquad \hbox{and} \quad G_{2i}^\bullet=t_\bullet (B_{2i}-B_{2i-1})\ ,\quad G_{2i+1}^\bullet=t_\circ (B_{2i+1}-B_{2i})\quad i>1 \ ,
\label{eq:teopointbullet}
\end{equation}
where we have re-introduced the weight of the pointed vertex.
We may define alternatively the generating function $G_i^\circ$  
of  white-rooted bicolored planar maps with a root edge
whose white (resp.\ black) extremity is at graph distance $i$ (resp.\ $i-1$) from the pointed vertex.
Obviously, we have the relations
\begin{equation}
G_1^\circ=t_\bullet (W_1-t_\circ) \qquad \hbox{and} \quad G_{2i}^\circ=t_\circ(W_{2i}-W_{2i-1})\ , \quad G_{2i+1}^\circ=t_\bullet (W_{2i+1}-W_{2i})\quad i>1 \ .
\label{eq:teopointcirc}
\end{equation}

\subsection{Recursive equations for $i$-slice generating functions}
\label{sec:recusiveislice}
As slice generating functions, $B_i$ and $W_i$ satisfy non-linear recursion relations which can
be obtained as follows: assuming $\ell>1$, the face on the left of the root edge of an $i$-slice 
is necessarily different from the root face and has, say degree $2k$. The set of distances to the apex of the successive
vertices, when going around this face of degree $2k$ from the root vertex to the other extremity of the root edge
forms a path made of $+1$ and $-1$ steps, of length $2k-1$ from height $i$ to height $i-1$. Drawing 
from each of these vertices its leftmost geodesic path to the apex results into a slice decomposition 
of the $i$-slice, see Figure \ref{fig:islice}, from which we immediately deduce (see \cite{BG12} for more explanations)
\begin{equation}
\begin{split}
&B_i=t_\bullet+\sum_{k\geq 1} g_k Z_{i,i-1}^{\bullet\circ}(2k-1,\{B_j\}_{j\geq 1},\{W_j\}_{j\geq 1})\\
&W_i=t_\circ+\sum_{k\geq 1} g_k Z_{i,i-1}^{\circ\bullet}(2k-1,\{B_j\}_{j\geq 1},\{W_j\}_{j\geq 1})\ .\\
\label{eq:recurBiWi}
\end{split}
\end{equation}
Here, $Z_{d,d'}^{\bullet\circ}(2k-1)$ ($d,d'\geq 0,\ d'=d+1 \mod 2$) denotes the generating function 
of paths of length $2k-1$ made of $+1$ or $-1$ steps, colored alternatively in black and white, starting at black height $d$ and 
ending at white height $d'$ and remaining above height $0$, where each descending step from a black height 
$i$ to a white height $i-1$ receives a weight
$B_i$ and each descending step from a white height $i$ to a black height $i-1$ receives a weight
$W_i$ (and with no weights assigned to ascending steps).  
The quantity $Z_{d,d'}^{\circ\bullet }(2k-1)$ is defined similarly in an obvious way.
The first term $t_\bullet$ (resp.\ $t_\circ$) in \eqref {eq:recurBiWi} arises as the contribution, when $\ell=1$,
of the black (resp.\ white) $i$-slice reduced to a single $i\to i-1$ edge (which receives the 
weight of the root vertex only since the other vertex belongs to the right boundary).

As an example, let us consider the case of quadrangulations, i.e.,\ maps whose all faces 
have degree $4$, by taking $g_k=\delta_{k,2}$
\footnote{We decide not to keep track of 
the number of faces via an arbitrary weight $g_2$ since, by the Euler relation, 
this number of faces is the total number of black and white vertices minus $2$. A similar remark holds for hexangulations,
where this time the double number of faces is the total number of black and white vertices minus $2$.}. 
The two-point functions $G_i^\bullet$ and  $G_i^\circ$ for these maps are 
obtained via \eqref{eq:teopointbullet} and \eqref{eq:teopointcirc} where $B_i$ 
and $W_i$ are solutions of the system
\begin{equation}
\begin{split}
&B_i=t_\bullet+Z_{i,i-1}^{\bullet\circ}(3,\{B_j\}_{j\geq 1},\{W_j\}_{j\geq 1})=t_\bullet+B_i(W_{i-1}+B_i+W_{i+1})\\
&W_i=t_\circ+Z_{i,i-1}^{\circ\bullet}(3,\{B_j\}_{j\geq 1},\{W_j\}_{j\geq 1})=t_\circ+W_i(B_{i-1}+W_i+B_{i+1})\\
\end{split}
\label{BiWiquad}
\end{equation}
valid for all $i\geq 1$ with the convention $B_0=W_0=0$. As shown in Appendix A, the solution of these equations 
may be guessed (with the help of a computer), using the same technique as in \cite{GEOD}, based on a perturbative method.

In the case of hexangulations, i.e.,\ maps whose all faces 
have degree $6$, obtained by taking $g_k=\delta_{k,3}$, we get instead
\begin{equation*}
\begin{split}
& \hspace{-1.2cm} B_i=t_\bullet+B_i(B_{i-2}W_{i-1}\!+\!2B_i(W_{i-1}\!+\!W_{i+1})\!+\!B_i^2\!+\!W_{i-1}^2\!+\!W_{i+1}^2\!+\!W_{i-1}W_{i+1}
\!+\!W_{i+1}B_{i+2})\\
& \hspace{-1.2cm} 
W_i=t_\circ+W_i(W_{i-2}B_{i-1}\!+\!2W_i(B_{i-1}\!+\!B_{i+1})\!+\!W_i^2+B_{i-1}^2\!+\!B_{i+1}^2 \!+\!B_{i-1}B_{i+1}\!+B_{i+1}W_{i+2})\\
\end{split}
\end{equation*}
valid for all $i\geq 1$ with the convention $B_0=W_0=B_{-1}=W_{-1}=0$.
Note that, in all generality, the equations \eqref{eq:recurBiWi} form two independent systems of equations,
one involving only $B_i$'s with even indices $i$ and $W_j$'s with odd indices $j$, and one 
involving only $B_i$'s with odd $i$ and $W_j$'s with even $j$.

Clearly, in the definition of $i$-slices, the quantity $i$ only acts as an upper bound on the length of the left boundary.
We can release this upper bound by simply letting $i\to \infty$ in which case $B_i$ and $W_i$ tend to well-defined 
limits $B$ and $W$ solutions of
\begin{equation}
\begin{split}
&B=t_\bullet+\sum_{k\geq 1} g_k \mathbb{Z}_{0,-1}^{\bullet\circ}(2k-1,B,W)\\
&W=t_\circ+\sum_{k\geq 1} g_k \mathbb{Z}_{0,-1}^{\circ\bullet}(2k-1,B,W)\ .\\
\label{eq:recurBW}
\end{split}
\end{equation}
where $\mathbb{Z}_{m,n}^{\bullet\circ}(2k-1)\equiv \mathbb{Z}_{m,n}^{\bullet\circ}(2k-1,B,W)$ is the analog of $Z_{m,n}^{\bullet\circ}(2k-1)$
except that heights are allowed to be negative and all descending steps from a black height to a white height receive the same weight $B$ \emph{irrespectively of their height},
and all descending steps from a white height to a black height receive the same weight $W$. Clearly 
$\mathbb{Z}_{m,n}^{\bullet\circ}(2k-1)=\mathbb{Z}_{0,n-m}^{\bullet\circ}(2k-1)$ for all $m,n$. We have a similar definition
of $\mathbb{Z}_{m,n}^{\circ\bullet}(2k-1)$ and in the following, we shall also consider all variants $\mathbb{Z}_{m,n}^{\bullet\bullet}(2k)$,
$\mathbb{Z}_{m,n}^{\bullet\bullet +}(2k)$, ... with obvious definitions.

As an example, eqs.~\eqref{eq:recurBW} give for quadrangulations
\begin{equation}
\begin{split}
&B=t_\bullet+B(B+2W)\\
&W=t_\circ+W(W+2B)\\
\label{eq:recurBWquad}
\end{split}
\end{equation}
and for hexangulations
\begin{equation}
\begin{split}
&B=t_\bullet+B(B^2+3 W^2 + 6 B W)\\
&W=t_\circ+W(W^2 +3 B^2 +6 B W)\ .\\
\label{eq:recurBWhex}
\end{split}
\end{equation}

\section{Expressions for $F_n^\bullet$ and $F_n^\circ$}
\label{sec:expressionsFn}
\subsection{Conserved quantities}
\label{sec:conservedquantities}
\begin{figure}
\begin{center}
\includegraphics[width=11cm]{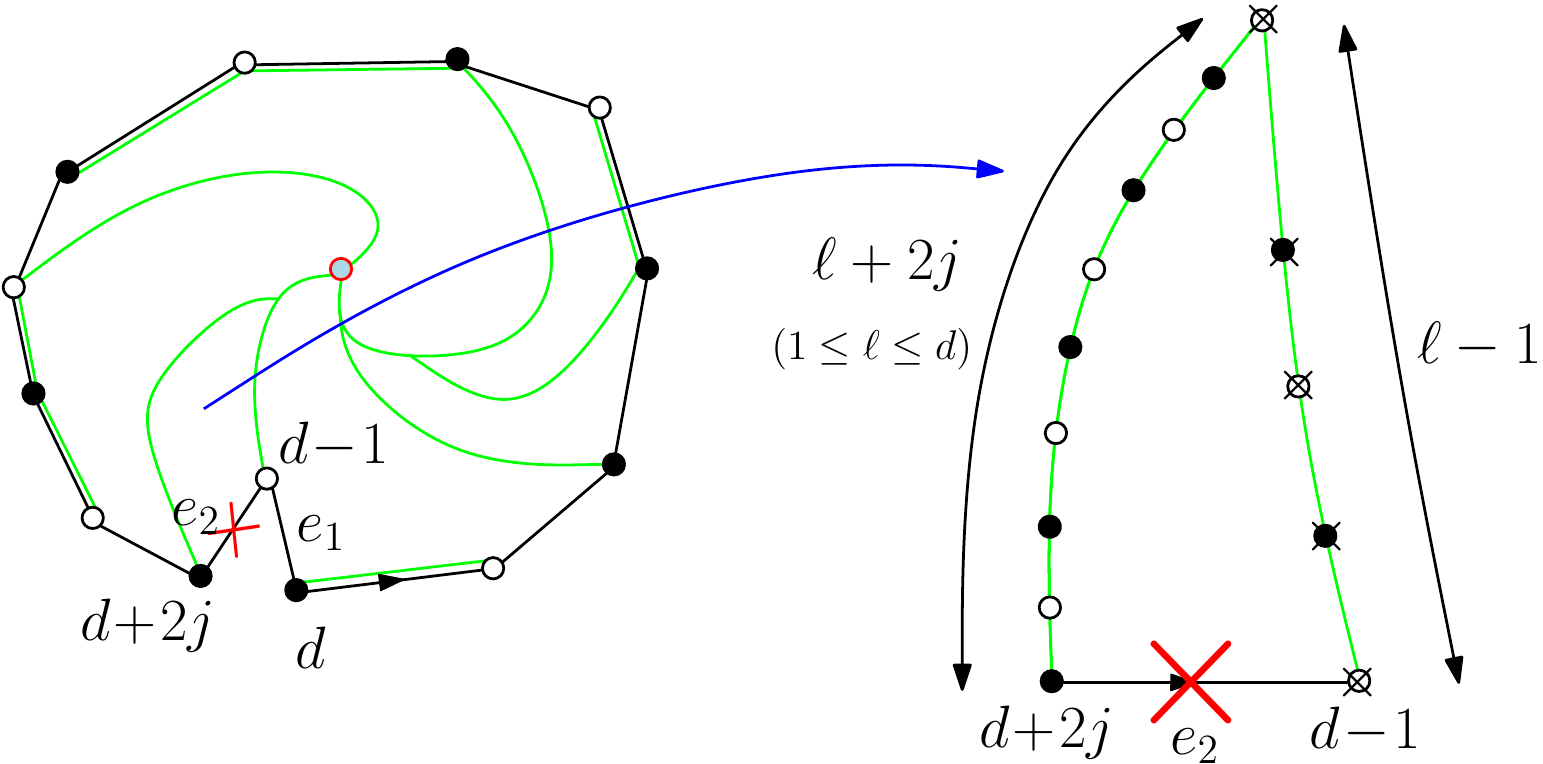}
\end{center}
\caption{Schematic picture of the map obtained from a map contributing to $F_n^\bullet(d)$ and having 
a distance $d_\bullet \geq 1$ between the root vertex and the pointed vertex, by cutting along the
leftmost edge from the root vertex to a neighboring vertex at distance $d_\bullet -1$ from the pointed vertex.
After cutting, this edge is split into two edges $e_1$ and $e_2$, linking vertices with successive labels $d+2j$
(for some $j\geq 1$) , $d-1$
and $d$ counterclockwise, where we use as labels for the boundary vertices $d-d_\bullet$ plus their distance 
to the pointed vertex \emph{using paths avoiding} $e_2$ (as indicated by the red cross). The slice decomposition
of this map gives rise to a new type of slice, with base $e_2$, whose left boundary now has length $\ell+2j$
and right boundary has length $\ell-1$ for some $\ell$ between $1$ and $d$.}
\label{fig:conserved}
\end{figure}

As was done in \cite{BG12}, the quantity $F_n^\bullet=F_n^\bullet(0)$ may be obtained by subtracting from
$F_n^\bullet(d)$ the generating function of those configurations having a distance $d_\bullet$ strictly positive, i.e.,\ 
lying between $1$ and $d$ (recall that $d_\bullet$ is the distance from the root vertex to the pointed vertex). When 
$d_\bullet\geq 1$, the root vertex has some neighbors at distance $d_\bullet-1$ from the pointed vertex which 
necessarily lie strictly inside the map (i.e.,\ not on the boundary) as otherwise, the root vertex would not
be one of the closest vertices to the pointed vertex. Picking the leftmost edge $e$ leading from the root vertex to such a neighbor,
we may duplicate this edge by cutting along it and create a map with a boundary
of length $2n+2$ whose two new boundary edges are the two copies $e_1$ and $e_2$ of the edge $e$, with $e_1$
say the closest to the root vertex, see Figure \ref{fig:conserved}, so that, counterclockwise around the map, $e_2$ links the duplicate
of the root vertex to its chosen neighbor at distance $d_\bullet-1$, and $e_1$ links this neighbor to the original root vertex. 
We may now apply to this new map the same slice decomposition as before,
using now a new graph distance by \emph{preventing paths to go through the edge $e_2$} and labeling
again the boundary vertices by this distance plus $(d-d_\bullet)$. Because of
the new constraint of using only paths avoiding $e_2$, the duplicate of the root vertex now receives a label $d+2j$ for 
some $j\geq 1$ (the fact that the distance 
to the pointed vertex is strictly larger than $d_\bullet$ is due to our choice of leftmost edge, see \cite{BG12} 
for a detailed argument). As for all the other boundary vertices, their distance necessarily increases (weakly) in
the edge cutting procedure so their label also increases, hence remains above $d$. 
Moreover, all the obtained slices are $i$-slices as we defined them before except
for the last slice whose base is $e_2$ which has a left boundary of length $\ell+2j$ for some $\ell$ between $1$ and $d$, 
and a right boundary of length $\ell-1$ (distances in the slice are also measured with paths avoiding $e_2$),
see Figure \ref{fig:conserved}.
By an argument similar to that used to derive  \eqref{eq:recurBiWi}, the generating function for this last slice is
\begin{equation*}
\sum_{k\geq 1} g_k Z_{d+2j,d-1}^{\bullet\circ}(2k-1,\{B_j\}_{j\geq 1},\{W_j\}_{j\geq 1})
\end{equation*}
so that the quantity to be subtracted from $F_n^{\bullet}(d)$ to get $F_n^\bullet$ is
\begin{equation*}
\frac{1}{t_\bullet}\sum_{j\geq 1} Z_{d,d+2j}^{\bullet\bullet+}(2n) \sum_{k\geq 1} g_k Z_{d+2j,d-1}^{\bullet\circ}(2k-1)
\end{equation*}
where the role of the factor $1/t_\bullet$ is to avoid counting the weight of the original root vertex twice.
Using the expression \eqref{eq:Fnd} for $F_n^\bullet(d)$ (and repeating the argument for $F_n^\circ$), we arrive at 
\begin{equation}
\begin{split}
& F_n^\bullet=Z_{d,d}^{\bullet\bullet+}(2n)-\frac{1}{t_\bullet}\sum_{j\geq 1} Z_{d,d+2j}^{\bullet\bullet+}(2n) \sum_{k\geq 1} g_k Z_{d+2j,d-1}^{\bullet\circ}(2k-1)\\
& F_n^\circ=Z_{d,d}^{\circ\circ+}(2n)-\frac{1}{t_\circ}\sum_{j\geq 1} Z_{d,d+2j}^{\circ\circ+}(2n) \sum_{k\geq 1} g_k Z_{d+2j,d-1}^{\circ\bullet}(2k-1)\ .\\
\end{split}
\label{eq:conserved}
\end{equation}
The quantities on the right hand side are called \emph{conserved quantities} to emphasize the fact that their actual values are independent
of $d$ (since the left hand size does not depend on $d$).

In the case of quadrangulations for instance taking $n=1$, we get the following conserved 
quantities
\begin{equation}
\begin{split}
& F_1^\bullet=W_{d+1}-\frac{1}{t_\bullet}B_{d}W_{d+1}B_{d+2} = W_1=W-\frac{1}{t_\bullet}B^2W \\
& F_1^\circ=B_{d+1}-\frac{1}{t_\circ}W_{d}B_{d+1}W_{d+2} = B_1=B-\frac{1}{t_\circ}W^2B \\
\end{split}
\label{eq:F1quad}
\end{equation}
for all $d\geq 0$ (the last two members of the equalities correspond to $d=0$ and $d\to \infty$ respectively).

Note that in this case, a direct proof of this equation may be obtained by writing the relation giving $B_i$ in \eqref{BiWiquad} for $i=d+1$, 
multiplying it by $W_{d}$, then writing the relation giving $W_i$ for $i=d$ and
multiplying it by $B_{d+1}$ and finally subtracting the two. This leads to $t_\circ c_d=t_\bullet \tilde{c}_{d-1}$
where $c_d=B_{d+1}-\frac{1}{t_\circ}W_{d}B_{d+1}W_{d+2}$ and $\tilde{c}_d=W_{d+1}-\frac{1}{t_\bullet}B_{d}W_{d+1}B_{d+2}$.
By symmetry, we also have $t_\bullet \tilde{c}_d=t_\circ c_{d-1}$ so that $c_{2i}=\frac{t_\bullet}{t_\circ} \tilde{c}_{2i-1}=c_0=B_1$
and $\tilde{c}_{2i}=\frac{t_\circ}{t_\bullet} c_{2i-1}=\tilde{c}_0=W_1$ for all $i\geq 1$. This is equivalent to $c_d=B_1$ and 
$\tilde{c}_d=W_1$ for all $d$ thanks to the identity \eqref{eq:WoneBone}. As explained in Appendix A, the use 
of the above two conserved quantities in the case of quadrangulations allows us to directly derive an explicit expression
for $B_i$ and $W_i$ without recourse to the general formalism that we shall develop below.

\subsection{Expression for $F_n^\bullet$ and $F_n^\circ$}
\label{sec:expressionsFnexpli}

Expressions \eqref{eq:conserved} for the conserved quantities are particularly interesting as they give expressions
for $F_n^\bullet$ and $F_n^\circ$ in terms of $B$ and $W$ only, by simply letting $d\to \infty$. 
Indeed, we immediately get
\begin{equation*}
\begin{split}
& F_n^\bullet=\mathbb{Z}_{0,0}^{\bullet\bullet+}(2n)-\frac{1}{t_\bullet}\sum_{j\geq 1} \mathbb{Z}_{0,2j}^{\bullet\bullet+}(2n) \sum_{k\geq 1} g_k 
\mathbb{Z}_{2j,-1}^{\bullet\circ}(2k-1)\\
& F_n^\circ=\mathbb{Z}_{0,0}^{\circ\circ+}(2n)-\frac{1}{t_\circ}\sum_{j\geq 1} \mathbb{Z}_{0,2j}^{\circ\circ+}(2n) \sum_{k\geq 1} g_k 
\mathbb{Z}_{2j,-1}^{\circ\bullet}(2k-1)\ .\\
\end{split}
\end{equation*}
Using \eqref{eq:recurBW}, these equations read equivalently
 \begin{equation*}
\begin{split}
& F_n^\bullet=\frac{1}{t_\bullet}\left(B \mathbb{Z}_{0,0}^{\bullet\bullet+}(2n)-\sum_{k\geq 1} g_k \sum_{j\geq 0} \mathbb{Z}_{0,2j}^{\bullet\bullet+}(2n)  
\mathbb{Z}_{2j,-1}^{\bullet\circ}(2k-1)\right)\\
& F_n^\circ=\frac{1}{t_\circ}\left(W \mathbb{Z}_{0,0}^{\circ\circ+}(2n)-\sum_{k\geq 1} g_k \sum_{j\geq 0} \mathbb{Z}_{0,2j}^{\circ\circ+}(2n)  
\mathbb{Z}_{2j,-1}^{\circ\bullet}(2k-1)\right)\\
\end{split}
\end{equation*}
where the sum over $j$ now starts at $j=0$. Now $\sum_{j\geq 0} \mathbb{Z}_{0,2j}^{\bullet\bullet+}(2n)  
\mathbb{Z}_{2j,-1}^{\bullet\circ}(2k-1)$ enumerates paths form $0$ to $-1$, of total length $2n+2k-1$, and whose
$2n$ first steps stay above $0$. Their first $0\to -1$ step therefore occurs after some length $2n+2q$ for some $q\geq 0$.
This step receives the weight $B$ and is followed by a path from a white height $-1$ to a white height $-1$ 
of length $2k-2q-2$ (which thus implies $q\leq k-1$).
We may therefore write 
\begin{equation*}
\begin{split}
&\sum_{j\geq 0} \mathbb{Z}_{0,2j}^{\bullet\bullet+}(2n)  
\mathbb{Z}_{2j,-1}^{\bullet\circ}(2k-1)=B \sum_{q=0}^{k-1} \mathbb{Z}_{0,0}^{\bullet\bullet+}(2n+2q)  
\mathbb{Z}_{-1,-1}^{\circ\circ}(2k-2q-2)\\
&\sum_{j\geq 0} \mathbb{Z}_{0,2j}^{\circ\circ+}(2n)  
\mathbb{Z}_{2j,-1}^{\circ\bullet}(2k-1)=W \sum_{q=0}^{k-1} \mathbb{Z}_{0,0}^{\circ\circ+}(2n+2q)  
\mathbb{Z}_{-1,-1}^{\bullet\bullet}(2k-2q-2)\\
\end{split}
\end{equation*}
\begin{figure}
\begin{center}
\includegraphics[width=10cm]{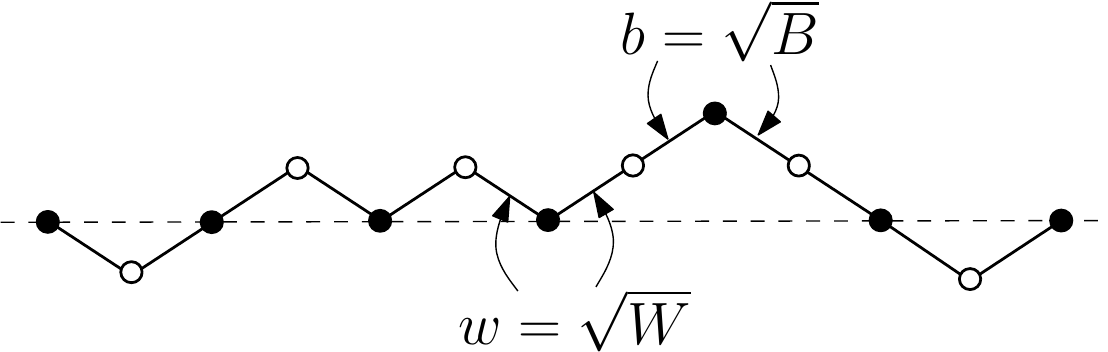}
\end{center}
\caption{Path weights for the generating functions $\hat{\mathbb{Z}}$.  A weight is assigned to both descending 
and ascending steps: $b=\sqrt{B}$ if the lower height is white and $w=\sqrt{W}$ if the lower height is black.}
\label{fig:pathweights}
\end{figure}
So far all path weights have been assigned to descending steps only. From now on we shall use slightly
different conventions for path weights by assigning, see Figure \ref{fig:pathweights}:
\begin{itemize}
\item a weight $b=\sqrt{B}$  to each descending step from a black height to a white height and to each
ascending step from a white height to a black height \footnote{When taking the square-root of a generating 
function that is positive for positive weights, we naturally take the positive determination.};
\item a weight $w=\sqrt{W}$ to each descending step from a white height to a black height and to each
ascending step from a black height to a white height.
\end{itemize}
We shall denote by $\hat{\mathbb{Z}}$ (instead of $\mathbb{Z}$) the corresponding generating functions.
Clearly, when the heights of the two extremities of the path are the same, the ascending and descending steps
of each sort are well balanced so that, for instance $\mathbb{Z}_{0,0}^{\bullet\bullet+}(2n+2q)=\hat{\mathbb{Z}}_{0,0}^{\bullet\bullet+}(2n+2q)$
and $\mathbb{Z}_{-1,-1}^{\circ\circ}(2k-2q-2)=\hat{\mathbb{Z}}_{-1,-1}^{\circ\circ}(2k-2q-2)$. Moreover, by reversing the
paths vertically, we have $\hat{\mathbb{Z}}_{0,-2k}^{\bullet\bullet}(2n)=\hat{\mathbb{Z}}_{-2k,0}^{\bullet\bullet}(2n)$
since the weights are unchanged under reversing. If we now reverse the paths horizontally, we get instead
$\hat{\mathbb{Z}}_{-2k,0}^{\bullet\bullet}(2n)=\hat{\mathbb{Z}}_{0,-2k}^{\circ\circ}(2n)$ since under this reversing, the colors
have to be exchanged for the weights $b$ and $w$ to remain correct.
We may thus introduce the function
\begin{equation}
L_k(2n)\equiv \hat{\mathbb{Z}}_{i,i-2k}^{\bullet\bullet}(2n)=\hat{\mathbb{Z}}_{i,i-2k}^{\circ\circ}(2n)
\label{eq:Lkdef}
\end{equation} 
(note that the last two terms are indeed independent of $i$)
and it satisfies the relation
\begin{equation*}
L_k(2n)=L_{-k}(2n)\ .
\end{equation*}
With these new notations, we end up with
\begin{equation}
\begin{split}
& F_n^\bullet=\sum_{g\geq 0} \alpha_q \hat{\mathbb{Z}}_{0,0}^{\bullet\bullet+}(2n+2q) 
\qquad
\alpha_q=\frac{B}{t_\bullet} \left(\delta_{q,0}- \sum_{k\geq q+1} g_kL_0(2k-2q-2)\right)\\ 
& F_n^\circ=\sum_{g\geq 0} \tilde{\alpha}_q \hat{\mathbb{Z}}_{0,0}^{\circ\circ+}(2n+2q) 
\qquad
\tilde{\alpha}_q=\frac{W}{t_\circ} \left(\delta_{q,0}- \sum_{k\geq q+1} g_kL_0(2k-2q-2)\right)\\
\end{split}
\label{Fnformula}
\end{equation}

To illustrate this formula, let us return to the case of quadrangulations (where $g_2$ is omitted). 
In this case, only $\alpha_0$ and $\alpha_1$ are
non-zero, and have values
\begin{equation*}
\begin{split}
& \alpha_0=\frac{B}{t_\bullet}(1-L_0(2))=\frac{B}{t_\bullet}(1-(B+W))=1+\frac{B W}{t_\bullet}\\
&\alpha_1=\frac{B}{t_\bullet}(-L_0(0))=-\frac{B}{t_\bullet}\\
&\tilde{\alpha}_0=\frac{W}{t_\circ}(1-L_0(2))=\frac{W}{t_\circ}(1-(B+W))=1+\frac{B W}{t_\circ}\\
&\tilde{\alpha}_1=\frac{W}{t_\circ}(-L_0(0))=-\frac{W}{t_\circ}\ .\\
\end{split}
\end{equation*}
Here we have used eqs.~\eqref{eq:recurBWquad} to simplify the first and third lines.
This leads to
\begin{equation*}
\begin{split}
F_1^\bullet=& \alpha_0 \hat{\mathbb{Z}}_{0,0}^{\bullet\bullet +}(2)+\alpha_1 \hat{\mathbb{Z}}_{0,0}^{\bullet\bullet +}(4)\\
&= \alpha_0 W+\alpha_1 (W^2+B W)=W-\frac{1}{t_\bullet} B^2 W \\
F_1^\circ=& \tilde{\alpha}_0 \hat{\mathbb{Z}}_{0,0}^{\circ\circ+}(2)+\tilde{\alpha}_1 \hat{\mathbb{Z}}_{0,0}^{\circ\circ +}(4)\\
&= \tilde{\alpha}_0 B+\tilde{\alpha}_1 (B^2+B W)=B-\frac{1}{t_\circ} W^2 B \\
\end{split}
\end{equation*}
in agreement with eqs.~\eqref{eq:F1quad}.

\section{Computation of the Hankel determinants}
\label{sec:computationHankel}
\subsection{Computation of $h_i^{(1)}$ and $\tilde{h}_i^{(1)}$}
\label{sec:computationHone}
The computation of $h_i^{(1)}$ and $\tilde{h}_i^{(1)}$ turns out to be simple as it takes
exactly the same form as that performed in \cite{BG12}. Indeed, we may use the 
relation 
\begin{equation*}
\hat{\mathbb{Z}}_{0,0}^{\bullet\bullet+}(2m+2n+2+2q)=\sum_{k=1}^{m+1}\sum_{\ell=1}^{n+1}
\hat{\mathbb{Z}}_{0,2k-1}^{\bullet\circ +}(2m+1)A_{2k-1,2\ell-1}^{\circ\circ}(2q)\hat{\mathbb{Z}}_{2\ell-1,0}^{\circ\bullet+}(2n+1)
\end{equation*}
obtained by classifying the paths according to the heights $2k-1$ and $2\ell-1$ after $2m+1$ and $2m+1+2q$ steps.
Here $A_{2k-1,2\ell-1}^{\circ\circ}(2q)$ denotes the generating function of paths of length $2q$ from white height 
$2k-1$ to white height $2\ell-1$ (with the weights $b=\sqrt{B}$ and $w=\sqrt{W}$ assigned to both ascending and descending
steps according to the color of their extremities) which remain above height $0$ (note that height $0$ is black in this case).

This results into the matrix identity
\begin{equation*}
\hspace{-1.3cm}(F_{n+m+1}^\bullet)_{0\leq m,n\leq i}= (\hat{\mathbb{Z}}_{0,2k-1}^{\bullet\circ +}(2m+1))_{0\leq m\leq i \atop 1\leq k\leq i+1}\cdot(\sum_{q\geq 0} \alpha_q A_{2k-1,2\ell-1}^{\circ\circ}(2q))_{1\leq k,\ell\leq i+1}\cdot (\hat{\mathbb{Z}}_{2\ell-1,0}^{\circ\bullet+}(2n+1))_{1\leq \ell\leq i+1\atop 0\leq n\leq i}
\end{equation*}
Taking the determinant of both sides of this identity and noting that the two extremal matrices in the right hand side
are triangular matrices whose determinants are trivially computed, we immediately get to
\begin{equation*}
h_i^{(1)}= W^{i+1} (B W)^{\frac{i(i+1)}{2}} \det_{1\leq k,\ell\leq i+1} (\sum_{q\geq 0} \alpha_q A_{2k-1,2\ell-1}^{\circ\circ}(2q))
\end{equation*} 
and similarly
\begin{equation*}
\tilde{h}_i^{(1)}= B^{i+1} (B W)^{\frac{i(i+1)}{2}} \det_{1\leq k,\ell\leq i+1} (\sum_{q\geq 0} \alpha_q A_{2k-1,2\ell-1}^{\bullet\bullet}(2q))
\end{equation*}
where $A_{2k-1,2\ell-1}^{\bullet\bullet}(2q)$ now denotes the generating function of paths of length $2q$ from black height 
$2k-1$ to black height $2\ell-1$ (with the weights $b=\sqrt{B}$ and $w=\sqrt{W}$ assigned to both ascending and descending
steps according to the color of their extremities) which remain above height $0$ (height $0$ being white in this case).
\begin{figure}
\begin{center}
\includegraphics[width=11cm]{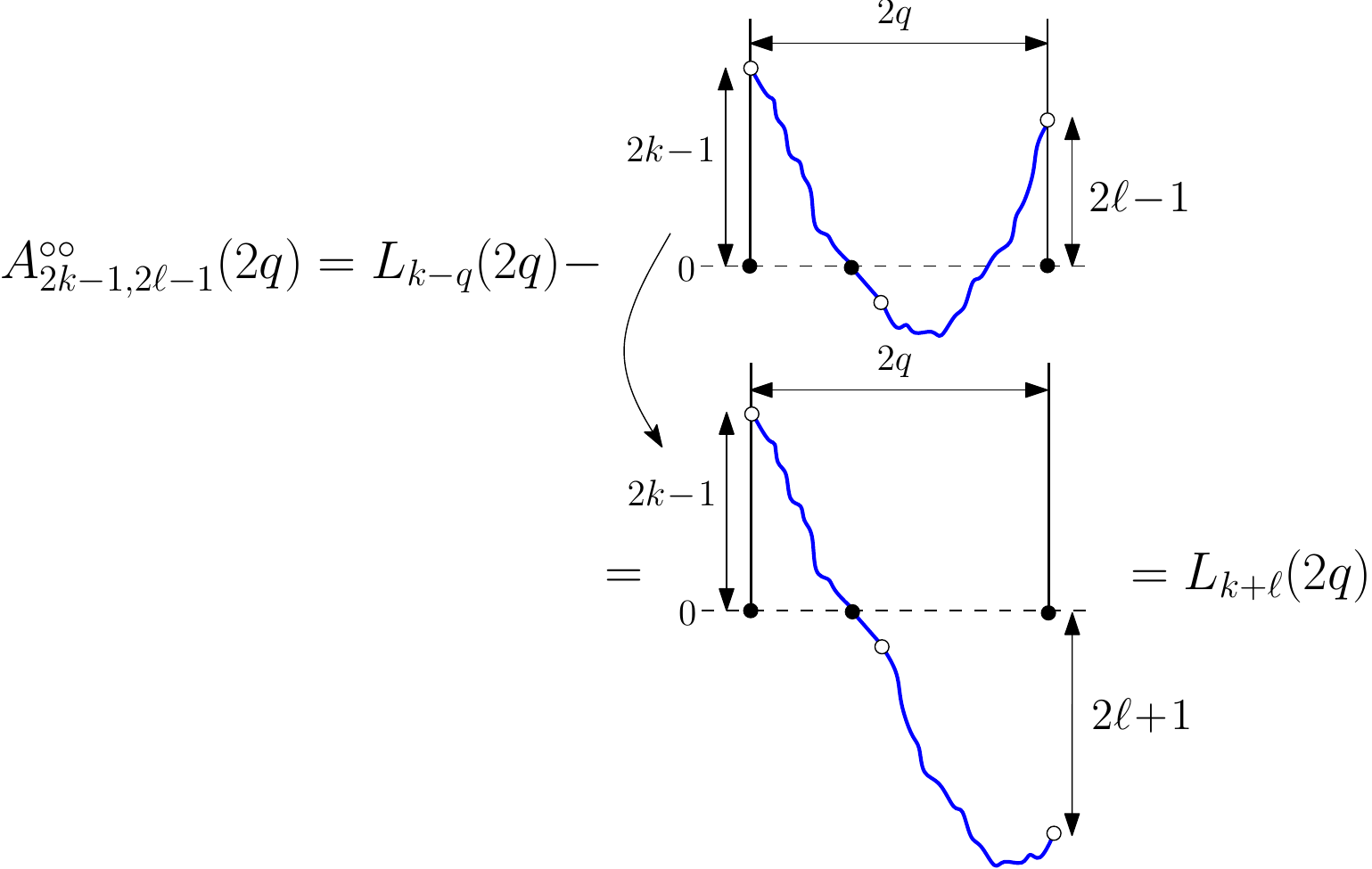}
\end{center}
\caption{By a standard reflection principle, $A_{2k-1,2\ell-1}^{\circ\circ}(2q)$ is obtained by subtracting from $L_{k-\ell}(2q)$
paths which go below $0$, which are in bijection with paths with height decrease $(k+\ell)$, as enumerated by $L_{k+\ell}(2q)$.
Note that the rightmost part of the path has been returned \emph{vertically} in the argument in order for
the weights after reversing to still be correct.}
\label{fig:reflectionone}
\end{figure}

Let us from now on concentrate on $h_i^{(1)}$. The quantity $A_{2k-1,2\ell-1}^{\circ\circ}(2q)$ may be
obtained via a simple (and standard) reflection principle, namely
\begin{equation*}
A_{2k-1,2\ell-1}^{\circ\circ}(2q)=L_{k-\ell}(2q)-L_{k+\ell}(2q)\ .
\end{equation*}
Indeed, paths contributing to  $A_{2k-1,2\ell-1}^{\circ\circ}(2q)$ are identical to those which contribute
to $\hat{\mathbb{Z}}_{2k-1,2\ell-1}^{\circ\circ}(2q)=L_{k-\ell}(2q)$ except that they have to remain
above height $0$. The paths to be subtracted are those paths reaching a negative height. Considering
the first step on the negative side (this $0\to -1$ step receives the weight $b=\sqrt{B}$ since height $0$ is black),
the rest of the path goes from a white height $-1$ to a white height $2\ell-1$. Returning this
path \emph{vertically} (which does not modify the weight prescription), see Figure \ref{fig:reflectionone}, and shifting the heights by $2\ell$,
we get a properly weighted path from height $2k-1$ to height $-1-2\ell$, as enumerated by
$\hat{\mathbb{Z}}_{2k-1,-1-2\ell}^{\circ\circ}(2q)=L_{k+\ell}(2q)$, hence the formula.
We deduce 
\begin{equation}
h_i^{(1)}= W^{i+1} (B W)^{\frac{i(i+1)}{2}} \det_{1\leq k,\ell\leq i+1} (C_{k-\ell}-C_{k+\ell})
\label{eq:dethi}
\end{equation} 
where
\begin{equation*}
C_k=\sum_{q\geq 0} \alpha_q L_k(2q)\ .
\end{equation*}
From now on, we shall assume that the $g_k$'s for $k> p+1$ are all $0$ and that $g_{p+1}\neq 0$,
for some $p\geq 1$. In other words, we enumerate maps with faces of degree at most $2p+2$. The
case of quadrangulations corresponds to $p=1$ (and $g_1=0$) and that of hexangulations to $p=2$
(and $g_1=g_2=0$).
Note in this case that $\alpha_q=0$ for $q>p$, hence, from its definition, $C_k=0$ for $|k|>p$, while
\begin{equation*}
C_p=C_{-p}=-\frac{B}{t_\bullet}g_{p+1} (B W)^{p/2}\ .
\end{equation*}
A simple formula can then be written for the determinant appearing in eq.~\eqref{eq:dethi} in terms of
the solutions $x_a$ of the so-called characteristic equation
\begin{equation}
0=\sum_{k=-p}^p C_k x^k= C_0+\sum_{k=1}^{p} C_k \left(x^k+\frac{1}{x^k}\right)
\label{eq:chareq}
\end{equation}
since $L_k(2q)=L_{-k}(2q)$, hence $C_k=C_{-k}$. This equation has $2p$ solutions which we denote by $x_a$ and $1/x_a$, $a=1,\cdots,p$ 
($x_a$ being chosen, say with modulus less than $1$). The $p$ quantities $x_a$ may be viewed 
as a parametrization of the $p$ quantities $C_0,C_1,\cdots C_{p-1}$ via
\begin{equation}
C_k=(-1)^{p+k} C_p e_{p+k}\left(x_1,x_2,\dots, x_p,\frac{1}{x_1},\frac{1}{x_2},\cdots,\frac{1}{x_p}\right)
\label{eq:Ctoe}
\end{equation}
where the $e_i$'s denote the usual elementary symmetric polynomials.
It is now a standard result of representation theory \cite{FH91} (already used in \cite{BG12}) that
\begin{equation}
\det\limits_{1\leq k,\ell\leq i+1} (C_{k-\ell}-C_{k+\ell})=(-1)^{p(i+1)} C_p^{i+1} \frac{\det\limits_{1\leq a,a'\leq p} (x_a^{i+1+a'}-x_a^{-(i+1+a')})}
{\det\limits_{1\leq a,a'\leq p} (x_a^{a'}-x_a^{-a'})}\ .
\label{eq:weylrep}
\end{equation}
This leads eventually to 
\begin{equation}
h_i^{(1)}= W^{i+1} (B W)^{\frac{i(i+1)}{2}}(-1)^{p(i+1)} C_p^{i+1} \frac{\det\limits_{1\leq a,a'\leq p} (x_a^{i+1+a'}-x_a^{-(i+1+a')})}
{\det\limits_{1\leq a,a'\leq p} (x_a^{a'}-x_a^{-a'})})
\label{eq:dethiexp}
\end{equation}
and, by a similar argument
\begin{equation*}
\tilde{h}_i^{(1)}=B^{i+1} (B W)^{\frac{i(i+1)}{2}}(-1)^{p(i+1)} \tilde{C}_p^{i+1} \frac{\det\limits_{1\leq a,a'\leq p} (x_a^{i+1+a'}-x_a^{-(i+1+a')})}
{\det\limits_{1\leq a,a'\leq p} (x_a^{a'}-x_a^{-a'})})
\end{equation*}
with $\tilde{C}_p=-\frac{W}{t_\circ}g_{p+1} (B W)^{p/2}$. Note that, since the $\alpha_q$'s and the $\tilde{\alpha}_q$'s are proportional,
and because of \eqref{eq:Lkdef}, the characteristic equation is the same in the calculation of $\tilde{h}_i^{(1)}$ as that
for $h_i^{(1)}$, hence the $x_a$'s are the same.

As just mentioned, eq.~\eqref{eq:weylrep} is a standard result of representation theory, whose proof can be found in Appendix A of 
\cite{FH91}. Still, the proof of \cite{FH91} is not so enlightening to the neophyte and it is instructive to recover this result via a more 
heuristic argument. From the characteristic equation, we deduce that the vectors $(x_a^\ell)_{\ell \in \mathbb{Z}}$ and
$(x_a^{-\ell})_{\ell \in \mathbb{Z}}$ (for any $x_a$ solution of the characteristic equation)
are both in the kernel of the infinite matrix $(C_{k-\ell})_{k,\ell\in \mathbb{Z}}$, namely
\begin{equation*}
\sum_{\ell\in \mathbb{Z}}C_{k-\ell}x_a^\ell= x_a^k(\sum_{m \in \mathbb{Z}}C_{m}x_a^{-m})=0\qquad
\sum_{\ell\in \mathbb{Z}}C_{k-\ell}x_a^{-\ell}= x_a^{-k}(\sum_{m \in \mathbb{Z}}C_{m}x_a^{m})=0
\end{equation*} 
for $a=1,\cdots,p$ (recall that $C_m=C_{-m}$). To now find a vector $v_\ell$ in the kernel of the semi-infinite
matrix $(C_{k-\ell}-C_{k+\ell})_{k,\ell\geq 1}$, we note that, for $k\geq 1$
\begin{equation*}
\begin{split}
\sum_{\ell\geq 1}(C_{k-\ell}-C_{k+\ell})v_\ell& =\sum_{\ell\geq 1} C_{k-\ell} v_\ell -\sum_{\ell\leq -1} C_{k-\ell} v_{-\ell}\\
& = \sum_{\ell\in \mathbb{Z}}C_{k-\ell} v_\ell \\
\end{split}
\end{equation*}
provided $v_{-\ell}=-v_\ell$ for all $\ell$ (in particular $v_0=0$). For $a=1,\cdots,p$, the vectors
with components
\begin{equation}
v_\ell^{(a)}=x_a^\ell-x_a^{-\ell} \qquad \ell\geq 1
\label{eq:vla}
\end{equation}
are therefore in the kernel of our semi-infinite matrix. Taking now $k\leq i+1$, the sum over all $\ell\geq 1$ runs in
practice only from $1$ to $i+p+1$ so that we get a vector satisfying
\begin{equation*}
\sum_{\ell=1}^{i+1}(C_{k-\ell}-C_{k+\ell})v_\ell=0
\end{equation*}
by simply imposing $v_{i+2}=v_{i+3}=\cdots=v_{i+p+1}=0$. These extra conditions can be achieved by taking
a linear combination of the $p$ vectors $(v_\ell^{(a)})_{\ell\geq 1}$ and lead to a non-zero vector if the $p$ conditions
are not linearly independent, namely whenever  
\begin{equation*}
\det_{1\leq a,a'\leq p} v_{i+a'+1}^{(a)} =0\ .
\end{equation*}
The determinant in the left hand side of \eqref{eq:weylrep} therefore vanishes whenever the determinant 
$\det_{1\leq a,a'\leq p} v_{i+a'+1}^{(a)}$ vanishes. This latter determinant (which is anti-symmetric in the $x_a$'s instead
of the desired determinant which is symmetric) has however more zeros than desired: it indeed vanishes
whenever $x_a=x_{a'}$ for some $a\neq a'$ (as it implies $v_\ell^{(a)}=v_\ell^{(a')}$) or $x_a=1/x_{a'}$ for any $a,a'$ (as it implies $v_\ell^{(a)}=-v_\ell^{(a')}$),
and in particular (for $a=a'$) when $x_a=\pm 1$ (in which case $v_\ell^{(a)}=0$). These cases correspond
precisely to the zeros of $\det_{1\leq a,a'\leq p} v_{a'}^{(a)}=(-1)^p \prod_{a=1}^p (1-x_a^2)\prod_{1\leq a<a'\leq p}(x_a-x_{a'})(1-x_a x_{a'})
/\prod_{a=1}^p x_a^p$
and we must suppress them by dividing $\det_{1\leq a,a'\leq p} v_{i+a'+1}^{(a)}$ by $\det_{1\leq a,a'\leq p} v_{a'}^{(a)}$.
This eventually explains \eqref{eq:weylrep} by adjusting the proportionality constant so that, say  the $(x_1x_2\cdots x_p)^{i+1}$ term
on both sides be the same. Indeed, in the left hand side, this term comes from the largest possible power of $C_0$, namely
$C_0^{i+1}$, leading to a term $((-1)^{p}C_p)^{i+1}$, while in the
ratio of determinants in the right hand side, it is easily seen to be $1$.

We gave here the argument as we find it more enlightening than the proof in \cite{FH91}. Still, as presented here,
this is just an argument and promoting it into a real proof would need a better control on the
various determinants involved (in particular deal with the possibility of multiple roots, ...). 
 \vskip .5cm

To illustrate our result, let us give the expression for $h_i^{(1)}$ in the case of quadrangulations and 
hexangulations. For quadrangulations, we have $p=1$ and $C_p=-\frac{B}{t_\bullet} (B W)^{1/2}$, so that
\begin{equation}
h_i^{(1)} =W^{i+1} (B W)^{\frac{(i+1)^2}{2}} \left(\frac{B}{t_\bullet}\right)^{i+1} \frac{1}{x^{i+1}}\frac{1}{1-x^2}\,  u_{2i+4}
\  \hbox{where}\ u_i\equiv 1-x^i\ ,
\label{eq:honequad}
\end{equation}
with $B$ and $W$ solutions of \eqref{eq:recurBWquad}, and where $x$ is the solution (with modulus less than one) of  the characteristic equation  (obtained after some
straightforward simplifications)
\begin{equation}
1-2(B+W)-\sqrt{B W} \left(x+\frac{1}{x}\right)=0\ .
\label{eq:charquad}
\end{equation}
As for hexangulation ($p=2$ and $C_p=-\frac{B}{t_\bullet} (B W)$), we get
\begin{equation}
\begin{split}
h_i^{(1)}& =W^{i+1} (B W)^{\frac{(i+1)(i+2)}{2}} \left(-\frac{B}{t_\bullet}\right)^{i+1} \frac{1}{(x_1 x_2)^{i+1}}
\frac{1}{1-x_1^2}\frac{1}{1-x_2^2}\frac{1}{1-x_1x_2}\, u_{2i+4}\\
&\hbox{where}\ u_i\equiv 1-\frac{1-x_1x_2}{x_1-x_2} x_1^{i+1}-\frac{1-x_1x_2}{x_2-x_1} x_2^{i+1}- (x_1 x_2)^{i+1}\\
\end{split}
\label{eq:honehex}
\end{equation}
with $B$ and $W$ solutions of \eqref{eq:recurBWhex}, and where $x_1$ and $x_2$ are the solutions (with modulus less than one) of
\begin{equation}
1-3(B^2+W^2)-10 B W -3\sqrt{B W}(B+W)  \left(x+\frac{1}{x}\right)-B W \left(x^2+\frac{1}{x^2}\right)=0\ .
\label{eq:charhex}
\end{equation}

\subsection{Computation of $h_i^{(0)}$ and $\tilde{h}_i^{(0)}$}
\label{sec:computationHzero}
By an argument similar to the previous subsection, we immediately get
\begin{equation}
\begin{split}
&h_i^{(0)}=  (B W)^{\frac{i(i+1)}{2}} \det_{0\leq k,\ell\leq i} (\sum_{q\geq 0} \alpha_q A_{2k,2\ell}^{\bullet\bullet}(2q))\\
&\tilde{h}_i^{(0)}=  (B W)^{\frac{i(i+1)}{2}} \det_{0\leq k,\ell\leq i} (\sum_{q\geq 0} \alpha_q A_{2k,2\ell}^{\circ\circ}(2q))\\
\end{split}
\label{eq:hizero}
\end{equation}
where $A_{2k,2\ell}^{\bullet\bullet}(2q)$ denotes the generating function of paths of length $2q$ from black height 
$2k$ to black height $2\ell$ which remain above height $0$ (note that height $0$ is black in this case)
and $A_{2k,2\ell}^{\circ\circ}(2q)$ denotes the generating function of paths of length $2q$ from white height 
$2k$ to white height $2\ell$ which remain above height $0$ (note that height $0$ is white in this case).
The difficulty is now that, because of the (even) parity of the heights of the extremities of the path, we can no longer use a reflection 
principle as simple as that of the previous section. Nevertheless, we have the following formula, for $k,\ell\geq 0$:
\begin{equation}
A_{2k,2\ell}^{\bullet\bullet}(2q)=L_{k-\ell}(2q)-c\, L_{k+\ell+1}(2q)+(c^2-1) \sum_{m\geq 2} L_{k+\ell+m}(2q) (-c)^{m-2}
\label{eq:Atwoktwol}
\end{equation}
where 
\begin{equation*}
c \equiv \frac{b}{w}=\sqrt{\frac{B}{W}}
\end{equation*}
and a similar expression for $A_{2k,2\ell}^{\circ\circ}(2q)$ with $b$ and $w$ exchanged, i.e., with $c\to 1/c$.
This in turn implies
\begin{equation}
\sum_{q\geq 0} \alpha_q A_{2k,2\ell}^{\bullet\bullet}(2q)=C_{k-\ell}-c\, C_{k+\ell+1}+(c^2-1) \sum_{m\geq 2} C_{k+\ell+m} (-c)^{m-2}\ .
\label{eq:Csum}
\end{equation}
Note that the sum in the right hand side is in practice finite.
\begin{figure}
\begin{center}
\includegraphics[width=13cm]{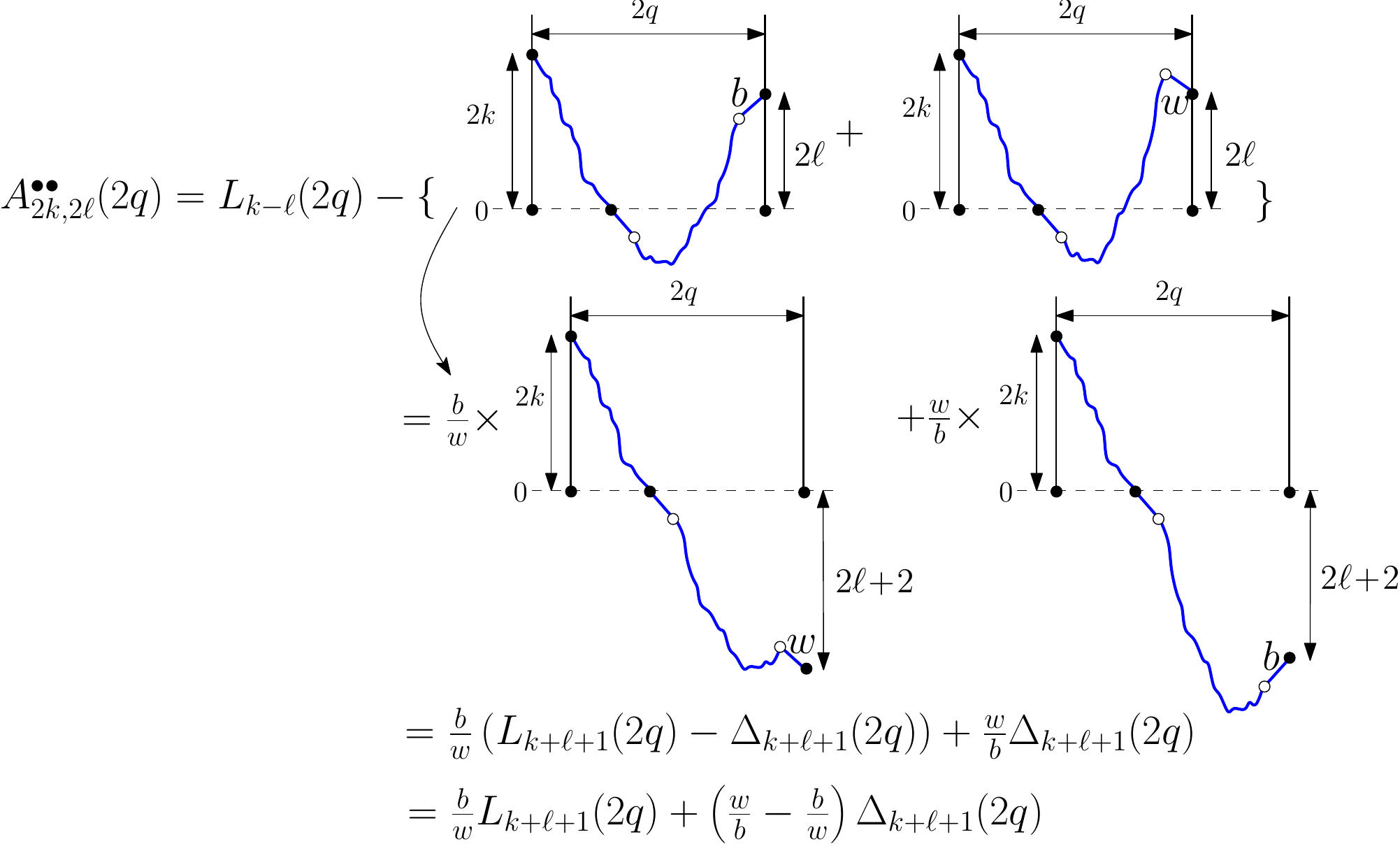}
\end{center}
\caption{We get $A_{2k,2\ell}^{\bullet\bullet}(2q)$ by subtracting from $L_{k-\ell}(2q)$
paths which go below $0$, which are decomposed in two sets: those ending with an up step (weighted $b$)
and those ending with a down step (weighted $w$). By returning  \emph{vertically} the rightmost part of the paths \emph{but
the last step} which we return \emph{horizontally}, we get paths with height decrease $2(k+\ell+1)$. 
To get the correct weights, we must apply a multiplicative factor $b/w$ to the first set (enumerated by 
$L_{k+\ell+1}(2q)-\Delta_{k+\ell+1}(2q)$) and a multiplicative factor $w/b$ to the second set (enumerated by 
$\Delta_{k+\ell+1}(2q)$), hence the formula.}
\label{fig:reflectiontwo}
\end{figure}
\begin{figure}
\begin{center}
\includegraphics[width=12cm]{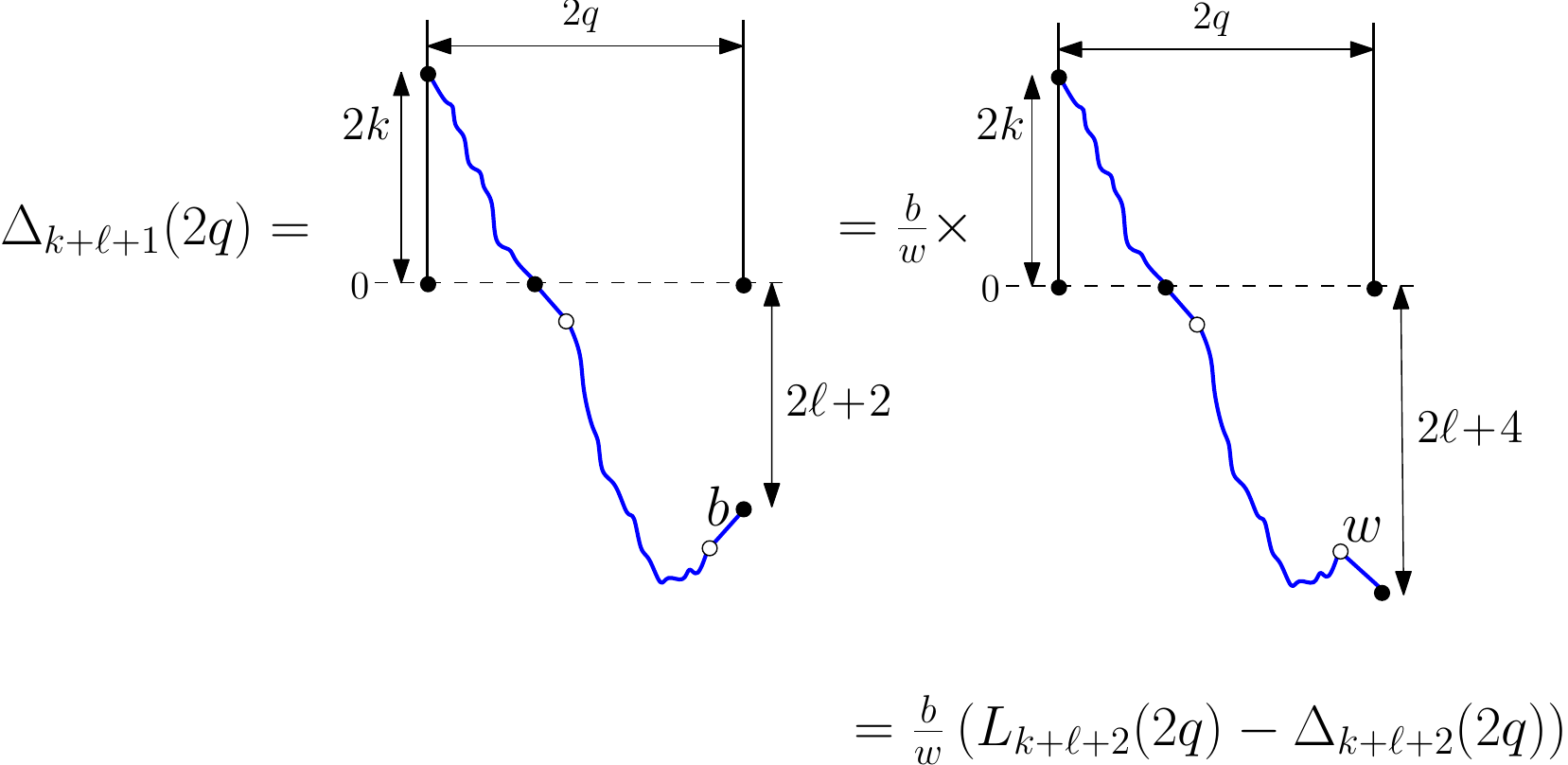}
\end{center}
\caption{By simply returning the last step horizontally, we see that the generating function $\Delta_{k+\ell+1}(2q)$ for paths
of length $2q$, height decrease $2(k+\ell+1)$ and ending with an up step is $(b/w)$ times the 
generating function $L_{k+\ell+2}(2q)-\Delta_{k+\ell+2}(2q)$ for paths
of length $2q$, height decrease $2(k+\ell+2)$ and ending with a down step. By repeating this argument recursively, we
find that $\Delta_{k+\ell+1}(2q)=-\sum_{m\geq 2} \left(-\frac{b}{w}\right)^{m-1}L_{k+\ell+m}(2q)$.}
\label{fig:reflectionthree}
\end{figure}

Let us now explain the formula \eqref{eq:Atwoktwol}. We may assume $q>0$ since, for $q=0$, the formula is obvious (only
the first term in the right hand side contributes). To get $A_{2k,2\ell}^{\bullet\bullet}(2q)$, we wish as before to subtract from $L_{k-\ell}(2q)$
the generating function of those paths which go below height $0$. To apply again some reflection principle, we look at the
first passage below $0$, which is a step from a black height $0$ to a white height $-1$ (hence receives a weight $b$).
The rest of the path is formed of an ``intermediate" part going from white height $-1$ to the last encountered white height 
(equal to $2\ell+1$ or $2\ell-1$) and of a final step reaching black height $2\ell$, see Figure \ref{fig:reflectiontwo}.  If we now return 
\emph{vertically} the intermediate
part and \emph{horizontally} the last step, we get a total path of length $2q$ going from height $2k$ to height $-2-2\ell$, hence 
with height decrease $2(k+\ell+1)$. The weights of all steps after reversing are correct (recall that a vertical reversing conserves the weights)
except for that of the last step which is $b$ (before reversing) instead of $w$ (as required after reversing) if the last step
after reversing is a descending step from white height $-1-2\ell$ to black height $-2-2\ell$ or is $w$ (before reversing) instead of $b$ (after reversing) 
if the last step after reversing is an ascending step from white height $-3-2\ell$ to black height $-2-2\ell$, see Figure \ref{fig:reflectionthree}.
 Both kinds of paths are enumerated
by $L_{k+\ell+1}(2q)$ but for the first kind, we must apply a multiplicative correction $b/w$ while for the second kind,
we must apply a multiplicative correction $w/b$. The quantity $L_{k-\ell}(2q)-(b/w) L_{k+\ell+1}(2q)$ therefore performs the correct subtraction
of paths of the first kind but paths of the second kind must be re-added with a multiplicative factor $(b/w-w/b)$ to
obtain a correct result. If we denote by $\Delta_{k+\ell+1}(2q)$ the generating function of those paths, the correct 
formula is therefore $A_{2k,2\ell}^{\bullet\bullet}(2q)=L_{k-\ell}(2q)-(b/w)\, L_{k+\ell+1}(2q)+(b/w-w/b) \Delta_{k+\ell+1}(2q)$.
Now, by returning the last (ascending) step in a path enumerated by $\Delta_{k+\ell+1}(2q)$, we get a path whose last step
is now descending from white height $-3-2\ell$ to black height $-4-2\ell$, hence a path with height decrease $2(k+\ell+2)$ 
and, in our terminology, being of the first kind. This immediately leads to the
relation $\Delta_{k+\ell+1}(2q)=\frac{b}{w}(L_{k+\ell+2}(2q)-\Delta_{k+\ell+2}(2q))$ and, by repeating the argument recursively to
$\Delta_{k+\ell+1}(2q)=-\sum_{m\geq 2} \left(-\frac{b}{w}\right)^{m-1}L_{k+\ell+m}(2q)$. Setting $c=b/w$ yields eventually the
desired formula \eqref{eq:Atwoktwol}. To conclude, let us mention that we have a formula for $A_{2k,2\ell}^{\circ\circ}(2q)$
similar to \eqref{eq:Atwoktwol} with $c$ changed into $1/c$.

With the above formula \eqref{eq:Csum}, the computation of the determinants in  eqs.~\eqref{eq:hizero} is much more involved and
we detail it in Appendix B. Still the result is remarkably simple as we get eventually
\begin{equation}
\begin{split}
h_i^{(0)}&=(B W)^{\frac{i(i+1)}{2}}(-1)^{p(i+1)} C_p^{i+1} \prod_{a=1}^p (1+ c\, x_a) \frac{\det\limits_{1\leq a,a'\leq p} (\gamma_a x_a^{i+a'}-x_a^{-(i+1+a')})}
{\det\limits_{1\leq a,a'\leq p} (x_a^{a'}-x_a^{-a'})})\\
&\qquad \hbox{where}\ \gamma_a=\frac{c+x_a}{1+c x_a}\\
\end{split}
\label{eq:dethizeroexp}
\end{equation}
while
\begin{equation}
\tilde{h}_i^{(0)}=(B W)^{\frac{i(i+1)}{2}}(-1)^{p(i+1)} \tilde{C}_p^{i+1} \prod_{a=1}^p (1+ x_a/c) \frac{\det\limits_{1\leq a,a'\leq p} (x_a^{i+a'}/\gamma_a-x_a^{-(i+1+a')})}
{\det\limits_{1\leq a,a'\leq p} (x_a^{a'}-x_a^{-a'})})
\end{equation}
(note the change $c\to 1/c$, which in turn  implies the change $\gamma_a\to 1/\gamma_a$). Note also that both $h_i^{(0)}$ and
$\tilde{h}_i^{(0)}$ are actually invariant under $x_a\leftrightarrow 1/x_a$ for any $a$ since $(1+c\, x_a) (\gamma_a x_a^{i+a'}-x_a^{-(i+1+a')})=
(x_a^{i+a'+1}-x_a^{-(i+a'+1)})+c (x_a^{i+a'}-x_a^{-(i+a')})$.

Again, besides the actual proof of Appendix B, we can give a more heuristic argument along the same lines as before.
Writing
\begin{equation*}
\begin{split}
\hspace{-1.3cm}\sum_{\ell\geq 0} \left(\sum_{q\geq 0} A_{2k,2\ell}^{\bullet\bullet}(2q)\right) w_{\ell}&=\sum_{\ell\geq 0}\big(C_{k-\ell}-c\, C_{k+\ell+1}+(c^2-1) \sum_{m\geq 2} C_{k+\ell+m} (-c)^{m-2}\big)w_\ell\\
&=\sum_{\ell\geq 0}C_{k-\ell}w_\ell-c\sum_{\ell\leq -1} C_{k-\ell}w_{-\ell-1}\\
& \hspace{3.cm} +(c^2-1) \sum_{m\geq 2} (-c)^{m-2} \sum_{\ell\leq -m}C_{k-\ell} w_{-\ell-m}\\
\end{split}
\end{equation*}
we see that, for $\ell\leq -1$, the net coefficient in front of $C_{k-\ell}$ is
\begin{equation*}
-c\, w_{-\ell-1}+(c^2-1) \sum_{m=2}^{-\ell} (-c)^{m-2} w_{-\ell-m}
\end{equation*}
while we would have liked it to be $w_\ell$ so as to reproduce $\sum_{\ell\in \mathbb{Z}} C_{k-\ell}w_\ell$ which is known to give $0$ for 
$w_\ell=x_a^\ell$ or $w_\ell=x_a^{-\ell}$. To construct a vector in the kernel of $(\sum_{q\geq 0} A_{2k,2\ell}^{\bullet\bullet}(2q))_{k,\ell\geq 0}$,  
we thus may as before take a linear combination of $x_a^\ell$ or $x_a^{-\ell}$, now satisfying for all $\ell\leq -1$ the condition
\begin{equation*}
w_{\ell}=-c\, w_{\ell-1}+(c^2-1) \sum_{m=2}^{-\ell} (-c)^{m-2} w_{-\ell-m}
\end{equation*}
(the sum being empty if $\ell=-1$). Writing $(c^2-1)\sum_{m=2}^{-\ell} (-c)^{m-2} w_{\ell-m}=(c^2-1) w_{-\ell-2} -c (c^2-1)\sum_{m=2}^{-\ell-1} (-c)^{m-2} w_{-\ell-1-m}$
and using the above condition for $\ell+1$, we obtain that this condition is equivalent recursively (over $|\ell|=-\ell$) to
$w_{\ell}=-c\, w_{\ell-1}+(c^2-1) w_{-\ell-2} -c\big(w_{\ell+1}+c\, w_{-\ell-2})\big)$, namely
\begin{equation*}
(w_{\ell}+w_{-\ell-2}) +c(w_{\ell+1}+w_{-\ell-1})=0
\end{equation*}
for all $\ell\leq -1$ (i.e., for all $\ell$ since it is symmetric under $\ell\to -\ell-2$). This leads immediately to the linear combination
\begin{equation*}
w_\ell^{(a)}=\frac{c+x_a}{1+c\, x_a}x_a^\ell-x_a^{-\ell-1} \qquad \ell\geq 0
\end{equation*}
for $a=1,\cdots,p$, which satisfy
\begin{equation*}
\sum_{\ell\geq 0} \left(\sum_{q\geq 0} A_{2k,2\ell}^{\bullet\bullet}(2q)\right) w_\ell^{(a)}=0
\end{equation*}
for all $k\geq 0$. Restricting us now to $k\leq i$, the sum over all $\ell\geq 0$ runs in
practice only from $0$ to $i+p$ so that we get a vector satisfying
\begin{equation*}
\sum_{\ell=0}^{i} \left(\sum_{q\geq 0} A_{2k,2\ell}^{\bullet\bullet}(2q)\right)w_\ell=0
\end{equation*}
by simply imposing $w_{i+1}=w_{i+2}=\cdots=w_{i+p}=0$. As before, these extra conditions are achieved by taking
a linear combination of the $p$ vectors $(w_\ell^{(a)})_{\ell\geq 0}$ and a non-zero vector is found if the $p$ conditions
are not linearly independent, namely if  
\begin{equation*}
\det_{1\leq a,a'\leq p} w_{i+a'}^{(a)} =0\ .
\end{equation*}
The determinant in the right hand side of the first line in \eqref{eq:hizero} therefore vanishes whenever the determinant 
$\det_{1\leq a,a'\leq p} w_{i+a'}^{(a)}$ vanishes. As before, this latter determinant (which is anti-symmetric in the $x_a$'s instead
of the desired determinant which is symmetric) has however more zeros than desired as it vanishes
again whenever $x_a=x_{a'}$ for some $a\neq a'$ (as it implies $w_\ell^{(a)}=w_\ell^{(a')}$) or $x_a=1/x_{a'}$ for any $a,a'$ (as it implies 
$w_\ell^{(a)}=-x_{a'} \frac{1+c x_{a'}}{c+x_{a'}} w_\ell^{(a')}$),
and in particular (for $a=a'$) when $x_a=\pm 1$ (in which case $w_\ell^{(a)}=0$).
Again we must suppress these spurious zeros by dividing $\det_{1\leq a,a'\leq p} w_{i+a'}^{(a)}$ by \emph{the same determinant as 
before}, namely $\det_{1\leq a,a'\leq p} v_{a'}^{(a)}$
with $v_\ell^{(a)}$ as in \eqref{eq:vla}.
This eventually explains \eqref{eq:dethizeroexp} by adjusting the proportionality constant so that  the $(x_1x_2\cdots x_p)^{i+1}$ term
on both sides be the same. Indeed, up to the trivial factor $(BW)^{i+1}$, this term in $h_i^{(0)}$ comes as before from the 
$C_0^{i+1}$ term in the determinant in \eqref{eq:hizero},  hence equals $((-1)^{p}C_p)^{i+1}$, while in the ratio of determinants 
in the right hand side, it is easily seen to be $1$ after factoring out a term $1/\prod_{a=1}^{p}(1+c\, x_a)$ (coming from
the denominators of the $\gamma_a$'s).

To conclude this section, let us give the expression for $h_i^{(0)}$ in the case of quadrangulations and 
hexangulations. For quadrangulations, we get
\begin{equation}
h_i^{(0)} = (B W)^{\frac{(i+1)^2}{2}} \left(\frac{B}{t_\bullet}\right)^{i+1} \frac{1}{x^{i+1}}\frac{1+c\, x}{1-x^2}\,  \bar{u}_{2i+3}
\  \hbox{where}\ \bar{u}_i\equiv 1-\frac{c+x}{1+c\, x} x^i\ ,
\label{eq:hzeroquad}
\end{equation}
with $B$ and $W$ solutions of \eqref{eq:recurBWquad}, $c=\sqrt{B/W}$ and $x$ is solution of  the associated characteristic equation
\eqref{eq:charquad}.
For hexangulations , we get
\begin{equation}
\begin{split}
\hspace{-1.3cm} h_i^{(0)}& =(B W)^{\frac{(i+1)(i+2)}{2}} \left(-\frac{B}{t_\bullet}\right)^{i+1} \frac{1}{(x_1 x_2)^{i+1}}
\frac{1+c\, x_1}{1-x_1^2}\frac{1+c\, x_2}{1-x_2^2}\frac{1}{1-x_1x_2}\, \bar{u}_{2i+3}\\
\hspace{-1.3cm} \hbox{where}\ \bar{u}_i&\equiv 1-\frac{1-x_1x_2}{x_1-x_2} \frac{c+x_1}{1+c\, x_1}x_1^{i+1}-\frac{1-x_1x_2}{x_2-x_1} \frac{c+x_2}{1+c\, x_2}x_2^{i+1}- \frac{c+x_1}{1+c\, x_1}\frac{c+x_2}{1+c\, x_2}(x_1 x_2)^{i+1}\\
\end{split}
\label{eq:hzerohex}
\end{equation}
with $B$ and $W$ solutions of \eqref{eq:recurBWhex}, $c=\sqrt{B/W}$ and  $x_1$ and $x_2$ solutions of
the associated characteristic equation \eqref{eq:charhex}. Similar expressions for $\tilde{h}_i^{(0)}$ are obtained by changing $B$ into $W$, 
$t_\bullet$ into $t_\circ$ and $c$ into $1/c$.

\section{Final result}
\label{sec:finalresult}

We may now plug our expression \eqref{eq:dethizeroexp} for $h_i^{(0)}$ and \eqref{eq:dethiexp} for $h_i^{(1)}$ in 
the general formula \eqref{eq:BWtoh} to get our main results
\begin{equation*}
\begin{split}
&B_{2i}=B \frac{\det\limits_{1\leq a,a' \leq p}\left(x_a^{i+a'-1}-x_a^{-(i+a'-1)}\right) \det\limits_{1\leq a,a' \leq p}\left(\gamma_a x_a^{i+a'}-x_a^{-(i+a'+1)}\right)}
{\det\limits_{1\leq a,a' \leq p}\left(\gamma_ax_a^{i+a'-1}-x_a^{-(i+a')}\right) \det\limits_{1\leq a,a' \leq p}\left(x_a^{i+a'}-x_a^{-(i+a')}\right)} \\
&W_{2i+1}=W \frac{\det\limits_{1\leq a,a' \leq p}\left(\gamma_ax_a^{i+a'-1}-x_a^{-(i+a')}\right) \det\limits_{1\leq a,a' \leq p}\left(x_a^{i+a'+1}-x_a^{-(i+a'+1)}\right)}
{\det\limits_{1\leq a,a' \leq p}\left(x_a^{i+a'}-x_a^{-(i+a')}\right) \det\limits_{1\leq a,a' \leq p}\left(\gamma_ax_a^{i+a'}-x_a^{-(i+a'+1)}\right)} \\
& \hbox{where}\ \gamma_a=\frac{c+x_a}{1+c\, x_a}\ .\\
\end{split}
\end{equation*}
Note in particular that $B_0=0$, as wanted. The other parity is obtained by symmetry, and reads
\begin{equation*}
\begin{split}
&B_{2i+1}=B \frac{\det\limits_{1\leq a,a' \leq p}\left(x_a^{i+a'-1}/\gamma_a-x_a^{-(i+a')}\right) \det\limits_{1\leq a,a' \leq p}\left(x_a^{i+a'+1}-x_a^{-(i+a'+1)}\right)}
{\det\limits_{1\leq a,a' \leq p}\left(x_a^{i+a'}-x_a^{-(i+a')}\right) \det\limits_{1\leq a,a' \leq p}\left(x_a^{i+a'}/\gamma_a-x_a^{-(i+a'+1)}\right)} \\
&W_{2i}=W \frac{\det\limits_{1\leq a,a' \leq p}\left(x_a^{i+a'-1}-x_a^{-(i+a'-1)}\right) \det\limits_{1\leq a,a' \leq p}\left(x_a^{i+a'}/\gamma_a-x_a^{-(i+a'+1)}\right)}
{\det\limits_{1\leq a,a' \leq p}\left(x_a^{i+a'-1}/\gamma_a-x_a^{-(i+a')}\right) \det\limits_{1\leq a,a' \leq p}\left(x_a^{i+a'}-x_a^{-(i+a')}\right)} \\
& \hbox{where}\ \gamma_a=\frac{c+x_a}{1+c\, x_a}\ .\\
\end{split}
\end{equation*}
Note in particular that $W_0=0$. (Note also that, when forgetting the vertex colors, i.e., setting $\tb=\tw=t$, one has $B_i=W_i$ and 
$B=W$, so that $c=1$ and $\gamma_a=1$
for all $1\leq a\leq p$; it is then easy to see that the above expressions of 
$B_i$ and $W_i$ specialize to the determinant expressions in~\cite{BG12} for uncolored bipartite maps.)    

The distance-dependent two-point generating functions $G_{i}^\bullet$ and $G_{i}^\circ$ are then obtained from
eqs.~\eqref{eq:teopointbullet} and \eqref{eq:teopointcirc}, namely
\begin{equation*}
\begin{split}
& G_{2i+2}^\bullet=t_\bullet (B_{2i+2}-B_{2i+1})\ ,\quad G_{2i+1}^\bullet=t_\circ (B_{2i+1}-B_{2i}-t_\bullet \delta_{i,0})\\
& G_{2i+2}^\circ=t_\circ(W_{2i+2}-W_{2i+1})\ , \quad W_{2i+1}^\circ=t_\bullet (W_{2i+1}-W_{2i}-t_\circ \delta_{i,0})\\
\end{split}
\end{equation*}
for $i\geq 0$.

For both quadrangulations and hexangulations, we get
\begin{equation}
\begin{split}
& B_{2i}=B\frac{u_{2i}\bar{u}_{2i+3}}{\bar{u}_{2i+1}u_{2i+2}}\qquad W_{2i+1}=W\frac{\bar{u}_{2i+1}u_{2i+4}}{u_{2i+2}\bar{u}_{2i+3}}\\
& B_{2i+1}=B\frac{\hat{u}_{2i+1}u_{2i+4}}{u_{2i+2}\hat{u}_{2i+3}}\qquad W_{2i}=W\frac{u_{2i}\hat{u}_{2i+3}}{\hat{u}_{2i+1}u_{2i+2}}\\
\end{split}
\label{eq:solquadexpl}
\end{equation}
with $u_i$ and $\bar{u}_i$ as in \eqref{eq:honequad} and \eqref{eq:hzeroquad} for quadrangulations and as 
in \eqref{eq:honehex} and \eqref{eq:hzerohex} for hexangulations, while $\hat{u}_i$ is obtained from 
$\bar{u}_i$ by simply changing $c$ into $1/c$.

It is interesting to expand our various generating functions into powers of $t_\bullet$ and $t_\circ$ so as to get \emph{numbers} of maps
instead of generating functions. For quadrangulations, this is best done upon introducing the two quantities $d\equiv c\, x$
and $y\equiv x^2$. From the characteristic equation \eqref{eq:charquad}, $d$ is solution of
\begin{equation*}
W d^2+ (2(B + W)-1)d  + B =0
\end{equation*}
which yields its power expansion from those of $B$ and $W$, namely
\begin{equation*}
d=\tb+(3\tb^2+4\tb\tw)+(10\tb^3+33\tb^2\tw+16\tb\tw^2)+(35\tb^4+202\tb^3\tw+243\tb^2\tw^2+64\tb\tw^3)+\cdots
\end{equation*}
while that of $y$ follows via the relation $y=d^2 W/B$, namely
\begin{equation*}
 y=\tb\tw+(7\tb^2\tw+7\tb\tw^2)+(38\tb^3\tw+91\tb^2\tw^2+38\tb\tw^3)+\cdots
\end{equation*}
Now we have the expressions
\begin{equation*}
\begin{split}
&B_{2i}=B\frac{(1-y^{i})(1-\beta y^{i+1})}{(1-y^{i+1})(1-\beta y^{i})},\ \ \
W_{2i}=W\frac{(1-y^{i})(1-\beta^{-1} y^{i+2})}{(1-y^{i+1})(1-\beta^{-1} y^{i+1})},\\
&B_{2i+1}=B\frac{(1-y^{i+2})(1-\beta^{-1}  y^{i+1})}{(1-y^{i+1})(1-\beta^{-1}  y^{i+2})},\ \ \
W_{2i+1}=W\frac{(1-y^{i+2})(1-\beta y^{i})}{(1-y^{i+1})(1-\beta y^{i+1})},\\
&\hspace{2.cm} \hbox{where}\ \beta=\frac{d+y}{1+d} \\
\end{split}
\end{equation*}
which are well suited for series expansions. For instance, we get
\begin{equation*}
\begin{split}
&\hspace{-1.3cm}B_1 =\tb+\tb(\tb+\tw)+\tb(2\tb^2+5\tb\tw+2\tw^2)+\tb(5\tb^3+22\tb^2\tw+22\tb\tw^2+5\tw^3)+\cdots \\
&\hspace{-1.3cm}B_2=\tb+\tb(\tb+2\tw)+\tb(2\tb^2+9\tb\tw+6\tw^2)+\tb(5\tb^3+37\tb^2\tw+57\tb\tw^2+20\tw^3)+\cdots \\
&\hspace{-1.3cm}B_3=\tb+\tb(\tb+2\tw)+\tb(2\tb^2+10\tb\tw+6\tw^2)+\tb(5\tb^3+44\tb^2\tw+65\tb\tw^2+20\tw^3)+\cdots \\
\end{split}
\end{equation*}
so that
\begin{equation*}
\begin{split}
&\hspace{-1.3cm}G_1^\bullet=\tb\tw(\tb+\tw)+\tb\tw(2\tb^2+5\tb\tw+2\tw^2)+\tb\tw(5\tb^3+22\tb^2\tw+22\tb\tw^2+5\tw^3)+\cdots \\
&\hspace{-1.3cm}G_2^\bullet=t_\bullet^2 t_\circ + 4 t_\bullet^2 t_\circ(t_\bullet+t_\circ)+5t_\bullet^2 t_\circ(3 t_\bullet^2+7 t_\bullet t_\circ+3 t_\circ^2)+\cdots \\
&\hspace{-1.3cm}G_3^\bullet=\tb^2\tw^2+\tb^2\tw^2(7\tb+8\tw)+\cdots \\ 
\end{split}
\end{equation*}

As for hexangulations, we may proceed in a slightly more involved (although quite similar) way by setting 
$z\equiv c(x+1/x)$ which, from eq.~\eqref{eq:charhex} is solution of
\begin{equation*}
W^2 z^2 + 3 W(B+W) z +8 B W +3 (B^2+W^2)-1=0\ .
\end{equation*}
This leads to two solutions $z_1$ and $z_2$
\begin{equation*}
\hspace{-1.3cm}z_1=\frac{-3B\!-\!3W\!-\!\sqrt{4\!-\!3B^2\!-\!14W B\!-\!3W^2}}{2 W}\ , \qquad z_2=\frac{-3B\!-\!3W\!+\!\sqrt{4\!-\!3B^2\!-\!14W B\!-\!3 W^2}}{2 W}
\end{equation*}
which implicitly define the two values $x_1$ and $x_2$ to be incorporated in eqs.~\eqref{eq:honehex} and \eqref{eq:hzerohex}.
As for quadrangulations, we then define $d_1\equiv c x_1$ and $d_2\equiv c x_2$, solutions of $d_i^2-d_i z_i+B/W=0$ ($i=1,2$) as well as 
$y_1\equiv x_1^2=d_1^2 W/B$ and $y_2\equiv x_2^2=d_2^2 W/B$. Picking the correct determination for $d_1$ and $d_2$, we get
their power series expansions from those of $B$ and $W$, namely
\begin{equation*}
\begin{split}
&\hspace{-1.3cm}d_1=-\tb\!+\!\frac{3}{2}\tb(\tb\!+\!\tw)\!-\!\frac{1}{8} \tb\left(29
   \tb^2\!+\!106\tw\tb\!+\!45\tw^2\right)\!+\!\frac{3}{2}\tb \left(5
   \tb^3\!+\!30\tw\tb^2\!+\!32\tw^2\tb\!+\!7\tw^3\right)\!+\!\cdots \\
&\hspace{-1.3cm}d_2=\tb\!+\!\frac{3}{2}\tb(\tb\!+\!\tw)\!+\!\frac{1}{8}\tb\left(29\tb^2\!+\!106
   \tw \tb\!+\!45 \tw^2\right)\!+\!\frac{3}{2} \tb \left(5 \tb^3\!+\!30 \tw
   \tb^2\!+\!32 \tw^2 \tb\!+\!7 \tw^3\right)\!+\!\cdots \\
&\hspace{-1.3cm}y_1=\tb\tw\!-\!3\tb\tw(\tb\!+\!\tw)\!+\!\frac{1}{2}\tb\tw
   \left(23\tb^2\!+\!62\tw\tb\!+\!23\tw^2\right)\!+\!\cdots\\
&\hspace{-1.3cm}y_2=\tb\tw\!+\!3\tb\tw(\tb\!+\!\tw)\!+\!\frac{1}{2}\tb\tw
   \left(23\tb^2\!+\!62\tw\tb\!+\!23\tw^2\right)\!+\!\cdots\\
\end{split}
\end{equation*} 
Now we have the expressions
\begin{equation*}
\begin{split}
&\hspace{-1.3cm}B_{2i}=B\frac{(1\!-\!\lambda_1 y_1^i\!-\!\lambda_2 y_2^i
\!-\!\frac{W}{B}d_1d_2(y_1y_2)^i)}{(1\!-\!\lambda_1 y_1^{i+1}\!-\!\lambda_2 y_2^{i+1}
\!-\!\frac{W}{B}d_1d_2(y_1y_2)^{i+1})}\times \\
&\hspace{2.cm}\times\frac{(1\!-\!\lambda_1\beta_1 y_1^{i+1}\!-\!\lambda_2\beta_2 y_2^{i+1}
\!-\!\frac{W}{B}d_1d_2\beta_1\beta_2 (y_1y_2)^{i+1})}{(1\!-\!\lambda_1\beta_1 y_1^i\!-\!\lambda_2\beta_2 y_2^i
\!-\!\frac{W}{B}d_1d_2\beta_1\beta_2(y_1y_2)^i)}\\
&\hspace{-1.3cm}W_{2i}=W\frac{(1\!-\!\lambda_1 y_1^i\!-\!\lambda_2 y_2^i
\!-\!\frac{W}{B}d_1d_2(y_1y_2)^i)}{(1\!-\!\lambda_1 y_1^{i+1}\!-\!\lambda_2 y_2^{i+1}
\!-\!\frac{W}{B}d_1d_2(y_1y_2)^{i+1})}\times \\
&\hspace{2.cm}\times\frac{(1\!-\!\lambda_1\beta_1^{-1}y_1^{i+2}\!-\!\lambda_2 \beta_2^{-1}y_2^{i+2}
\!-\!\frac{W}{B}d_1d_2\beta_1^{-1} \beta_2^{-1}(y_1y_2)^{i+2})}{(1\!-\!\lambda_1\beta_1^{-1}y_1^{i+1}\!-\!\lambda_2 \beta_2^{-1}y_2^{i+1}
\!-\!\frac{W}{B}d_1d_2\beta_1^{-1} \beta_2^{-1}(y_1y_2)^{i+1})}\\
\end{split}
\end{equation*}
\begin{equation*}
\begin{split}
&\hspace{-1.3cm}B_{2i+1}=B\frac{(1\!-\!\lambda_1 y_1^{i+2}\!-\!\lambda_2 y_2^{i+2}
\!-\!\frac{W}{B}d_1d_2(y_1y_2)^{i+2})}{(1\!-\!\lambda_1 y_1^{i+1}\!-\!\lambda_2 y_2^{i+1}
\!-\!\frac{W}{B}d_1d_2(y_1y_2)^{i+1})}\times \\
&\hspace{2.cm}\times\frac{(1\!-\!\lambda_1\beta_1^{-1}y_1^{i+1}\!-\!\lambda_2  \beta_2^{-1}y_2^{i+1}
\!-\!\frac{W}{B}d_1d_2\beta_1^{-1} \beta_2^{-1}(y_1y_2)^{i+1})}{(1\!-\!\lambda_1\beta_1^{-1}y_1^{i+2}\!-\!\lambda_2  \beta_2^{-1}y_2^{i+2}
\!-\!\frac{W}{B}d_1d_2\beta_1^{-1} \beta_2^{-1}(y_1y_2)^{i+2})}\\
&\hspace{-1.3cm}W_{2i+1}=W\frac{(1\!-\!\lambda_1y_1^{i+2}\!-\!\lambda_2 y_2^{i+2}
\!-\!\frac{W}{B}d_1d_2(y_1y_2)^{i+2})}{(1\!-\!\lambda_1 y_1^{i+1}\!-\!\lambda_2 y_2^{i+1}
\!-\!\frac{W}{B}d_1d_2(y_1y_2)^{i+1})}\times \\
&\hspace{2.cm}\times\frac{(1\!-\!\lambda_1\beta_1 y_1^{i}\!-\!\lambda_2 \beta_2 y_2^{i}
\!-\!\frac{W}{B}d_1d_2\beta_1\beta_2 (y_1y_2)^{i})}{(1\!-\!\lambda_1\beta_1 y_1^{i+1}\!-\!\lambda_2\beta_2 y_2^{i+1}
\!-\!\frac{W}{B}d_1d_2\beta_1 \beta_2 (y_1y_2)^{i+1})}\\
&\hspace{-1.3cm}\hbox{where}\ \lambda_1= \frac{d_1-y_1d_2}{d_1-d_2}, \quad \lambda_2= \frac{d_2-y_2d_1}{d_2-d_1}, \quad \beta_1=\frac{d_1+y_1}{1+d_1},
\quad \beta_2=\frac{d_2+y_2}{1+d_2},
\end{split}
\end{equation*}
which are well suited for series expansions. For instance, we get
\begin{equation*}
\begin{split}
&\hspace{-1.3cm}B_1=\tb + \tb(\tb^2 + 3 \tb 
        \tw + \tw^2) + \tb(3 \tb^4 + 24 \tb^3 \tw + 46 \tb^2 \tw^2 + 24 \tb \tw^3 + 3 \tw^4) + \cdots \\
&\hspace{-1.3cm}B_2 = \tb + \tb(\tb^2 + 5 \tb \tw + 3 \tw^2) + \tb(3 \tb^4 + 36 \tb^3 
              \tw + 99 \tb^2 \tw^2 + 77 \tb \tw^3 + 15 \tw^4)+\cdots \\
&\hspace{-1.3cm}B_3 = \tb + \tb (\tb^2 + 6 \tb \tw + 3 \tw^2) + \tb(3 \tb^4 + 48 \tb^3 \tw + 132 \tb^2 \tw^2 + 91 \tb \tw^3 + 15 \tw^4)+\cdots \\
\end{split}
\end{equation*}
so that
\begin{equation*}
\begin{split}
&\hspace{-1.3cm}G_1^\bullet=(\tb^3\tw + 3\tb^2\tw^2 + \tb\tw^3) + (3\tb^5\tw + 24\tb^4\tw^2 + 46\tb^3\tw^3 + 24\tb^2\tw^4 + 3\tb\tw^5)+\cdots \\
&\hspace{-1.3cm}G_2^\bullet=(2\tb^3\tw + 2\tb^2\tw^2) + (12\tb^5\tw + 53\tb^4\tw^2 + 53\tb^3\tw^3 + 12\tb^2\tw^4)+\cdots \\
&\hspace{-1.3cm}G_3^\bullet=\tb^2\tw^2+ (12\tb^4\tw^2 + 33\tb^3\tw^3 + 14\tb^2\tw^4)+\cdots \\ 
\end{split}
\end{equation*}
 
\section{Another approach via hard dimers}
\label{sec:harddimers}
As a check of our results, it is a nice exercise to recover some of our formulas from a completely different
approach relating the Hankel determinants to generating functions of hard dimers on bicolored segments.
Such an approach was already used in \cite{BG12} to compute the two-point function of quadrangulations
and we will repeat the same arguments in our slightly more involved situation where we keep track of the
black and white vertex weights. Interestingly enough, this approach may also be extended to the case 
of hexangulations, as we shall discuss below.
 
\subsection{The case of quadrangulations}
\label{sec:harddimersquad}
\begin{figure}
\begin{center}
\includegraphics[width=8cm]{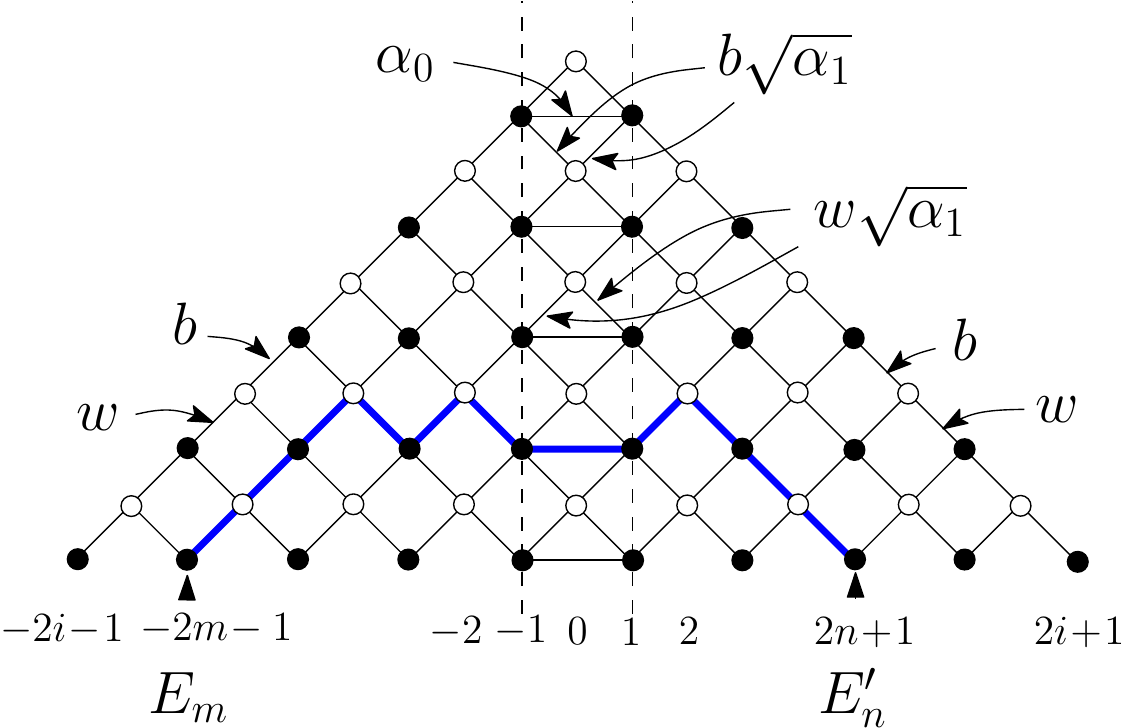}
\end{center}
\caption{A graph designed so that the generating function for directed (from left to right) paths 
from $E_m$  (with coordinates $(-2m-1,0)$) to $E_n'$ (with coordinates $(2n+1,0)$) precisely 
reproduces $F_{n+m}$ for quadrangulations.}
\label{fig:graphone}
\end{figure}
\begin{figure}
\begin{center}
\includegraphics[width=12cm]{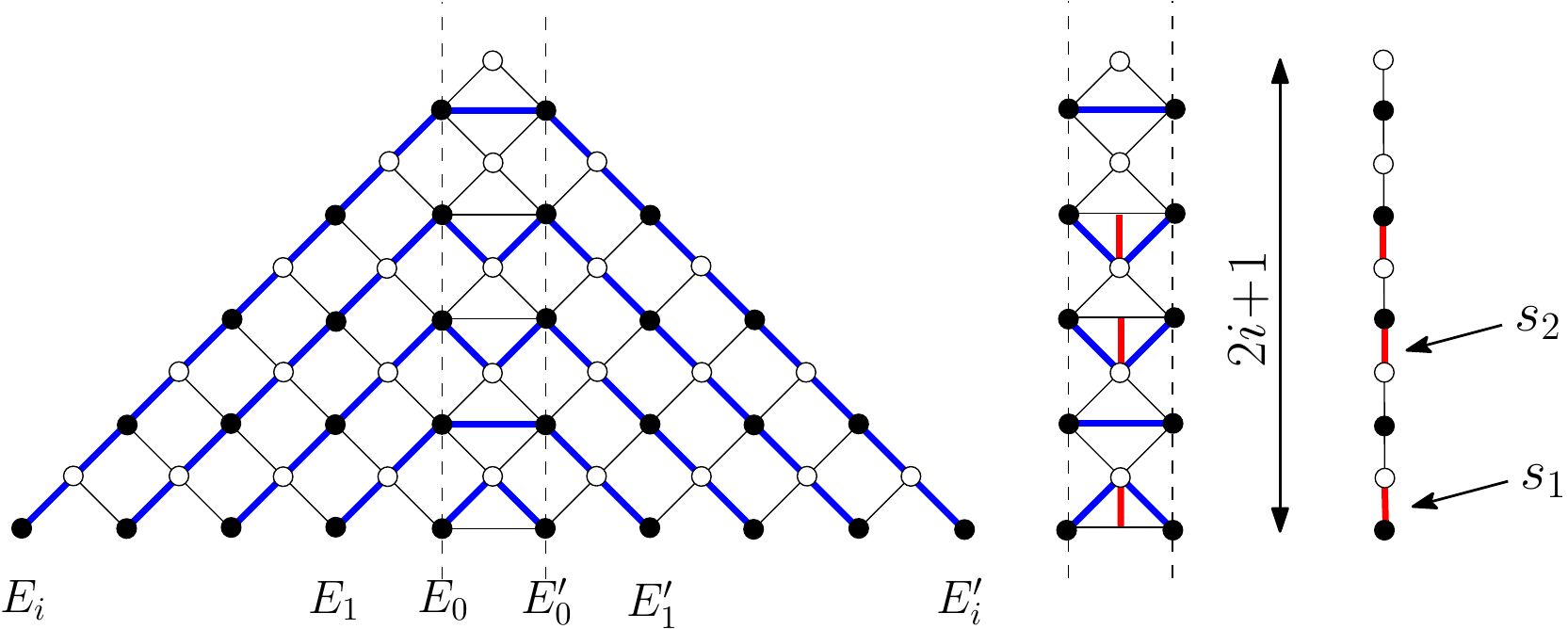}
\end{center}
\caption{Mutually avoiding paths on the graph of Figure \ref{fig:graphone}, from
$E_0,E_1, \dots E_i$ to $E_0', E_1',\cdots E_i'$. The only freedom 
comes from the central column (between abscissas $-1$ and $1$). The path configurations 
in this column are in one-to-one correspondence with hard dimers on a bicolored segment 
of length $2i+1$.}
\label{fig:graphonebis}
\end{figure}

In the case of quadrangulations, we have
\begin{equation*}
\begin{split}
F_n^\bullet=& \alpha_0 \hat{\mathbb{Z}}_{0,0}^{\bullet\bullet +}(2n)+\alpha_1 \hat{\mathbb{Z}}_{0,0}^{\bullet\bullet +}(2n+2)\\
&\alpha_0=1+\frac{BW}{t_\bullet} \qquad \alpha_1=-\frac{B}{t_\bullet}\ .\\
\end{split}
\end{equation*}
The function $F_{m+n}$ may therefore be understood as the generating function for configurations of a directed (from left to right) path 
starting from point $E_m$  (with coordinates $(-2m-1,0)$) and ending at point $E_n'$ (with coordinates $(2n+1,0)$), traveling 
on the graph of Figure \ref{fig:graphone}, with appropriate edge weights designed so as to reproduce the above formula.
More precisely, all diagonal edges used by the path receive a weight $b=\sqrt{B}$ or $w=\sqrt{W}$ respectively according to the
white or black color of their lower vertex, except for the diagonal edges lying in the central vertical column (i.e.,\ whose 
abscissas are between $-1$ and $1$), which receive instead a weight $b\sqrt{\alpha_1}$ or $w\sqrt{\alpha_1}$ 
(according again to the white or black color of their lower vertex). As for any horizontal edge used by the path in the central column,
it  gives rise instead to a weight $\alpha_0$. Now, from the LGV (Lindstr\"om-Gessel-Viennot) lemma, see for instance
\cite{GV85,GV89}, $h_i^{(0)}=\det_{0\leq m,n\leq i}F_{m+n}$
enumerates configurations of $i+1$ \emph{mutually avoiding} directed paths with starting points $E_0,\cdots E_i$ and 
endpoints $E_0',\cdots,E_i'$, see Figure \ref{fig:graphonebis}. Because of the constraint of mutual avoidance, the parts of the paths lying outside
the central column are entirely fixed, made only of ascending steps on the left side leading to black vertices with abscissa $-1$ and
heights $0,2,\cdots,2i$ and made only of descending steps
on the right side starting from black vertices with abscissa $1$ and heights $0,2,\cdots,2i$. The total weight of these portions of paths 
is $(BW)^{\frac{i(i+1)}{2}}$. 
As for the parts of the paths in the central column, they connect two black vertices of the same height (between $0$ and $2i$).
This connection is done either via a horizontal edge, or via a sequence of two consecutive up/down or down/up diagonal edges. 
We may therefore decide to
weigh the central column by $\alpha_0^{i+1}$ and to correct by a multiplicative weight $s_1=W (\alpha_1/\alpha_0)$ for
each used up/down sequence, and  a multiplicative weight $s_2=B (\alpha_1/\alpha_0)$ for
each used down/up sequence. Note now that the mutual avoidance constraint prevents up/down and down/up portions
to share a common white vertex, so that these portions repel each other and (by a simple vertical projection) act as
\emph{hard dimers on a bicolored} (vertical) \emph{oriented} (from bottom to top) \emph{segment}, see Figure \ref{fig:graphonebis}.

More precisely, we call a \emph{bicolored oriented segment} an oriented segment (i.e.,\ a finite oriented linear graph) 
made of links whose nodes are bicolored alternatively in black and white. 
Each link of the segment may be occupied by a dimer or not, with the constraint that \emph{a node is incident 
to at most one dimer}. Each dimer lying on a link oriented from a black to a white node receives the weight $s_1$ and each dimer lying 
on a link oriented from a white to a black node receives the weight $s_2$. 
We shall denote by $Z_{HD[0,2i]}^{\bullet\bullet}\equiv Z_{HD[0,2i]}^{\bullet\bullet}(s_1,s_2)$ the generating
function of hard dimers on a bicolored oriented segment made of $2i$ links whose first and last nodes are black.  We shall also use
the notations $Z_{HD[0,2i]}^{\circ\circ}$, $Z_{HD[0,2i+1]}^{\bullet\circ}$ and $Z_{HD[0,2i+1]}^{\circ\bullet}$
for the other possible colors of the extremal nodes, with obvious definitions. 

In the present case, the hard dimers configurations gathering the weights of the central column live on a segment of length $2i+1$
starting with a black node and ending with a white one, so that we 
may eventually write
\begin{equation*}
\begin{split}
&h_i^{(0)}=(BW)^{\frac{i(i+1)}{2}}\alpha_0^{i+1} Z_{HD[0,2i+1]}^{\bullet\circ}\\
& \hbox{with dimer weights}\ s_1=W\frac{\alpha_1}{\alpha_0}\qquad s_2=B\frac{\alpha_1}{\alpha_0}\ .\\
\end{split}
\end{equation*}
The generating functions  $Z_{HD[0,2i]}^{\bullet\bullet}$, $Z_{HD[0,2i]}^{\circ\circ}$, $Z_{HD[0,2i+1]}^{\bullet\circ}$ and $Z_{HD[0,2i+1]}^{\circ\bullet}$
are computed in Appendix C.  They are best expressed upon introducing 
the parametrization
 \begin{equation*}
s_1=- \frac{x}{(c+x)(1+c\, x)}\qquad s_2=-\frac{c^2x}{(c+x)(1+c\, x)}\ ,
\end{equation*}
which is achieved by taking
\begin{equation*}
c=\sqrt{\frac{s_2}{s_1}}=\sqrt{\frac{B}{W}}\qquad x+\frac{1}{x}=-\frac{1+s_1+s_2}{c\, s_1}=-\frac{\frac{\alpha_0}{\alpha_1}+B+W}{\sqrt{B W}}\ .
\end{equation*}
The definition of $c$ matches precisely our definition of the general formalism. As for $x$, using \eqref{eq:recurBWquad}, we find that 
$(\alpha_0/\alpha_1)+B+W=-(1-2(B+W))$ so that the above equation for $x$
matches precisely the characteristic equation \eqref{eq:charquad}.
In terms of $x$ and $c$, we have (see Appendix C)
\begin{equation*}
Z_{HD[0,2i+1]}^{\bullet\circ}=(1+c\, x)\left(\frac{c}{(c+x)(1+c\, x)}\right)^{i+1} \frac{1-\frac{c+x}{1+c\, x} x^{2i+3}}{1-x^2}\ .
\end{equation*}
so that 
\begin{equation*}
h_i^{(0)}=(BW)^{\frac{i(i+1)}{2}}\left(\frac{\alpha_0 c}{(c+x)(1+c\, x)}\right)^{i+1} \frac{(1+c\, x)}{1-x^2} \left(1-\frac{c+x}{1+c\, x} x^{2i+3}\right)\ .
\end{equation*}
This is precisely the result \eqref{eq:hzeroquad} of our general formalism, since
\begin{equation*}
\frac{\alpha_0 c}{(c+x)(1+c\, x)}=\frac{\alpha_0}{x}\sqrt{s_1\, s_2}=  \frac{\alpha_0}{x}\left(-\frac{\alpha_1}{\alpha_0}\right)\sqrt{BW}=\frac{B}{t_\bullet} \frac{\sqrt{BW}}{x}\ .
\end{equation*}
\begin{figure}
\begin{center}
\includegraphics[width=8cm]{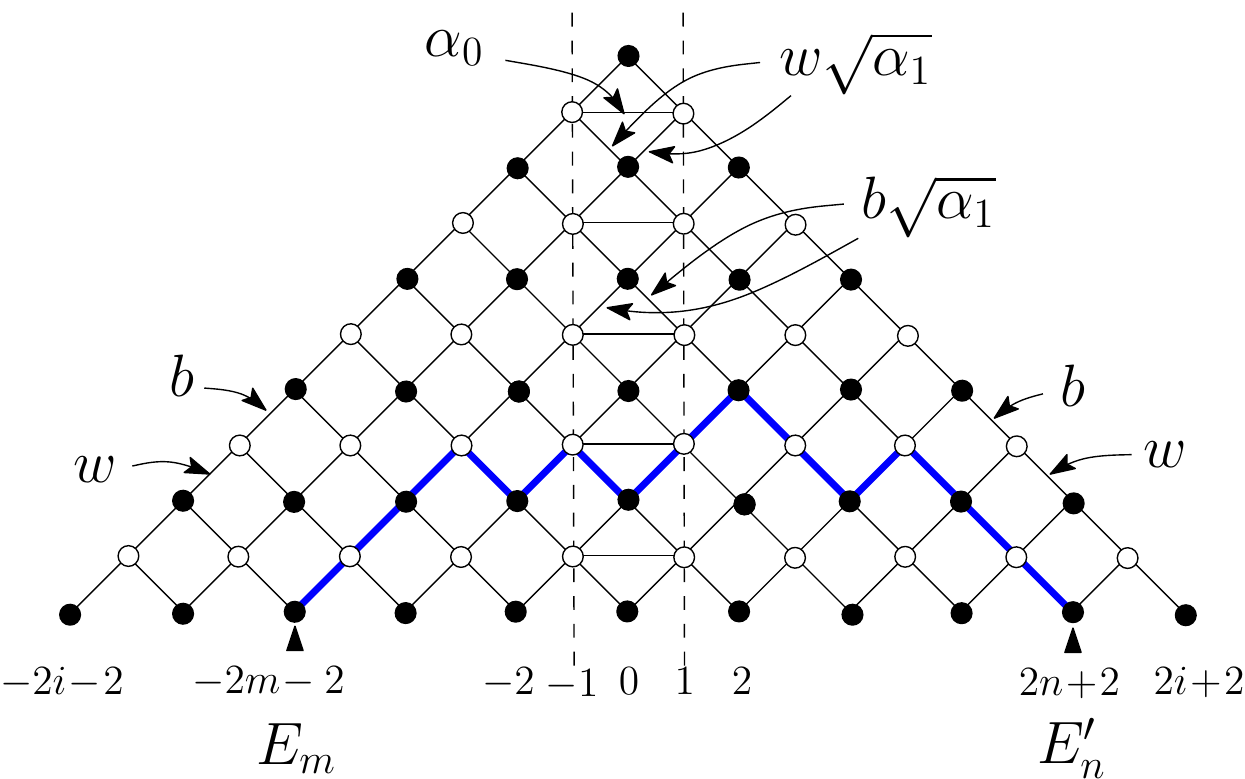}
\end{center}
\caption{A graph designed so that the generating function for directed (from left to right) paths 
from $E_m$  (with coordinates $(-2m-2,0)$) to $E_n'$ (with coordinates $(2n+2,0)$) precisely 
reproduces $F_{n+m+1}$ for quadrangulations.}
\label{fig:graphtwo}
\end{figure}
\begin{figure}
\begin{center}
\includegraphics[width=12cm]{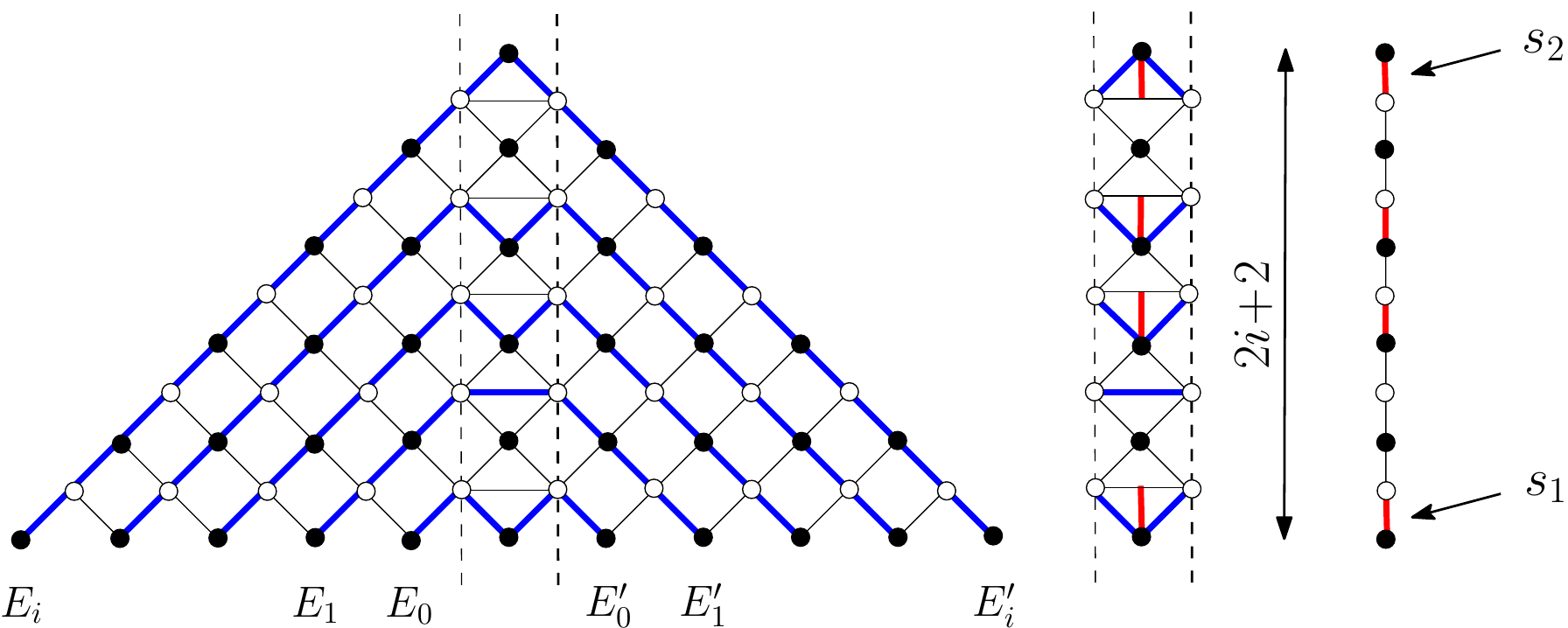}
\end{center}
\caption{Mutually avoiding paths on the graph of Figure \ref{fig:graphtwo}, from
$E_0,E_1, \dots E_i$ to $E_0', E_1',\cdots E_i'$. The only freedom 
comes from the central column (between abscissas $-1$ and $1$). The path configurations 
in this column are in one-to-one correspondence with hard dimers on a bicolored segment 
of length $2i+2$.}
\label{fig:graphtwobis}
\end{figure}

We may now play the same game to compute $h_i^{(1)}$ by interpreting the function $F_{m+n+1}$ as the generating function for configurations of a directed (from left to right) path starting from point $E_m$  (with coordinates $(-2m-2,0)$) and ending at point $E_n'$ (with coordinates $(2n+2,0)$), traveling now on the graph of Figure \ref{fig:graphtwo}, with the same weight prescription as before. From 
the LGV lemma, $h_i^{(1)}=\det_{0\leq m,n\leq i}F_{m+n+1}$
now enumerates configurations of $i+1$ mutually avoiding directed paths with starting points $E_0,\cdots E_i$ and 
endpoints $E_0',\cdots,E_i'$, see Figure \ref{fig:graphtwobis}. The total weight of the (entirely fixed) portions of paths outside 
of the central column is now
$W^{i+1}(BW)^{\frac{i(i+1)}{2}}$, while the contribution of the central column now reads $\alpha_0^{i+1}
Z_{HD[0,2i+2]}^{\bullet\bullet}$, so that (see Appendix C for the formula for $Z_{HD[0,2i+2]}^{\bullet\bullet}$)
\begin{equation*}
\begin{split}
h_i^{(1)}&=W^{i+1}(BW)^{\frac{i(i+1)}{2}}\alpha_0^{i+1}
Z_{HD[0,2i+2]}^{\bullet\bullet}\\
&=W^{i+1}(BW)^{\frac{i(i+1)}{2}}\left(\frac{\alpha_0 c}{(c+x)(1+c\, x)}\right)^{i+1} \frac{1}{1-x^2}\left(1- x^{2i+4}\right)
\end{split}
\end{equation*}
in agreement with the formula \eqref{eq:honequad} of the general formalism.
We leave as an exercise to the reader the care of computing $\tilde{h}_i^{(0)}$ and $\tilde{h}_i^{(1)}$ via the dimer
formalism and checking that their expressions match the formulas of the general formalism.

\subsection{The case of hexangulations}
\label{sec:harddimershex}
\begin{figure}
\begin{center}
\includegraphics[width=8cm]{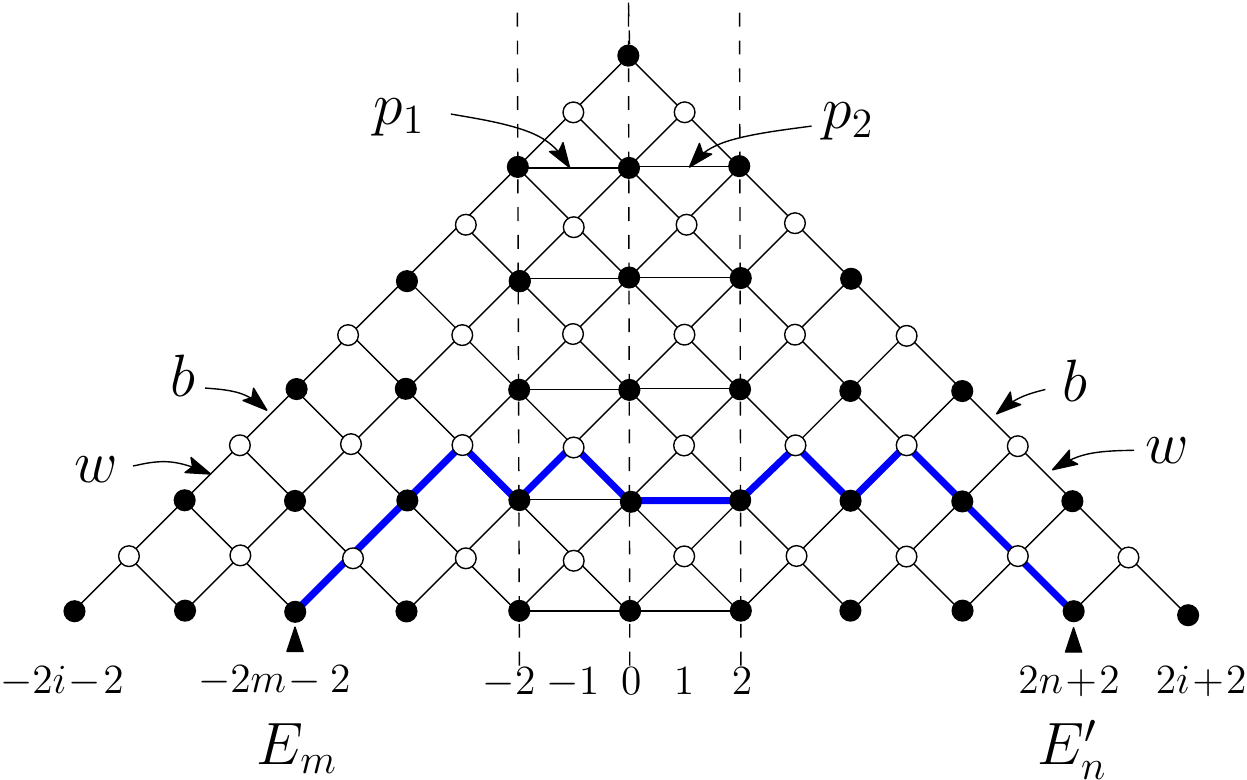}
\end{center}
\caption{A graph designed so that the generating function for directed (from left to right) paths 
from $E_m$  (with coordinates $(-2m-2,0)$) to $E_n'$ (with coordinates $(2n+2,0)$) precisely 
reproduces $F_{n+m}$ for hexangulations.}
\label{fig:graphthree}
\end{figure}
\begin{figure}
\begin{center}
\includegraphics[width=12cm]{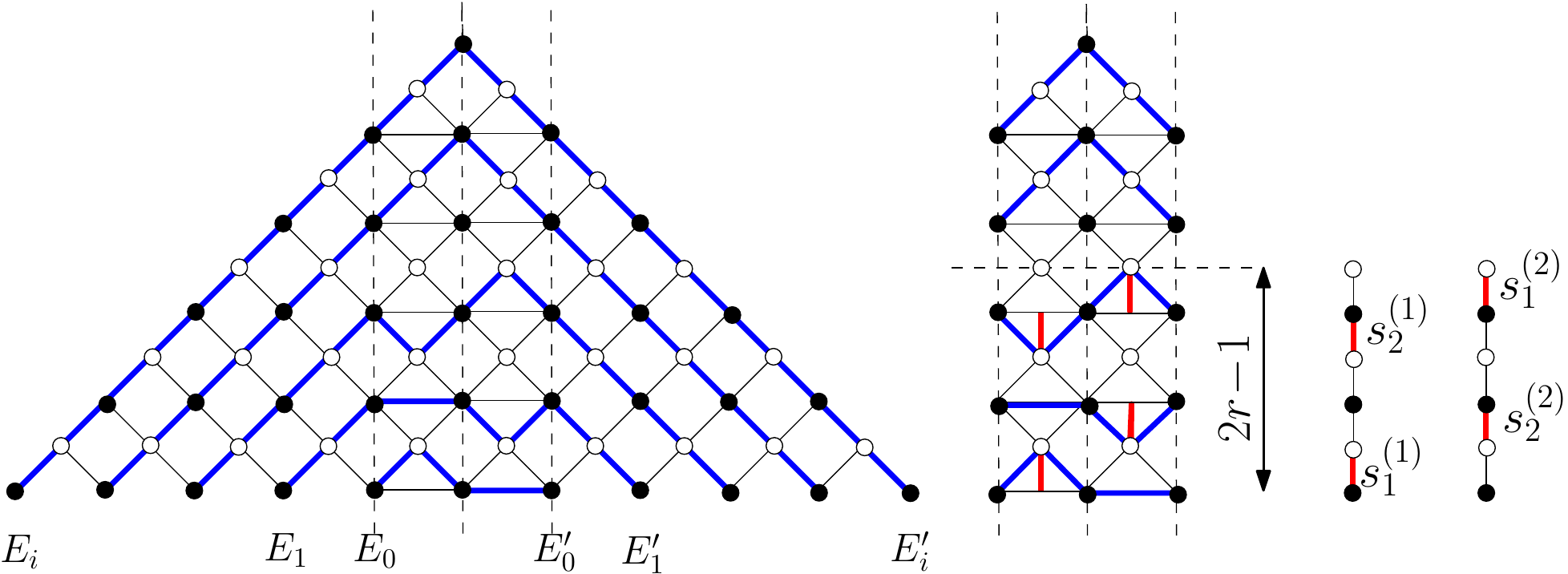}
\end{center}
\caption{Mutually avoiding paths on the graph of Figure \ref{fig:graphthree}, from
$E_0,E_1, \dots E_i$ to $E_0', E_1',\cdots E_i'$. The only freedom 
comes from the two central columns (between abscissas $-2$ and $2$). The path configurations 
in these columns are in one-to-one correspondence with couples of two hard dimer configurations on bicolored 
segments of the same length $2r-1$, with $r$ ranging from $0$ to $i+1$ (by convention, the contribution of configurations with $r=0$
is taken to be $1$).}
\label{fig:graphthreebis}
\end{figure}

 In the case of hexangulations, we have
\begin{equation*}
\begin{split}
F_n^\bullet=& \alpha_0 \hat{\mathbb{Z}}_{0,0}^{\bullet\bullet +}(2n)+\alpha_1 \hat{\mathbb{Z}}_{0,0}^{\bullet\bullet +}(2n+2)
+\alpha_2 \hat{\mathbb{Z}}_{0,0}^{\bullet\bullet +}(2n+4)\\
& \alpha_0=\frac{B}{t_\bullet}(1-L_0(4))=\frac{B}{t_\bullet}(1-B^2-W^2-4 B W)\\
&\alpha_1=\frac{B}{t_\bullet}(-L_0(2))=-\frac{B}{t_\bullet}(B+W)\\
&\alpha_2=\frac{B}{t_\bullet}(-L_0(0))=-\frac{B}{t_\bullet}\ .\\
\end{split}
\end{equation*}
To compute $h_i^{(0)}$, we shall here use a slightly different strategy from that used for quadrangulations, i.e.,\ look
at the function $F_{m+n}/\alpha_2$.
The function $F_{m+n}/\alpha_2$ is indeed the generating function for configurations of a directed (from left to right) path 
starting from point $E_m$  (with coordinates $(-2m-2,0)$) and ending at point $E_n'$ (with coordinates $(2n+2,0)$), traveling 
on the graph of Figure \ref{fig:graphthree}, with the following appropriate edge weights:
all diagonal edges used by the path receive a weight $b=\sqrt{B}$ or $w=\sqrt{W}$ 
respectively according to the
white or black color of their lower vertex, \emph{including those diagonal edges lying in the two central vertical columns} (i.e.,\ whose 
abscissas are between $-2$ and $2$). As for any horizontal edge used by the path in the two central columns,
they give rise instead to a weight $p_1$ for the first central column (with abscissas between $-2$ and $0$) and 
$p_2$ in the second central column (with abscissas between $0$ and $2$), with
\begin{equation*}
p_1+p_2=\frac{\alpha_1}{\alpha_2}=B+W \qquad p_1 p_2=\frac{\alpha_0}{\alpha_2}=B^2+W^2+4 B W-1\ .
\end{equation*}
Indeed, with these weights, the set of those paths using no horizontal edge indeed contributes $\hat{\mathbb{Z}}_{0,0}^{\bullet\bullet +}(2m+2n+4)$,
that of those paths using one horizontal edge contributes $(\alpha_1/\alpha_2)\hat{\mathbb{Z}}_{0,0}^{\bullet\bullet +}(2m+2n+2)$
and that of those paths using two horizontal edges contributes $(\alpha_0/\alpha_2)\hat{\mathbb{Z}}_{0,0}^{\bullet\bullet +}(2m+2n)$.
As before, $h_i^{(0)}/\alpha_2^{i+1}$ enumerates mutually avoiding paths with starting points $E_0,\cdots E_i$ and
endpoints $E_0',\cdots, E_i'$, see Figure \ref{fig:graphthreebis}. Again the parts of the paths lying outside
the two central columns is entirely fixed, made only of ascending steps on the left side leading to black vertices with abscissa $-2$ and
heights $0,2,\cdots,2i$ and made only of descending steps
on the right side starting from black vertices with abscissa $2$ and heights $0,2,\cdots,2i$. The total weight of these portions of paths 
is again $(BW)^{\frac{i(i+1)}{2}}$. 
The parts of the paths in the two central columns connect two black vertices of the same height (between $0$ and $2i$) and
it is interesting to classify these paths according to the position of their passage at abscissa $0$ (i.e.,\ at the contact of the
two central columns). Due to the mutual avoidance constraint, the corresponding heights are $0,2,\cdots,2 r-2$ for the lower
$r$ paths and $2r+2,2r+4,\cdots 2i+2$ for the higher $i+1-r$ paths, with $r$ some integer ranging from $0$ to $i+1$,
see Figure \ref{fig:graphthreebis}.
The last higher $i+1-r$ paths contribute $(BW)^{i+1-r}$ while the $r$ first ones contribute
$(p_1)^r Z_{HD[0,2r-1]}^{\bullet\circ}(s_1^{(1)},s_2^{(1)})$ (from the first central column) and 
$(p_2)^r Z_{HD[0,2r-1]}^{\bullet\circ}(s_1^{(2)},s_2^{(2)})$ (from the second central column) with
\begin{equation*}
s_1^{(1)}=\frac{W}{p_1}\qquad s_2^{(1)}=\frac{B}{p_1}\qquad s_1^{(2)}=\frac{W}{p_2}\qquad s_2^{(2)}=\frac{B}{p_2}
\end{equation*}
and the convention $Z_{HD[0,-1]}^{\bullet\bullet}=1$.
This leads to the formula
\begin{equation}
\hspace{-.3cm} \frac{h_i^{(0)}}{\alpha_2^{i+1}}= (BW)^{\frac{i(i+1)}{2}} \sum_{r=0}^{i+1}(BW)^{i+1-r} (p_1p_2)^r Z_{HD[0,2r-1]}^{\bullet\circ}(s_1^{(1)},s_2^{(1)})Z_{HD[0,2r-1]}^{\bullet\circ}(s_1^{(2)},s_2^{(2)})\ .
\label{eq:hzerohexdimer}
\end{equation}
Using the parametrization
\begin{equation*}
\begin{split}
&s_1^{(1)}=\frac{W}{p_1}=-\frac{x_1}{(c+x_1)(1+c\, x_1)}\qquad s_2^{(1)}=\frac{B}{p_1}=-\frac{c^2 x_1}{(c+x_1)(1+c\, x_1)}\\
&s_1^{(2)}=\frac{W}{p_2}=-\frac{x_2}{(c+x_2)(1+c\, x_2)}\qquad s_2^{(2)}=\frac{B}{p_2}=-\frac{c^2 x_2}{(c+x_2)(1+c\, x_1)}\\
\end{split}
\end{equation*}
by setting
\begin{equation*}
\begin{split}
&c=\sqrt{\frac{s_2^{(1)}}{s_1^{(1)}}}=\sqrt{\frac{s_2^{(2)}}{s_1^{(2)}}}=\sqrt{\frac{B}{W}}\\
&x_1+\frac{1}{x_1}=-\frac{1+s_1^{(1)}+s_2^{(1)}}{c\, s_1^{(1)}}=-\frac{p_1+B+W}{\sqrt{BW}}\\
&x_2+\frac{1}{x_2}=-\frac{1+s_1^{(2)}+s_2^{(2)}}{c\, s_1^{(2)}}=-\frac{p_2+B+W}{\sqrt{BW}}\ ,\\
\end{split}
\end{equation*}
we can use our general formula for $Z_{HD[0,2r-1]}^{\bullet\circ}$ and write for instance
\begin{equation*}
\begin{split}
(p_1)^r Z_{HD[0,2r-1]}^{\bullet\circ}(s_1^{(1)},s_2^{(1)})&=(1+c\, x_1)\left(\frac{p_1 c}{(c+x_1)(1+c\, x_1)}\right)^{r} \frac{1-\frac{c+x_1}{1+c\, x_1} x_1^{2r+1}}{1-x_1^2}\\
&=(1+c\, x_1)\left(\frac{\sqrt{BW}}{x_1}\right)^{r} \frac{1-\frac{c+ x_1}{1+c\,  x_1} x_1^{2r+1}}{1-x_1^2}\\
\end{split}
\end{equation*}
(note that the formula also holds for $r=0$ with our convention $Z_{HD[0,-1]}^{\bullet\bullet}=1$). Eq.~\eqref{eq:hzerohexdimer} yields
\begin{equation*}
\hspace{-1.3cm}h_i^{(0)}=\alpha_2^{i+1}(BW)^{\frac{(i+1)(i+2)}{2}}\frac{1+c\, x_1}{1-x_1^2}\frac{1+c\, x_2}{1-x_2^2}
\sum_{r=0}^{i+1} (x_1^{-r}-\frac{c+ x_1}{1+c\,  x_1} x_1^{r+1})(x_2^{-r}-\frac{c+ x_2}{1+c\,  x_2} x_2^{r+1})\ .
\end{equation*}
The sum over $r$ is easily performed and we end up with the desired result
\begin{equation*}
h_i^{(0)} =\alpha_2^{i+1} (B W)^{\frac{(i+1)(i+2)}{2}} \frac{1}{(x_1 x_2)^{i+1}}
\frac{1+c\, x_1}{1-x_1^2}\frac{1+c\, x_2}{1-x_2^2}\frac{1}{1-x_1x_2}\, \bar{u}_{2i+3}
\end{equation*}
with $\bar{u}_i$ as in eq.~\eqref{eq:hzerohex}.
For a full check of consistency, we still need to verify that $x_1$ and $x_2$ are the solutions
of the correct characteristic equation \eqref{eq:charhex}. From their definition, $p_1$ and $p_2$ are
solutions of the equation
\begin{equation*}
0=p^2-\frac{\alpha_1}{\alpha_2}p+\frac{\alpha_0}{\alpha_2}=p^2-(B+W)p-1+B^2+W^2+4BW
\end{equation*}
so that $x_1$ and $x_2$ are solutions of 
\begin{equation*}
\begin{split}
\hspace{-.3cm} 0=&\left(-(B+W)-\sqrt{BW} \left(x+\frac{1}{x}\right)\right)^2-(B+W)\left(-(B+W)-\sqrt{BW} \left(x+\frac{1}{x}\right)\right)\\
&\hspace{8.cm}  -1+B^2+W^2+4BW\\
\end{split}
\end{equation*}
which after simplification precisely reproduces \eqref{eq:charhex}.

\begin{figure}
\begin{center}
\includegraphics[width=8cm]{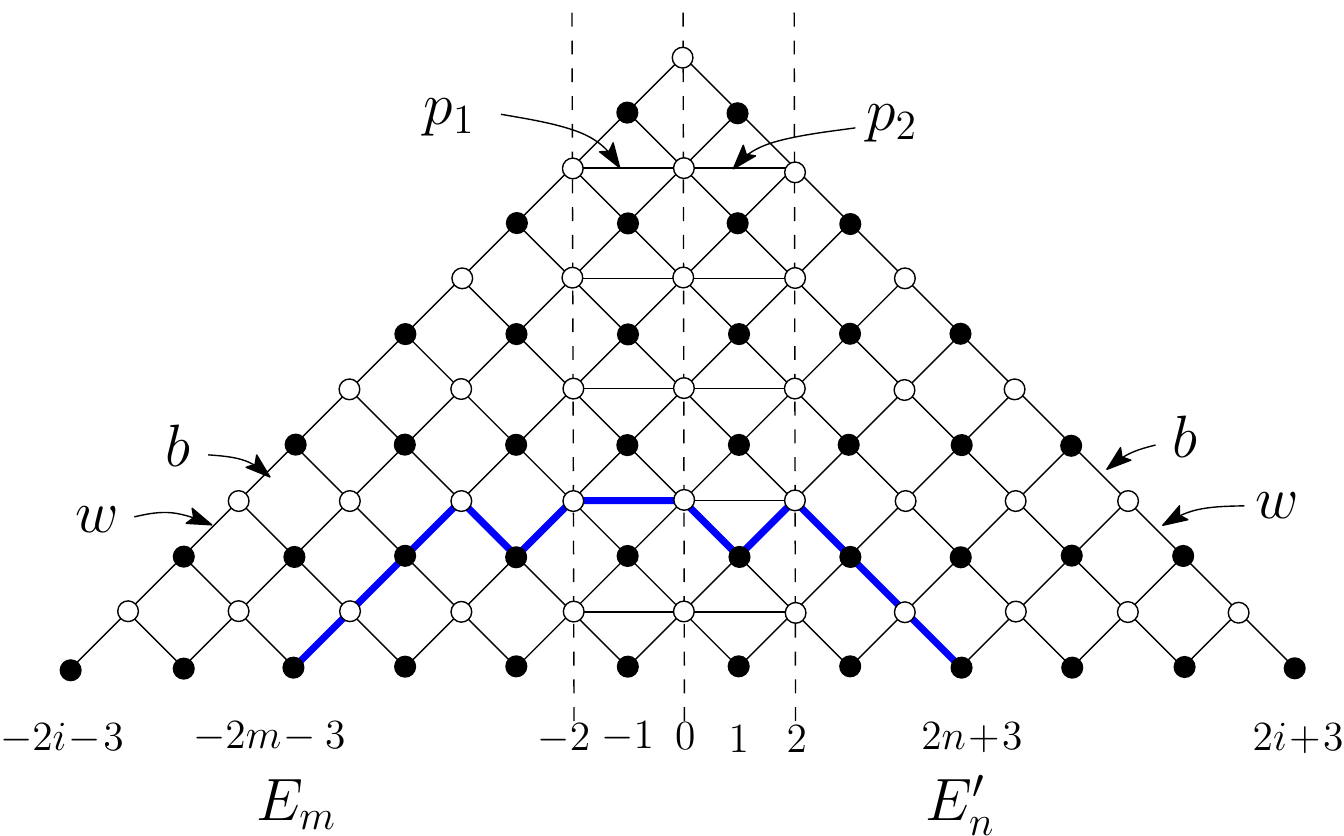}
\end{center}
\caption{A graph designed so that the generating function for directed (from left to right) paths 
from $E_m$  (with coordinates $(-2m-3,0)$) to $E_n'$ (with coordinates $(2n+3,0)$) precisely 
reproduces $F_{n+m+1}$ for hexangulations.}
\label{fig:graphfour}
\end{figure}
\begin{figure}
\begin{center}
\includegraphics[width=12cm]{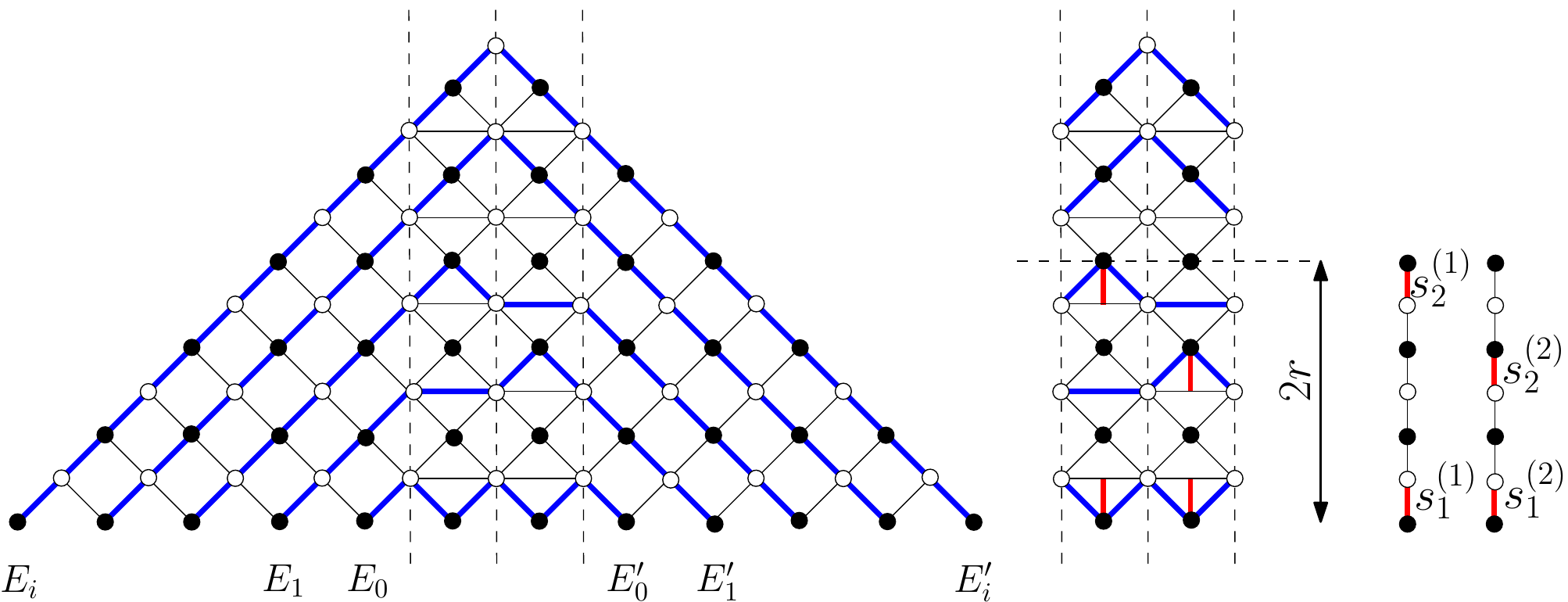}
\end{center}
\caption{Mutually avoiding paths on the graph of Figure \ref{fig:graphfour}, from
$E_0,E_1, \dots E_i$ to $E_0', E_1',\cdots E_i'$. The only freedom 
comes from the two central columns (between abscissas $-2$ and $2$). The path configurations 
in these columns are in one-to-one correspondence with pairs of two hard dimer configurations on bicolored 
segments of the same length $2r$, with $r$ ranging from $0$ to $i+1$.}
\label{fig:graphfourbis}
\end{figure}

If we now wish to play the same game to compute $h_i^{(1)}$,  we interprete $F_{m+n+1}$ as the generating function for configurations of a directed path starting from point $E_m$  (with coordinates $(-2m-3,0)$) and ending at point $E_n'$ (with coordinates $(2n+3,0)$) traveling on the graph of 
Figure \ref{fig:graphfour}, with the same weight prescription as before. Then  $h_i^{(1)}/\alpha_2^{i+1}$ enumerates mutually avoiding paths with starting points $E_0,\cdots E_i$ and endpoints $E_0',\cdots, E_i'$, see Figure~\ref{fig:graphfourbis}, and we now obtain 
\begin{equation*}
\hspace{-.3cm} \frac{h_i^{(1)}}{\alpha_2^{i+1}}= W^{i+1}(BW)^{\frac{i(i+1)}{2}} \sum_{r=0}^{i+1}(BW)^{i+1-r} (p_1p_2)^r Z_{HD[0,2r]}^{\bullet\bullet}(s_1^{(1)},s_2^{(1)})Z_{HD[0,2r]}^{\bullet\bullet}(s_1^{(2)},s_2^{(2)})
\end{equation*}
which yields
\begin{equation*}
\hspace{-1.3cm}h_i^{(1)}=\alpha_2^{i+1}W^{i+1}(BW)^{\frac{(i+1)(i+2)}{2}}\frac{1}{1-x_1^2}\frac{1}{1-x_2^2}
\sum_{r=0}^{i+1} (x_1^{-r}-x_1^{r+2})(x_2^{-r}-x_2^{r+2})
\end{equation*}
and, after summation over $r$
\begin{equation*}
h_i^{(1)} =\alpha_2^{i+1} W^{i+1} (B W)^{\frac{(i+1)(i+2)}{2}} \frac{1}{(x_1 x_2)^{i+1}}
\frac{1}{1-x_1^2}\frac{1}{1-x_2^2}\frac{1}{1-x_1x_2}\, u_{2i+4}
\end{equation*}
with $u_i$ as in eq.~\eqref{eq:honehex}.
We again leave as an exercise to the reader the care of computing $\tilde{h}_i^{(0)}$ and $\tilde{h}_i^{(1)}$ via the dimer
formalism and checking that their expressions match the formulas of general formalism.

 \section{Conclusion}
\label{sec:conclusion}

In this paper, we obtained expressions for the two-point function of bicolored maps, with a control on their (even) face degrees
and on their numbers of black and white vertices (via the weights $t_\bullet$ and $t_\circ$), in the form of explicit formulas
for the corresponding black or white $i$-slice generating functions. For maps with faces of degrees at most $2p+2$, 
these formulas take the form of ratios of $p\times p$ determinants, generalizing those found in \cite{BG12} in the uncolored version
of the problem. 
In the simplest case of bicolored quadrangulations ($p=1$) and hexangulations ($p=2$), the very same formulas may 
be recovered via an equivalence with appropriate hard dimer problems on segments, with parity dependent weights.

Let us now discuss a few extensions of our results. First, starting from generating functions for our families of maps, we can in some
cases, by a simple substitution, get generating functions for other families of maps, incorporating additional restrictions (such as 
the absence of multiple edges, ....). Such substitutions are particularly useful in the context of \emph{irreducible} maps
(with a control on the length of their smallest cycles) and, as shown in \cite{BG14}, the required substitution may then be determined 
so as to give (after substitution) trivial values (such as $1$ or so) to the first conserved quantities (which are non-trivial in the original problem).
This in turn provides a way to construct new equations which essentially have the same solution as the original ones,
up to a redefinition of the weights, see \cite{BG14}. We can play this game starting, say, from our solution of bicolored quadrangulations
to derive new sets of integrable equations, with explicit solutions.

Another extension deals with the case of $p$-constellations for $p>2$. Recall that a $p$-constellation
is a face-bicolored map (say with dark and light faces) such that all dark faces have degree $p$ and all
light faces have a degree multiple of $p$. In the planar case, we may naturally color the vertices of these
maps in $p$ colors in increasing (resp.\ decreasing) order clockwise around the black (resp.\ the white faces).
Our bicolored maps are nothing but $2$-constellations (where the dark faces -- of degree $2$ -- have been squeezed into 
edges) and we may try to extend our results to $p$-constellations with $p>2$, giving different weights to the vertices, 
according to their color. We discuss below an example with $p=3$.

\subsection{Other integrable equations from the solution of bicolored quadrangulations}
\label{sec:otherintregquad}
Using the conserved quantities \eqref{eq:F1quad} for bicolored quadrangulations, the corresponding $B_i$'s and $W_i$'s
may alternatively be viewed as the solutions of 
\begin{equation*}
\begin{split}
& F_1^\bullet=W_{i}-\frac{1}{t_\bullet}B_{i-1}W_{i}B_{i+1} = W_1\\
& F_1^\circ=B_{i}-\frac{1}{t_\circ}W_{i-1}B_{i}W_{i+1} = B_1\\
\end{split}
\end{equation*}
for $i\geq 1$ with $B_0=W_0=0$. Note that these equations are however weaker than the original system as they do not determine 
in practice $B_1$ and $W_1$, nor $B$ and $W$ but only the relations $B_1=B(1-W^2/t_\circ)$ and $W_1=W(1-B^2/t_\bullet)$
(while the original equations fix $B$ and $W$ via $t_\bullet = B(1-2W-B)$ and $t_\circ=W(1-2B-W)$, and the values of $B_1$ and $W_1$
follow). We may in practice eliminate $B_1$ and $W_1$ upon defining
\begin{equation*}
P_i=\frac{B_i}{B_1}\qquad Q_i=\frac{W_i}{W_1} 
\end{equation*} 
for $i\geq 0$. We indeed get for $P_i$ and $W_i$ the system of equations
\begin{equation}
\begin{split}
& P_i=1+z_\circ Q_{i-1}P_i Q_{i+1}\\
& Q_i=1+z_\bullet P_{i-1}Q_i P_{i+1}\\
\end{split}
\label{eq:firstinteg}
\end{equation}
valid for $i\geq 1$, with initial conditions $P_0=Q_0=0$, and with
\begin{equation*}
z_\circ=\frac{W_1^2}{t_\circ}\qquad
z_\bullet=\frac{B_1^2}{t_\bullet}\ .
\end{equation*}
We may in practice forget about these latter relations and consider $z_\bullet$ and $z_\circ$ as our new input.
Note that $P_1=Q_1$ are now determined,
with value $1$.
We immediately deduce from \eqref{eq:solquadexpl} the solution of \eqref{eq:firstinteg} 
\begin{equation*}
\begin{split}
& P_{2i}=P\frac{u_{2i}\bar{u}_{2i+3}}{\bar{u}_{2i+1}u_{2i+2}}\qquad Q_{2i+1}=Q\frac{\bar{u}_{2i+1}u_{2i+4}}{u_{2i+2}\bar{u}_{2i+3}}\\
& P_{2i+1}=P\frac{\hat{u}_{2i+1}u_{2i+4}}{u_{2i+2}\hat{u}_{2i+3}}\qquad Q_{2i}=Q\frac{u_{2i}\hat{u}_{2i+3}}{\hat{u}_{2i+1}u_{2i+2}}\\
\end{split}
\end{equation*}
where $P\equiv B/B_1$ and $Q\equiv W/W_1$ are the solutions of
\begin{equation*}
\begin{split}
& P=1+z_\circ Q^2 P\\
& Q=1+z_\bullet P^2 Q\ ,\\
\end{split}
\end{equation*}
and $u_i$, $\bar{u}_i$ and $\hat{u}_i$ have the same definitions as before in terms of $x$ and $c$.
As for $c$ and $x$ themselves, they can be related directly to $P$ and $Q$ via
\begin{equation*}
c=\sqrt{\frac{Q-1}{P-1}}\ , \qquad \left(x+\frac{1}{x} \right) \sqrt{(P-1)(Q-1)}=1\ .
\end{equation*}
These equations are easily obtained by rewriting $B$ and $W$ in terms of $P$ and $Q$.
Using $(P-1)/(Q^2 P)=z_\circ=W_1^2/t_\circ=(W/Q)^2/t_\circ$ and, from \eqref{eq:recurBWquad},  the relation $t_\circ=W(1-2B-W)$,
we get $(P-1)=W/(1-2B-2W)$ and, similarly, $(Q-1)=B/(1-2B-2W)$ so that $B/W=(Q-1)/(P-1)$, while \eqref{eq:charquad} 
immediately yields the above equation for $x$. Alternatively, both equations may be obtained 
directly without recourse to neither \eqref{eq:recurBWquad} nor \eqref{eq:charquad} upon simply writing $P_1=Q_1=1$ and, using their explicit expressions above 
$P_{1}=P(\hat{u}_{1}u_{4})/(u_{2}\hat{u}_{3})$ and $Q_{1}=Q(\bar{u}_{1}u_{4})/(u_{2}\bar{u}_{3})$, solving for $c$ and $x$.

From the form of \eqref{eq:firstinteg}, we can immediately interpret $P_i$ and $Q_i$ as 
generating functions for \emph{naturally embedded ternary trees in a semi-infinite line}, i.e.,\ ternary trees whose vertices 
occupy positive integer positions on a line and where each internal vertex at position $j$ has its three children  
at positions $j-1$, $j$ and $j+1$ respectively. Such trees start with a univalent (uncolored) root vertex and have their internal vertices
bicolored in black and white according to the parity of their position, and weighted by $z_\bullet$ 
and $z_\circ$ accordingly. More precisely (up to a first trivial term $1$ corresponding to the tree without internal vertices), 
$P_i$ corresponds to trees with a 
first internal vertex (i.e.,\ that attached to the univalent root vertex)
being at position $i$ and where each internal vertex whose position has the same parity as $i$ (resp. a different parity)
is white (resp. black), while $Q_i$ corresponds to a first internal vertex at position
$i$ with each internal vertex whose position has the same parity as $i$ (resp. a different parity)
being black (resp. white). Naturally embedded ternary trees in a semi-infinite line are known to appear
in the context of quadrangulations \emph{without multiple edges} (which
also correspond to nonseparable planar maps) 
 \cite{BG10,BG14,JaSc98}, and the above generating functions
would naturally appear in the bicolored version of this problem.

We may go one step further by using the next conserved quantities for bicolored quadrangulations, 
given (according to \eqref{eq:conserved} with $d=i-1$) by
\begin{equation*}
\begin{split}
& F_2^\bullet=W_{i}(W_i+B_{i+1})-\frac{1}{t_\bullet}(W_{i}+B_{i+1}+W_{i+2})B_{i-1}W_{i}B_{i+1} = W_1(W_1+B_2)\\
& F_2^\circ=B_{i}(B_i+W_{i+1})-\frac{1}{t_\circ}(B_{i}+W_{i+1}+B_{i+2})W_{i-1}B_{i}W_{i+1} = B_1(B_1+W_2)\\
\end{split}
\end{equation*}
for $i\geq 1$ with $B_0=W_0=0$. 
Dividing the first equation by $t_\circ$ and the second by $t_\bullet$, we get alternatively
\begin{equation*}
\begin{split}
& \frac{F_2^\bullet}{t_\circ}=Q_{i}(z_\circ Q_i+z_\bullet P_{i+1})-z_\bullet(z_\circ P_{i}+z_\bullet P_{i+1}+z_\circ Q_{i+2})P_{i-1}Q_{i}P_{i+1} = (z_\circ +z_\bullet P_2)\\
&\frac{F_2^\circ}{t_\bullet}=P_{i}(z_\bullet P_i+z_\circ Q_{i+1})-z_\circ(z_\bullet P_{i}+z_\circ Q_{i+1}+z_\bullet P_{i+2})Q_{i-1}P_{i}Q_{i+1} = (z_\bullet+z_\circ Q_2)\ ,\\
\end{split}
\end{equation*}
a system which however does not determine $P_2$ and $Q_2$.
Let us recall that, from \eqref{eq:FbulletFcirc}, $F_2^{\bullet}/t_\circ= F_2^{\circ}/t_\bullet$ so that 
both lines in the above system are in practice equal. In particular, $(z_\circ +z_\bullet P_2)= (z_\bullet+z_\circ Q_2)$,
so we can define
\begin{equation*}
\delta=\frac{Q_2-1}{z_\bullet}=\frac{P_2-1}{z_\circ}\ .
\end{equation*}
Using $z_\circ Q_{i-1}P_i Q_{i+1}=P_i-1$ and 
$z_\bullet P_{i-1}Q_i P_{i+1}=Q_i-1$, this system simplifies into
\begin{equation*}
\begin{split}
& z_\bullet P_{i+1}+z_\circ(Q_{i}+Q_{i+2}-Q_i Q_{i+2})=z_\bullet +z_\circ +z_\bullet z_\circ \delta\\
& z_\circ Q_{i+1}+z_\bullet(P_{i}+P_{i+2}-P_i P_{i+2})= z_\circ+z_\bullet +z_\circ z_\bullet \delta\ .\\
\end{split}
\end{equation*}
Upon setting
\begin{equation*}
R_i=\frac{P_{i+1}-1}{z_\circ \delta}\ , \qquad S_i=\frac{Q_{i+1}-1}{z_\bullet \delta}\ ,
\end{equation*}
the system becomes
\begin{equation}
\begin{split}
& R_i=1+y_\bullet S_{i-1}S_{i+1}\\
& S_i=1+y_\circ R_{i-1} R_{i+1}\\
\end{split}
\label{eq:secondtinteg}
\end{equation}
valid for $i\geq 1$ with $R_0=S_0=0$ (so that $R_1=S_1=1$), where we have set: 
\begin{equation*}
y_\bullet=z_\bullet \delta\ , \qquad y_\circ=z_\circ \delta\ .
\end{equation*}
Again we may forget about $z_\bullet$ and $z_\circ$ and consider $y_\bullet$ and $y_\circ$ as our new input
(the above change of functions being a way to get rid of the undetermined $P_2$ and $Q_2$).
Introducing $R=(P-1)/(z_\circ \delta)$ and $S=(Q-1)/(z_\bullet \delta)$, determined by the system:
\begin{equation*}
 R=1+y_\bullet S^2\ , \qquad 
 S=1+y_\circ R^2\ ,
\end{equation*}
and using the equations for $P_i$ and $Q_i$ to write $R_i=Q_i P_{i+1}Q_{i+2}/\delta$ and
$S_i=P_i Q_{i+1} P_{i+2}/\delta$ (and, accordingly, $R=Q^2P/\delta$ and $S=P^2 Q/\delta$),
we eventually get
 \begin{equation*}
\begin{split}
& R_{2i}=R\frac{u_{2i}\hat{u}_{2i+5}}{u_{2i+2}\hat{u}_{2i+3}}\qquad S_{2i+1}=S\frac{\hat{u}_{2i+1}u_{2i+6}}{\hat{u}_{2i+3}u_{2i+4}}\\
& R_{2i+1}=R\frac{\bar{u}_{2i+1}u_{2i+6}}{\bar{u}_{2i+3}u_{2i+4}}\qquad S_{2i}=S\frac{u_{2i}\bar{u}_{2i+5}}{u_{2i+2}\bar{u}_{2i+3}}\ .\\
\end{split}
\end{equation*}
As for $c$ and $x$, they are now related to $R$ and $S$ via
\begin{equation*}
c=\sqrt{\frac{R(R-1)}{S(S-1)}}\ , \qquad \left(x+\frac{1}{x}\right)\sqrt{\frac{(R-1)(S-1)}{RS}}=1\ .
\end{equation*}
These equations are obtained by writing $(P-1)=z_\circ \delta R=y_\circ R= (S-1)/R$ and, similarly, $(Q-1)=(R-1)/S$.
As before, both equations may alternatively be obtained by simply writing $R_1=S_1=1$ with their general expressions
above, and solving for $c$ and $x$.

From the form of \eqref{eq:secondtinteg}, we can immediately interpret $R_i$ and $S_i$ as generating functions for 
\emph{bicolored naturally embedded binary trees in a semi-infinite line} with black and white internal vertex weights
$y_\bullet$ and $y_\circ$. Such (uncolored) trees appear
in the context of \emph{irreducible} quadrangulations  \cite{BG14}, and the above generating functions
would therefore appear naturally in the bicolored version of this problem.

\subsection{An integrable system with $3$ colors}
\label{sec:oherintegthreecolor}
Another remarkable system of equations which may be solved is
\begin{equation*}
\begin{split}
&T_i=t_\bullet+T_i(U_{i-1}+V_{i+1})\\
&U_i=t_\circ+U_i(V_{i-1}+T_{i+1})\\
&V_i=t_{\o}+V_i(T_{i-1}+U_{i+1})\\
\end{split}
\end{equation*}
for $i\geq 1$, with $T_0=U_0=V_0=0$. For $t_\bullet=t_\circ=t_{\o}$, we have $T_i=U_i=V_i$ and the three 
equations are the same. This common equation appears in
the context of \emph{Eulerian triangulations} \cite{BDG03} which are the simplest example of $3$-constellations.
These maps are naturally divided into $3$ sublattices and the above system corresponds to giving a
different weight to the vertices according to which sublattice they belong to. Introducing the solutions $T,U,V$ of
\begin{equation*}
T=t_\bullet+T(U+V)\, \qquad
U=t_\circ+U(V+T)\, \qquad
V=t_{\o}+V(T+U)\ , 
\end{equation*}
the solution now depends on the congruence modulo $3$ of $i$. We find, for $i\geq 0$:
\begin{equation*}
\begin{split}
&\hspace{-1.3cm}T_{3i}=T \frac{(1\!-\!x^{3i})(1\!-\!\alpha x^{3i+4})}{(1\!-\!\alpha x^{3i+1})(1\!-\!x^{3i+3})}\quad
T_{3i+1}=T \frac{(1\!-\!\gamma x^{3i+1})(1\!-\!x^{3i+5}/\epsilon)}{(1\!-\!x^{3i+2}/\epsilon)(1\!-\!\gamma x^{3i+4})}\quad
T_{3i+2}=T \frac{(1\!-\!x^{3i+2}/\alpha)(1\!-\!x^{3i+6})}{(1\!-\!x^{3i+3})(1\!-\!x^{3i+5}/\alpha)}\\
&\hspace{-1.3cm}U_{3i}=U \frac{(1\!-\!x^{3i})(1\!-\!\gamma x^{3i+4})}{(1\!-\!\gamma x^{3i+1})(1\!-\!x^{3i+3})}\quad
U_{3i+1}=U \frac{(1\!-\!\epsilon x^{3i+1})(1\!-\!x^{3i+5}/\alpha)}{(1\!-\!x^{3i+2}/\alpha)(1\!-\!\epsilon x^{3i+4})}\quad
U_{3i+2}=U \frac{(1\!-\!x^{3i+2}/\gamma)(1\!-\!x^{3i+6})}{(1\!-\!x^{3i+3})(1\!-\!x^{3i+5}/\gamma)}\\
&\hspace{-1.3cm}V_{3i}=V \frac{(1\!-\!x^{3i})(1\!-\!\epsilon x^{3i+4})}{(1\!-\!\epsilon x^{3i+1})(1\!-\!x^{3i+3})}\quad
V_{3i+1}=V \frac{(1\!-\!\alpha x^{3i+1})(1\!-\!x^{3i+5}/\gamma)}{(1\!-\!x^{3i+2}/\gamma)(1\!-\!\alpha x^{3i+4})}\quad
V_{3i+2}=V \frac{(1\!-\!x^{3i+2}/\epsilon)(1\!-\!x^{3i+6})}{(1\!-\!x^{3i+3})(1\!-\!x^{3i+5}/\epsilon)}\ .\\
\end{split}
\end{equation*}
 Here $\alpha,\gamma$ and $\epsilon$ are expressed in terms of four quantities $t,u,v$ and $x$ only via
 \begin{equation*}
 \alpha=\frac{v+u\, x+t\, x^2}{t+v\, x+u\, x^2}  \qquad 
 \gamma=\frac{t+v\, x+u\, x^2}{u+t\, x+v\, x^2}  \qquad 
 \epsilon=\frac{u+t\, x+v\, x^2}{v+u\, x+t\, x^2}  
 \end{equation*} 
(note that $\alpha\gamma\epsilon=1$ and that these three quantities depend in practice only on the three quantities 
$x$, $u/t$ and $v/t$ so that we could therefore set $t=1$ without loss of generality). As for the quantities $t,u,v$ and 
$x$ themselves, they are obtained from $U,V$ and $T$ via
\begin{equation*}
T=\frac{u\,v\, x}{(u+t\, x)(t+v\, x)}\qquad U=\frac{t\,v\, x}{(u+t\, x)(v+u\, x)}\qquad V=\frac{t\,u\, x}{(v+u\, x)(t+v\, x)}
\end{equation*}
(again these three equations determine only $x$ and the ratios $u/t$ and $v/t$, which is all what we need in practice).
With these expressions, it is easy to check that $x$ satisfies the characteristic equation
\begin{equation*}
T U V \left(x^3+\frac{1}{x^3}+2\right)-(1-T-U-V)^2=0
\end{equation*}
(note that if $x$ is a solution, $\omega x$ and $\omega^2 x$, with $\omega=e^{2i\pi/3}$ are also solutions, as well as their inverses. The precise 
determination of $x$ is in practice irrelevant, for instance changing $x\to \omega x$ will result into $t\to t$, $u\to \omega u$, $v\to \omega^2 v$,
and then $\alpha\to \omega^2 \alpha$,
$\gamma\to \omega^2 \gamma$ and $\epsilon \to \omega^2 \epsilon$ so that the expressions for $T_i$, $U_i$ and $V_i$ remain the same).
The solution above was obtained by simple guessing. It should in principle be possible to obtain it 
in a constructive way via the formalism of multicontinued fractions developed in \cite{AB11}.

Let us finally mention that one can extract trivariate series expansions from these expressions, similarly
as was done in Section~\ref{sec:finalresult}. Defining $u'\equiv u/t$, $v'\equiv v/t$, the parametrization of $T, U, V$ in terms of $t,u,v,x$
is equivalent to the system
\begin{equation*}
\hspace{-1.3cm}1 = U\left(1+\frac{u'}{x}\right)+V(1+v'x),\quad u'= V\left(u'+\frac{v'}{x}\right)+T(u'+x),\quad v'= T\left(v'+\frac{1}{x}\right)+U(v'+u'x)\ .
\end{equation*}
Defining $y\equiv x^3$, $d\equiv u'x^2$, $e\equiv v'x$, this rewrites as
\begin{equation*}
y=U(y+d)+Vy(1+e),\quad  d=V(d+e)+T(d+y),\quad e=T(e+1)+U(e+d)
\end{equation*} 
so that $y,e,d$ have (positive) series expansions in $\{T,U,V\}$,
and thus also (positive) series expansions in $\{t_\bullet,t_\circ,t_{\o}\}$. 
Then the expressions of $T_i,U_i,V_i$ above
can in all cases be rewritten as rational expressions in terms of $d,e,y$, thereby well suited for series expansions.
For instance 
\begin{equation*}
T_{3i}=T\frac{(1-y^i)(1-\hat{\alpha}y^{i+1})}{(1-\hat{\alpha}y^i)(1-y^{i+1})},\ \mathrm{where}\ \hat{\alpha}=x\alpha=\frac{e+d+y}{1+e+d}\ . 
\end{equation*}

\appendix
 \section{Direct approaches for quadrangulations}
 \subsection{Using conserved quantities}
A direct expression for $B_i$ and $W_i$ can be obtained for quadrangulations by 
simply using the conserved quantities
\begin{equation}
\begin{split}
c_i&=B_{i+1}-\frac{1}{t_\circ}W_{i}B_{i+1}W_{i+2}\\
\tilde{c}_i&=W_{i+1}-\frac{1}{t_\bullet}B_{i}W_{i+i}B_{i+2}\\
\end{split}
\label{eq:cictildei}
\end{equation}
which don't depend on $i$ for all $i\geq 0$. If we look for $B_{i}$ and $W_{i}$ in the form
\begin{equation*}
\begin{split}
& B_{2i}=B\frac{u_{2i}\bar{u}_{2i+3}}{\bar{u}_{2i+1}u_{2i+2}}\qquad W_{2i+1}=W\frac{\bar{u}_{2i+1}u_{2i+4}}{u_{2i+2}\bar{u}_{2i+3}}\\
& B_{2i+1}=B\frac{\hat{u}_{2i+1}u_{2i+4}}{u_{2i+2}\hat{u}_{2i+3}}\qquad W_{2i}=W\frac{u_{2i}\hat{u}_{2i+3}}{\hat{u}_{2i+1}u_{2i+2}}\\
\end{split}
\end{equation*}
for some unknown functions $u_i$, $\bar{u}_i$ and $\hat{u}_i$, writing $c_{2i-1}=B-\frac{1}{t_\circ}W^2 B$ and 
$\tilde{c}_{2i}=W-\frac{1}{t_\bullet }B^2 W$ yields the two equations
\begin{equation*}
\begin{split}
&B(\bar{u}_{2i+3}u_{2i}-\bar{u}_{2i+1}u_{2i+2})-\frac{1}{t_\circ}W^2 B (\bar{u}_{2i-1}u_{2i+4}-\bar{u}_{2i+1}u_{2i+2})=0\\
&W(\bar{u}_{2i+1}u_{2i+4}-\bar{u}_{2i+3}u_{2i+2})-\frac{1}{t_\bullet}B^2 W (\bar{u}_{2i+5}u_{2i}-\bar{u}_{2i+3}u_{2i+2})=0\ .\\
\end{split}
\end{equation*}
Taking now
\begin{equation*}
u_{2i}=1-\lambda\, x^{2i}\qquad \bar{u}_{2i+1}=1-\mu\, x^{2i+1}\ ,
\end{equation*}
we see that the constant term (i.e.,\ the term independent of $\lambda$ and $\mu$) clearly vanishes in both equations
as well as the term proportional to $\lambda \mu$ since, in both equations, the sum of the indices is
the same (respectively $4i+3$ and $4i+5$) in all $\bar{u}\, u$ terms. As for the linear terms in $\lambda$ and $\mu$, their
vanishing implies
\begin{equation*}
\begin{split}
&B(\mu x^{3}+\lambda-\mu x -\lambda x^{2})-\frac{1}{t_\circ}W^2 B (\mu x^{-1}+\lambda x^{4}-\mu x-\lambda x^{2})=0\\
&W(\mu x+\lambda x^{4}-\mu x^{3}-\lambda x^{2})-\frac{1}{t_\bullet}B^2 W (\mu x^{5}+\lambda -\mu x^{3}-\lambda x^{2})=0\ .\\
\end{split}
\end{equation*}
Now, imposing $B_0=W_0$ leads us to choose $\lambda=1$. Eliminating $\mu$ from the above system yields then the equation
for $x$:
\begin{equation*}
B^2 W^2 \left(x^2+\frac{1}{x^2}+1\right)+t_\circ B^2+t_\bullet W^2 -t_\bullet t_\circ=0
\end{equation*}
while $\mu$ is obtained for instance via
\begin{equation*}
\mu=x\, \frac{t_\circ+W^2 x^2}{W^2+t_\circ x^2}\ .
\end{equation*}
After writing $t_\bullet=B(1-2W-B)$ and $t_\circ=W(1-2B-W)$, the equation for $x$ factors into
\begin{equation*}
\left(1-2(B+W) -\sqrt{B W}\left(x+\frac{1}{x}\right)\right)\left(1-2(B+W) +\sqrt{B W} \left(x+\frac{1}{x}\right)\right)=0
\end{equation*}
Choosing for instance to cancel the first factor (note that choosing to cancel the second factor amounts to change $x$ into $-x$, 
which in turn changes $\mu$ into $-\mu$ and leaves $u_{2i}$ and $\bar{u}_{2i+1}$ invariant), we recover precisely the characteristic equation \eqref{eq:charquad} while, after simplification
\begin{equation*}
\mu=\frac{c+x}{1+c\, x}\quad c=\sqrt{\frac{B}{W}}\ .
\end{equation*}
We end up with
\begin{equation*}
u_{2i}=1-x^{2i}\qquad \bar{u}_{2i+1}=1-\frac{c+x}{1+c\, x}x^{2i+1}\ ,
\end{equation*}
from which $B_{2i}$ and $W_{2i+1}$ follow. As for $B_{2i+1}$ and $W_{2i}$, they follow from a similar 
calculation (using now $c_{2i-1}$ and $\tilde{c}_{2i}$), leading to 
\begin{equation*}
\hat{u}_{2i+1}=1-\frac{1+c\, x}{c+ x}x^{2i+1}
\end{equation*}
 with the same $x$ and $c$.

\subsection{Using a guessing technique}
Let us now discuss another approach to guess the expressions of $B_i$ and $W_i$ for quadrangulations,
whose starting point is the perturbative method used in \cite{GEOD}. Recall that $B_i$ and $W_i$ are specified by
\begin{equation}\label{eq:Bi}
B_i=\tb+B_i(W_{i-1}+B_i+W_{i+1}),\ \ \ W_i=\tw+W_i(B_{i-1}+W_i+B_{i+1}).
\end{equation}
and that $B\equiv \lim_{i\to\infty}B_i$ and $W\equiv \lim_{i\to \infty} W_i$ are given by
\begin{equation}\label{eq:B}
B=\tb+B(2W+B),\ \ \ W=\tw+W(2B+W).\end{equation}

Write $B_i$ and $W_i$ as
$$
B_i=B(1-\sigma x^i+O(x^{2i})),\ \ \ W_i=W(1-\tau x^i+O(x^{2i})),
$$where $\sigma,\tau,x$ are series in $\tb,\tw$. Injecting into~(\ref{eq:Bi}) and extracting the terms of order $x^{i}$ (the terms of order $1$ cancel out because of~(\ref{eq:B})), we obtain the following system of two equations:
$$\left\{\begin{array}{l}
B\sigma=BW(\sigma+\tau/x)+2B^2\sigma+BW(\sigma+\tau x),\\
W\tau=BW(\tau+\sigma/x)+2W^2\tau+BW(\tau+\sigma x).
\end{array}
\right.
$$
Defining $c\equiv\tau/\sigma$, this simplifies (dividing the first line by $B\sigma$ and the second
line by $W\sigma$) as
\begin{equation}\label{eq:syst_xc}
\left\{
\begin{array}{l}
1=c\, W(x+1/x)+2(W+B),\\c=B(x+1/x)+2c(W+B).
\end{array}
\right.
\end{equation}
This system is linear in $B,W$, so that $B$ and $W$ are rational in terms of $c$ and $x$, we find
\begin{equation}\label{eq:rat_expBW}
B=\frac{xc^2}{2c^2x+cx^2+c+2x},\ \ W=\frac{x}{2c^2x+cx^2+c+2x}.
\end{equation}
Note also that $x+1/x=\frac{1-2B-2W}{cW}=c\frac{1-2B-2W}{B}$, so that $c^2=B/W$, and $x$
fits with \eqref{eq:charquad} .

We have $B_0=W_0=0$, and for $i\geq 2$, (\ref{eq:Bi}) gives
\begin{equation}\label{eq:rec}
W_i=1-\tb/B_{i-1}-W_{i-2}-B_{i-1},\ \ B_i=1-\tw/W_{i-1}-B_{i-2}-W_{i-1}.
\end{equation}
In addition it is known~\cite{Sc97} that $B_1$ and $W_1$ have rational expressions in terms of $B,W$ (using
the fact that $\tw B_1=\tb W_1$ is the series of black-rooted quadrangulations where $\tb$ marks black vertices and $\tw$ marks white vertices)
\footnote{The values of $B_1$ and $W_1$ may alternatively be obtained by use of the conserved quantities
$c_i$ and ${\tilde c}_i$ of eq.~\eqref{eq:cictildei} upon writing $B_1=c_1=\lim_{i\to \infty} c_i= B-W^2 B/t_\circ$ and $W_1=\tilde{c}_1=\lim_{i\to \infty} \tilde{c}_i= W-
B^2 W/t_\bullet$.}:
\begin{equation*}
B_1=\frac{B(1-2B-2W)}{1-2B-W},\ \ W_1=\frac{W(1-2B-2W)}{1-2W-B}.
\end{equation*}
Since $\tb=B(1-2W-B)$ and $\tw=W(1-2B-W)$ are also rational in $\{B,W\}$, we can compute iteratively
(using~(\ref{eq:rec})) rational expressions of $B_i$ and $W_i$, in terms of $\{B,W\}$, for any $i\geq 2$.
Hence we also obtain rational expressions of $\overline{B}_i\equiv B_i/B$ and $\overline{W}_i\equiv W_i/W$ in terms of $\{B,W\}$.
Now in these expressions we can substitute  $B$ and $W$ by their rational expressions in $\{x,c\}$ given by~(\ref{eq:rat_expBW}). Inspecting these rational expressions in $\{x,c\}$ (with the help of a computer algebra system) we easily recognize
\begin{equation*}
\hspace{-1.3cm}\overline{B}_{2i}=\frac{(1-x^{2i})(1+c\, x-(c+x)x^{2i+3})}{(1-x^{2i+2})(1+c\, x-(c+x)x^{2i+1})},\
\overline{W}_{2i}=\frac{(1-x^{2i})(c+x-(1+c\, x)x^{2i+3})}{(1-x^{2i+2})(c+x-(1+c\, x)x^{2i+1})},
\end{equation*}
\begin{equation*}
\hspace{-1.3cm}\overline{B}_{2i+1}=\frac{(1-x^{2i+4})(c+x-(1+c\, x)x^{2i+1})}{(1-x^{2i+2})(c+x-(1+c\, x)x^{2i+3})},\
\overline{W}_{2i+1}=\frac{(1-x^{2i+4})(1+c\, x-(c+x)x^{2i+1})}{(1-x^{2i+2})(1+c\, x-(c+x)x^{2i+3})},
\end{equation*}
so that we recover the now familiar expressions. 

The guessing technique presented here is quite robust (it can also be applied to guess the bivariate
expressions of Ambj\o rn and Budd \cite{AmBudd} for well-labelled trees counted according to the numbers of edges and local maxima,
and to guess the trivariate expressions of Section~\ref{sec:oherintegthreecolor}), but as we have seen, a crucial point is to have rational expressions in $\{B,W\}$ for $B_1$ and $W_1$, which in the present case are known in the literature
or can be obtained from conserved quantities.
It is actually possible to guess as well these rational expressions.
The first step is to compute the series expansion of $B_1$ and $W_1$ to a large order $k$
in $\{\tb,\tw\}$ (i.e., compute all the terms $\tb^r\tw^s$ of the series expansions such that
  $r+s\leq k$). To do that, we
can observe that, for $q>k$, $B_q$ (resp. $W_q$)
has the same series expansion to order $k$ as $B$ (resp. as $W$). So a possible algorithm is to compute the series expansions of $B$ and $W$ to order $k$, and then compute the series expansions to order $k$
of $B_1,W_1,\ldots,B_k,W_k$ from the closed system of equations
\begin{equation*}
\begin{array}{ll}
B_1=\tb+B_1(B_1+W_2),&W_1=\tw+W_1(W_1+B_2),\\
B_2=\tb+B_2(W_1+B_2+W_3),&W_2=\tw+W_2(B_1+W_2+B_3),\\
\ \ \ \ \ \vdots &\ \ \ \ \ \ \ \vdots\\
B_{k-1}=\tb+B_{k-1}(W_{k-2}+B_{k-1}+W_k),&W_{k-1}=\tw+W_{k-1}(B_{k-2}+W_{k-1}+B_k),\\
B_k=\tb+B_k(W_{k-1}+B_k+W),&W_k=\tw+W_k(B_{k-1}+W_k+B),
\end{array}
\end{equation*}
valid order by order in $t_\bullet$ and $t_\circ$ up to (total) order $k$. Having computed the expansions of the series $B_1$ and $W_1$ of
order $k$ in the variables $\{\tb,\tw\}$, we can
obtain, substituting $\tb$ by $B(1-B-2W)$ and $\tw$ by $W(1-W-2B)$,
 expansions of $B_1$ and $W_1$ of order $k$ in the variables $\{B,W\}$.
Out of these expansions
the function gfun:seriestoratpoly of Maple correctly guesses (for $k$ large enough)
the rational expressions
$B_1=B(1-2B-2W)/(1-2B-W)$ and $W_1=W(1-2B-2W)/(1-2W-B)$.

 \section{A derivation of eq.~\eqref{eq:dethizeroexp}}
As we mentioned, eq.~\eqref{eq:weylrep} is a standard result whose proof may be found in \cite{FH91}.
More precisely, it is a direct consequence of Proposition A.50.\  in Appendix A of \cite{FH91}.
Here we will generalize this proposition so as to prove our desired formula \eqref{eq:dethizeroexp}, by
following exactly the same sequence of arguments as in \cite{FH91}.

Our first ingredient is an extension of Lemma A.43.\ of \cite{FH91}, which we state as follows:
\begin{lem}\label{lem:A43}
For $p$ and $i$ non negative integers, let $(H_{m,n})_{0\leq m,n\leq p+i}$ and 
$(E_{m,n})_{0\leq m,n\leq p+i}$ be two lower triangular matrices of size $(p+i+1)\times (p+i+1)$,
with $1$'s along the diagonal and 
that are inverse of each other. Let $(c_\ell)_{\ell\geq 0}$ be a sequence of indeterminates. 
Define 
\begin{equation}
\begin{split}
&H^+_{m,n}=\left\{ 
\begin{matrix}
H_{m,n}+(c_0 H_{m,2p-1-n}+c_1 H_{m,2p-2-n}+\cdots + c_{p-n-1} H_{m,p}) & \hbox{if}\ n<p \\
H_{m,n} & \hbox{if}\ n\geq p \\
\end{matrix}
\right.\\
&E^-_{m,n}=\left\{ 
\begin{matrix}
E_{m,n} & \hbox{if}\ m< p \\
E_{m,n}-(c_0 E_{2p-1-m,n}+c_1 E_{2p-2-m,n}+\cdots + c_{2p-m-1} E_{0,n}) & \hbox{if}\ m \geq p\\ 
\end{matrix}
\right.\\
\end{split}
\label{eq:HEplusminus}
\end{equation}
Then 
 $(H^+_{m,n})_{0\leq m,n\leq p+i}$ and 
$(E^-_{m,n})_{0\leq m,n\leq p+i}$ are lower triangular, with $1$'s on the diagonal, and are inverse of each other.
\end{lem}
\begin{proof}
That $(H^+_{m,n})_{0\leq m,n\leq p+i}$ is lower triangular with $1$'s along the diagonal is clear 
since the indices $\ell$ of the sequence of terms $H_{m,\ell}$ added to $H_{m,n}$ when $n<p$ are all larger
than or equal to $p$ hence strictly larger than $n$, so they are $0$ when $n\geq m$.
That $(E^-_{m,n})_{0\leq m,n\leq p+i}$ is lower triangular  with $1$'s along the diagonal is clear 
since the indices $\ell$ of the sequence of terms $E_{\ell,n}$ subtracted from $E_{m,n}$ when $m\geq p$ are all lower than
or equal to $2p-1-m$ hence strictly lower than $p$, and thus strictly lower than $m$, so they are $0$ when $m\leq n$.
Now 
\begin{equation*}
\begin{split}
\hspace{-1.3cm}(H^+\cdot E^-)_{m,n}&=\sum_{k=0}^{p+i}H^+_{m,k}E^-_{k,n}
=\sum_{k=0}^{p-1} \left(H_{m,k}+\left(\cdots\right)\right) E_{k,n}+\sum_{k=p}^{p+i} H_{m,k}
\left( E_{k,n}-\left(\cdots\right)\right)\\
& = \sum_{k=0}^{p+i} H_{m,k}E_{k,n}+\sum_q c_q \left(\sum_{k=0}^{p-1-q} H_{m,2p-1-q-k}E_{k,n}-
\sum_{k=p}^{p+i}H_{m,k} E_{2p-1-q-k,n}\right)\\
& = \delta_{m,n}+\sum_q c_q \left(\sum_{r,s,\ r\geq p \atop
r+s=2p-1-q} H_{m,r}E_{s,n}-
\sum_{r',s',\ r'\geq p \atop
r'+s'=2p-1-q}H_{m,r'} E_{s',n}\right)\\
& = \delta_{m,n} \\
\end{split}
\end{equation*}
so that  $(H^+_{m,n})_{0\leq m,n\leq p+i}$ and 
$(E^-_{m,n})_{0\leq m,n\leq p+i}$ are inverse of each other.
\end{proof}
We shall use the above lemma in the particular case:
\begin{equation}
\begin{split}
&H_{m,n}=h_{m-n}\left(x_1,x_2,\cdots,x_p,\frac{1}{x_1},\frac{1}{x_2},\cdots,\frac{1}{x_p}\right)\\
&E_{m,n}=(-1)^{m-n}e_{m-n}\left(x_1,x_2,\cdots,x_p,\frac{1}{x_1},\frac{1}{x_2},\cdots,\frac{1}{x_p}\right)\\
\end{split}
\label{eq:HEchoice}
\end{equation}
where the $x_a$'s are the solutions of the characteristic equation \eqref{eq:chareq} and $h_\ell(\cdots)$ and $e_\ell(\cdots)$ denote 
respectively the homogeneous and elementary symmetric polynomials of degree $\ell$
in their $2p$ variables. The matrices 
$(H_{m,n})_{0\leq m,n\leq p+i}$ and 
$(E_{m,n})_{0\leq m,n\leq p+i}$ are clearly lower triangular, with $1$'s on the diagonal, and it is a standard property of symmetric 
polynomials that they are indeed inverse of each other (this is easily shown by looking for instance at the generating functions
of the $h_\ell$'s and of the $e_\ell$'s). 
If we now choose
\begin{equation}
c_0=c\ , \qquad c_q=-(c^2-1)(-c)^{q-1}\ \hbox{for}\ q\geq 1
\label{eq:cchoice}
\end{equation}
with $c=b/w=\sqrt{B/W}$, then, from eqs.~\eqref{eq:Ctoe} and \eqref{eq:Csum}, we deduce
\begin{equation*}
h_i^{(0)}=(B W)^{\frac{i(i+1)}{2}}C_p^{i+1} \det_{0\leq m,n \leq i} E^-_{p+m,n}
\end{equation*}
where we have used $e_{\ell}\left(x_1,x_2,\cdots,x_p,\frac{1}{x_1},\frac{1}{x_2},\cdots,\frac{1}{x_p}\right)=
e_{2p-\ell}\left(x_1,x_2,\cdots,x_p,\frac{1}{x_1},\frac{1}{x_2},\cdots,\frac{1}{x_p}\right)$ for $\ell=0,\cdots p$.

\begin{figure}
\begin{center}
\includegraphics[width=7cm]{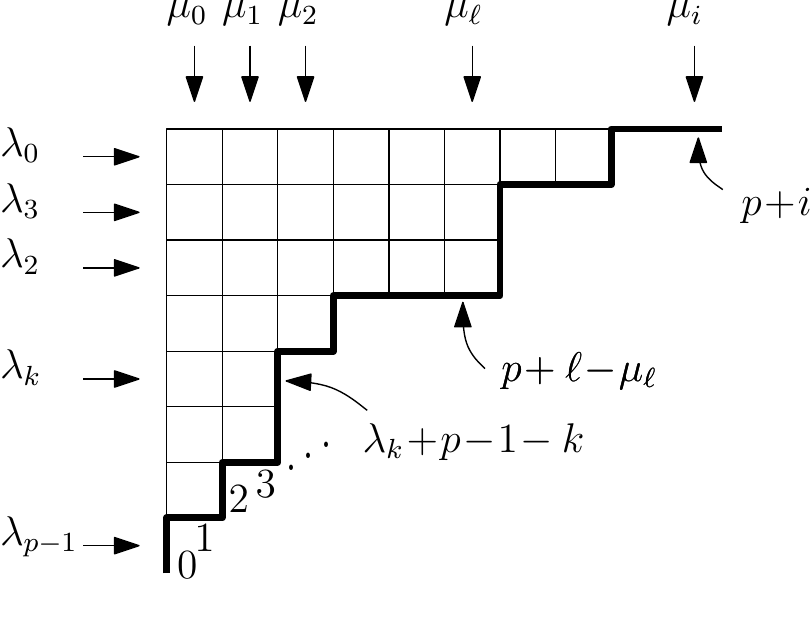}
\end{center}
\caption{A pictorial explanation of why the choice \eqref{eq:permutation} is a permutation of $(0,\cdots,p+i)$.}
\label{fig:permutation}
\end{figure}

Let us now recall Lemma A.42.\ of \cite{FH91} (see \cite{FH91} for a proof).
\begin{lem}\label{lem:A42}
Let $H^+$ and $E^-$ be $(p+i+1)\times (p+i+1)$ matrices which are inverse of each other. Let $(s_0,s_1,\cdots ,s_{p-1},s_0',s_1',\cdots ,s_i')$
and $(t_0,t_1,\cdots ,t_{p-1},t_0',t_1',\cdots ,t_i')$ be permutations of the sequence $(0,\cdots,p+i)$. Then
\begin{equation*}
\det(H^+_{s_m,t_n})_{0\leq m,n\leq p}= \epsilon \det(H^+)\ \det(E^-_{t_m',s_n'})_{0\leq m,n\leq i}
\end{equation*}
where $\epsilon$ is the product of the signs of the two permutations.
\end{lem}
We will use this lemma in the following particular case: let $\lambda_0\geq \lambda_1 \geq \cdots \geq  \lambda_{p-1}\geq 0$ 
be a partition of $\sum_{k=0}^{p-1} \lambda_k$ and $\mu_0\geq \mu_1 \geq \cdots \geq \mu_{i}\geq 0$ the conjugate 
partition (with in particular $\sum_{k=0}^{p-1} \lambda_k=\sum_{k=0}^{i} \mu_k$). Then it is a classical result that
$(s_0,s_1,\cdots ,s_{p-1},s_0',s_1',\cdots ,s_i')$ with
\begin{equation}
\begin{split}
&s_k= \lambda_k+p-1-k, \quad k=0,\cdots ,p-1 \\
&s_k'=p+k-\mu_k, \quad k=0,\cdots ,i\\
\end{split}
\label{eq:permutation}
\end{equation}
forms a permutation of the sequence $(0,\cdots,p+i)$. It simply corresponds to labelling each of the $p$
lines (resp. the $i+1$ columns) of the associated Young diagram by the index of its last vertical (resp.
horizontal) segment, upon indexing this sequence of segments (which forms a broken line of length $p+i+1$) 
by $0,\cdots ,p+i$ in the natural way (see Figure \ref{fig:permutation}).
For the other permutation $(t_0,t_1,\cdots ,t_{p-1},t_0',t_1',\cdots ,t_i')$, we shall use
\begin{equation*}
\begin{split}
&t_k= p-1-k, \quad k=0,\cdots ,p-1 \\
&t_k'=p+k, \quad k=0,\cdots ,i\\
\end{split}
\end{equation*}
in which case $\epsilon=(-1)^{\sum_{k=0}^{p-1}\lambda_k}=(-1)^{\sum_{k=0}^{i}\mu_k}$ (the permutation
$(t_0,t_1,\cdots , t_{p-1},t_0',t_1',\cdots ,t_i')$ corresponds indeed in Figure \ref{fig:permutation} to the case of a Young diagram 
without boxes, so that the broken line sticks to the left and upper sides. The permutation $(s_0,s_1,\cdots ,s_{p-1},s_0',s_1',\cdots ,s_i')$
may be viewed as obtained from $(t_0,t_1,\cdots ,t_{p-1},t_0',t_1',\cdots ,t_i')$
by successive additions of boxes to the Young diagram. Adding a box corresponds to performing a transposition so, each 
time a box is added, the sign of the permutation changes
and $\epsilon$ is therefore nothing but $-1$ to the power the total number of boxes). 
In our case $\det(H^+)=1$ and
therefore
\begin{equation*}
\det(E^-_{p+m,p+n-\mu_n})_{0\leq m,n\leq i}=(-1)^{\sum_{k=0}^{p-1}\lambda_k}
\det(H^+_{\lambda_m+p-1-m,p-1-n})_{0\leq m,n\leq p}\ .
\end{equation*}

In other words, if we define, for $r\in\mathbb{Z}$, $L(r)$ as the line-vector of length $p$
whose $n$-th entry (for $n\in 0,\cdots, p-1$) is $h_{r+n+1}+c\, h_{r-n}-\sum_{q=1}^{n}(c^2-1)(-c)^{q-1}h_{r-n+q}$,
and define, for any $(r_0,\cdots,r_{p-1})\in\mathbb{Z}^p$, $M(r_0,\cdots,r_{p-1})$ as the $p\times p$ matrix
whose $m$-th row (for $m\in 0,\cdots, p-1$) is $L(r_m)$, then the equality above rewrites as
\begin{equation*}
\det(E^-_{p+m,p+n-\mu_n})_{0\leq m,n\leq i}=(-1)^{\sum_{k=0}^{p-1}\lambda_k}\det(M(r_0,\cdots,r_{p-1})),
\end{equation*}
with $r_m=\lambda_m-m-1$.
In particular, if we choose $\lambda_k=(i+1)$ for all $k=0,\cdots, p-1$, so that $\mu_k=p$ for all 
$k=0,\cdots, i$, we deduce
\begin{equation*}
h_i^{(0)}=(B W)^{\frac{i(i+1)}{2}}(-1)^{p(i+1)} C_p^{i+1} \det(M(i,i-1,\cdots,i-p+1))
\end{equation*}
involving now the determinant of a matrix of \emph{fixed size} $p\times p$, instead of a possibly arbitrarily large
size $(i+1)\times (i+1)$.

To prove eq.~\eqref{eq:dethizeroexp}, it remains to give an explicit form of this latter determinant in terms of the $x_a$'s and $c$.
Let us define
\begin{equation*}
\zeta_a(r)\equiv x_a^r-x_a^{-r}
\end{equation*}
for $a=1,\cdots p$ and any integer $r\geq 0$.
Then, Lemma A.54.\ of \cite{FH91} states that
\begin{lem}
\begin{equation*}
\zeta_a(r)=\left(h_{r-p}, h_{r-p+1}+h_{r-p-1}, \cdots, h_{r-1}+h_{r-2p+1}\right)\cdot \bar{E}\cdot \left(
\begin{matrix}
\zeta_a(p)\\ \vdots \\ \zeta_a(1) \\
\end{matrix}
\right)
\end{equation*}
with $h_{\ell}=h_{\ell}\left(x_1,x_2,\cdots,x_p,\frac{1}{x_1},\frac{1}{x_2},\cdots,\frac{1}{x_p}\right)$ and $\bar{E}=(E_{m,n})_{0\leq m,n\leq p-1}$
the $p\times p$ matrix with matrix elements $E_{m,n}$ as in eq.~\eqref{eq:HEchoice}.
\end{lem}
A direct consequence of this lemma is that, if we now define, for $r\geq 0$
\begin{equation*}
\xi_a(r)\equiv c\, \zeta_a(r)+\zeta_a(r+1)\ ,
\end{equation*}
we have
\begin{equation*}
\begin{split}
&\xi_a(r)=\left(\eta_{r-p}, \eta_{r-p+1}+\eta_{r-p-1}, \cdots, \eta_{r-1}+\eta_{r-2p+1}\right)\cdot \bar{E}\cdot \left(
\begin{matrix}
\zeta_a(p)\\ \vdots \\ \zeta_a(1) \\
\end{matrix}\right)
\\
&\hbox{where}\ \eta_r=c \, h_r+h_{r+1}\ . \\
\end{split}
\end{equation*}
In other words, if we define for $r\in \mathbb{Z}$, $\widehat{L}(r)$ as the line-vector
$(\eta_r,\eta_{r+1}+\eta_{r-1},\cdots,\eta_{r+p-1}+\eta_{r-p+1})$, then for any $r\geq -p$,
\begin{equation*}
\xi_a(p+r)=\widehat{L}(r)\cdot \overline{E}\cdot
\left(
\begin{matrix}
\zeta_a(p)\\ \vdots \\ \zeta_a(1) \\
\end{matrix}
\right)\ .
\end{equation*}
Hence, if we define, for any integers $r_0,\cdots,r_{p-1} $, $\widehat{M}(r_0,\cdots,r_{p-1})$ as the $p\times p$ matrix whose $m$-th row is $\widehat{L}(r_m)$, then, since $\det(\bar{E})=1$, we have (for $r_0,\cdots,r_{p-1}$ all at least $-p$)
\begin{equation*}
\det(\widehat{M}(r_0,\cdots,r_{p-1}))=
\frac{\det\limits_{1\leq a',a \leq p} \xi_a(p+r_{a'-1})}{ \det\limits_{1\leq a',a \leq p} \zeta_a (p+1-a')}
\end{equation*}
\begin{figure}
\begin{center}
\includegraphics[width=14cm]{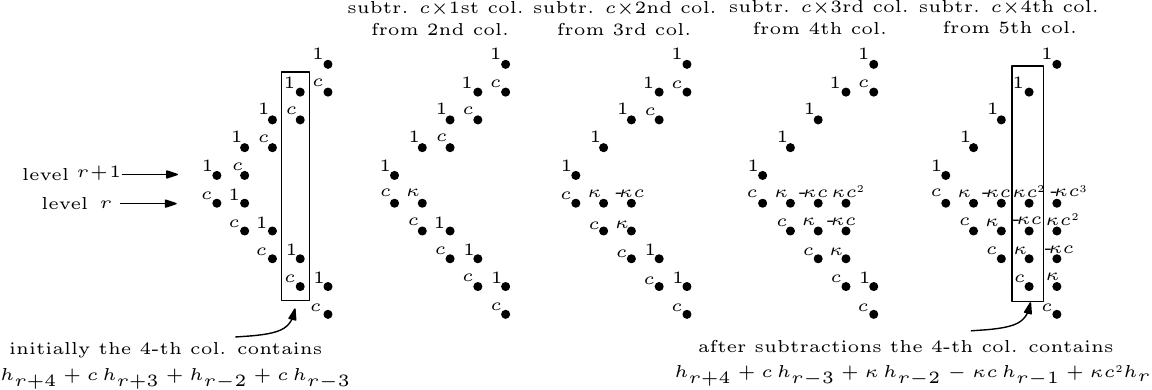}
\end{center}
\caption{A pictorial representation of the successive (from left to right) actions of $\tau_1,\tau_2,\cdots$ on $L(r)$.
The $(n+1)$-th column represents the set of indices $\ell$ of the various $h_\ell$ entering the linear combination 
of the $n$-th entry of $L(r),\tau_1 L(r),\tau_2\circ\tau_1L(r),\cdots$ (from left to right). Each appearing index is represented
by a dot at the corresponding height level beside which we indicate the corresponding coefficient. Here,
we use the short hand notation $\kappa=1-c^2$. For instance, we indicated the entry of the $4$th column (corresponding to $n=3$)
before and after subtractions (note that the $n$-th entry does not change after the $n$-th subtraction).}
\label{fig:columnoperations}
\end{figure}
Now, looking at Figure \ref{fig:columnoperations}, we have for any $r\in\mathbb{Z}$
\begin{equation*}
L(r)=\tau_{p-1}\circ\tau_{p-2}\circ\cdots\circ\tau_1(\widehat{L}(r)),
\end{equation*}
where for $n \in 1,\cdots, p-1$, $\tau_n$ is the operator (on $p$-line-vectors) that subtracts $c$ times the $(n-1)$-th entry from the $n$-th entry.
Hence, 
\begin{equation*}
M(r_0,\cdots,r_{p-1})=\tau_{p-1}\circ\tau_{p-2}\circ\cdots\circ\tau_1(\widehat{M}(r_0,\cdots,r_{p-1})),
\end{equation*}
where $\tau_n$ is now the operator on $p\times p$ matrices that subtracts $c$ times the $(n-1)$-th column from the $n$-th column. 
Since these column operations do not change the determinant, we have
\begin{equation*}
\det(M(r_0,\cdots,r_{p-1}))=\det(\widehat{M}(r_0,\cdots,r_{p-1}))
\end{equation*}
and in particular,
\begin{equation*}
\begin{split}
\det(M(i,i-1,\cdots,i-p+1))&=\frac{\det\limits_{1\leq a',a \leq p} \xi_a(i+(p+1-a'))}{\det\limits_{1\leq a',a \leq p} \zeta_a(p+1-a')}\\
&=\frac{\det\limits_{1\leq a',a \leq p} \xi_a(i+a')}{ \det\limits_{1\leq a',a \leq p} \zeta_a(a')}\\
&= \frac{\det\limits_{1\leq a',a\leq p} ((x_a+c) x_a^{i+a'}- (1/x_a+c) x_a^{-(i+a')})}
{\det\limits_{1\leq a',a\leq p} (x_a^{a'}-x_a^{-a'})})\\
&=\prod_{a=1}^p (1+ c\, x_a) \frac{\det\limits_{1\leq a',a\leq p} (\gamma_a x_a^{i+a'}-x_a^{-(i+1+a')})}
{\det\limits_{1\leq a',a\leq p} (x_a^{a'}-x_a^{-a'})})\\
&\qquad \hbox{where}\ \gamma_a=\frac{c+x_a}{1+c x_a}\\
\end{split}
\end{equation*}
from which (up to simple transposition) eq.~\eqref{eq:dethizeroexp} follows.

\section{Generating functions for hard dimers}
\begin{figure}
\begin{center}
\includegraphics[width=14cm]{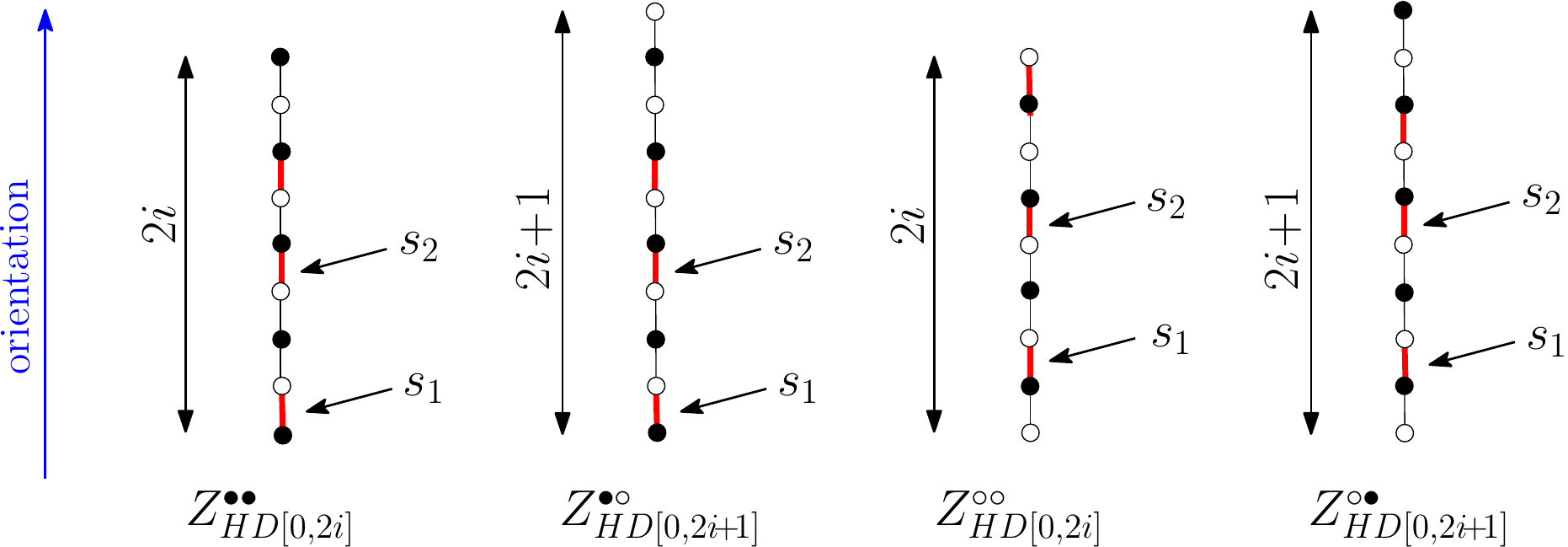}
\end{center}
\caption{Example of configurations of hard dimers contributing to $Z_{HD[0,2i]}^{\bullet\bullet}$, 
$Z_{HD[0,2i+1]}^{\bullet\circ}$, $Z_{HD[0,2i]}^{\circ\circ}$ and $Z_{HD[0,2i+1]}^{\circ\bullet}$ respectively.
The bicolored segments are viewed as oriented from bottom to top. A weight 
$s_1$ is assigned to each dimer (in red) lying on a link oriented (from bottom to top) from a black to a white node
 and a weight $s_2$
is assigned to each dimer lying on a link oriented from a white to a black node.}
\label{fig:harddimers}
\end{figure}
In this section, we shall derive a number of generating functions for hard dimers on bicolored segments.
We denote by $Z_{HD[0,2i]}^{\bullet\bullet}\equiv Z_{HD[0,2i]}^{\bullet\bullet}(s_1,s_2)$ the generating
function of hard dimers on an oriented segment (a finite oriented linear graph) made of $2i$ links whose 
nodes are bicolored alternatively in black and white and whose first and last node are black.  Each link
of the segment may be occupied by a dimer or not, with the constraint that \emph{a node is incident 
to at most one dimer}. A weight 
$s_1$ is assigned to each dimer lying on a link oriented from a black to a white node and a weight $s_2$
is assigned to each dimer lying on a link oriented from a white to a black node, see Figure \ref{fig:harddimers}. We also introduce
the generating functions $Z_{HD[0,2i]}^{\circ\circ}$, $Z_{HD[0,2i+1]}^{\bullet\circ}$ and $Z_{HD[0,2i+1]}^{\circ\bullet}$
with obvious definitions.

We introduce the parametrization
\begin{equation*}
s_1=- \frac{x}{(c+x)(1+c\, x)}\qquad s_2=-\frac{c^2x}{(c+x)(1+c\, x)}\ .
\end{equation*}
It is achieved by taking for instance
\begin{equation*}
c=\sqrt{\frac{s_2}{s_1}}\qquad x+\frac{1}{x}=-\frac{1+s_1+s_2}{c\, s_1}\ .
\end{equation*}
Note that exchanging $s_1$ and $s_2$ simply amounts to change $c$ into $1/c$, keeping $x$ unchanged.
Note also that $x$ is only defined up to $x\leftrightarrow 1/x$ and that,  for $(c,x)$ defined as above, our parametrization 
for $s_1$ and $s_2$ would have been realized as well by choosing $(-c,-x)$ instead. The reader is invited to verify that
all the final formulas below for our hard dimer generating functions are in practice invariant under
$x\leftrightarrow 1/x$ and under $(c,x)\leftrightarrow (-c,-x)$.

Upon decomposing the generating function according to the nature (i.e.,\ occupied by a dimer or not) of the
last link, we may write
\begin{equation*}
\begin{split}
&Z_{HD[0,2i]}^{\bullet\bullet}=Z_{HD[0,2i-1]}^{\bullet\circ}+s_2\, Z_{HD[0,2i-2]}^{\bullet\bullet}\\
&Z_{HD[0,2i+1]}^{\bullet\circ}=Z_{HD[0,2i]}^{\bullet\bullet}+s_1\, Z_{HD[0,2i-1]}^{\bullet\circ}\\
\end{split}
\end{equation*}
From the first equation for $i$ and $i+1$, we get 
\begin{equation}
\begin{split}
&Z_{HD[0,2i-1]}^{\bullet\circ}=Z_{HD[0,2i]}^{\bullet\bullet}-s_2\, Z_{HD[0,2i-2]}^{\bullet\bullet}\\
&Z_{HD[0,2i+1]}^{\bullet\circ}=Z_{HD[0,2i+2]}^{\bullet\bullet}-s_2\, Z_{HD[0,2i]}^{\bullet\bullet}\\
\label{eq:oddcase}
\end{split}
\end{equation}
and plugging these values in the second equation yields
\begin{equation*}
Z_{HD[0,2i+2]}^{\bullet\bullet}=(1+s_1+s_2) Z_{HD[0,2i]}^{\bullet\bullet}-s_1\, s_2 \, Z_{HD[0,2i-2]}^{\bullet\bullet}
\end{equation*}
for all $i\geq 0$, with initial conditions $Z_{HD[0,0]}^{\bullet\bullet}=1$ and $Z_{HD[0,-2]}^{\bullet\bullet}=0$ 
(which implies $Z_{HD[0,2]}^{\bullet\bullet}=(1+s_1+s_2)$ as wanted).
Setting 
\begin{equation*}
\tau_i=(-c\, s_1)^{-i}Z_{HD[0,2i]}^{\bullet\bullet}
\end{equation*}
and using the above parametrization of $s_1$ and $s_2$ (note in particular that $(-c\, s_1)^2=s_1s_2$), the equation reads
\begin{equation*}
\tau_{i+1}=\left( x+\frac{1}{x}\right) \tau_i-\tau_{i-1}
\end{equation*}
for $i\geq 0$ with $\tau_0=1$ and $\tau_{-1}=0$. Its solution is well known to be 
\begin{equation*}
\tau_i=\frac{x^{i+1}-x^{-(i+1)}}{x-x^{-1}}
\end{equation*}
which gives eventually
\begin{equation*}
Z_{HD[0,2i]}^{\bullet\bullet}=\left(\frac{c}{(c+x)(1+c\, x)}\right)^i \frac{1-x^{2i+2}}{1-x^2}\ .
\end{equation*}
As for $Z_{HD[0,2i+1]}^{\bullet\circ}$, it is obtained from the second line of eq.~\eqref{eq:oddcase}
which after simplification gives
\begin{equation*}
Z_{HD[0,2i+1]}^{\bullet\circ}=(1+c\, x)\left(\frac{c}{(c+x)(1+c\, x)}\right)^{i+1} \frac{1-\frac{c+x}{1+c\, x} x^{2i+3}}{1-x^2}\ .
\end{equation*}
Finally, $Z_{HD[0,2i]}^{\circ\circ}$ and $Z_{HD[0,2i+1]}^{\circ\bullet}$ are obtained by changing $c$ into $1/c$, 
namely
\begin{equation*}
\begin{split}
&Z_{HD[0,2i]}^{\circ\circ}=\left(\frac{c}{(c+x)(1+c\, x)}\right)^i \frac{1-x^{2i+2}}{1-x^2}\\
&Z_{HD[0,2i+1]}^{\circ\bullet}=\left(1+\frac{x}{c}\right)\left(\frac{c}{(c+x)(1+c\, x)}\right)^{i+1} \frac{1-\frac{1+c\, x}{c+ x} x^{2i+3}}{1-x^2}\ .
\end{split}
\end{equation*}
Note that the fact that $Z_{HD[0,2i]}^{\bullet\bullet}=Z_{HD[0,2i]}^{\circ\circ}$ is obvious by reversing the orientation
of the segment and exchanging the colors. For $c=1$ ($s_1=s_2$) all above formulas match well-known
expressions for dimer generating functions on uncolored segments.

\section*{Acknowledgements} We thank J. Bouttier for very useful discussions. The work of
\'EF was partly supported by the ANR grant  
``Cartaplus'' 12-JS02-001-01 and the ANR grant ``EGOS'' 12-JS02-002-01.

\bibliographystyle{plain}
\bibliography{gen2pbicolored}

\end{document}